\documentclass[12pt]{amsart}
\usepackage{amscd,amssymb,amsthm,amsmath,amssymb,textcomp,supertabular,wasysym,rotating,longtable,enumerate,rotating,mathrsfs,mathtools,multirow}
\usepackage[matrix,arrow,curve]{xy}
\usepackage{pdflscape}

\newtheorem{theorem}[equation]{Theorem}

\newtheorem{proposition}[equation]{Proposition}
\newtheorem{lemma}[equation]{Lemma}
\newtheorem{corollary}[equation]{Corollary}

\newtheorem*{maintheorem*}{Main Theorem}

\theoremstyle{definition}
\newtheorem{example}[equation]{Example}

\theoremstyle{remark}
\newtheorem{remark}[equation]{Remark}

\makeatletter\@addtoreset{equation}{section} \makeatother

\newcommand{\mumu}{\boldsymbol{\mu}}

\author{Ivan Cheltsov and Arman Sarikyan}

\title{Equivariant pliability of the projective space}

\address{\emph{Ivan Cheltsov}
\newline
\textnormal{The University of Edinburgh,  Edinburgh, Scotland}
\newline
\textnormal{\texttt{I.Cheltsov@ed.ac.uk}}}

\address{\emph{Arman Sarikyan}
\newline
\textnormal{The University of Edinburgh,  Edinburgh, Scotland}
\newline
\textnormal{\texttt{ar.sarikyan@gmail.com}}}

\begin{document}

\begin{abstract}
We classify finite subgroups $G\subset\mathrm{PGL}_4(\mathbb{C})$ such that $\mathbb{P}^3$ is not $G$-birational to conic bundles and del Pezzo fibrations,
and explicitly describe all $G$-Mori fibre spaces that are $G$-birational to $\mathbb{P}^3$ for these subgroups.
\end{abstract}

\maketitle

\tableofcontents

Throughout this paper, all varieties are assumed to be projective and defined over~$\mathbb{C}$.

\section{Introduction}
\label{section:Intro}

Finite subgroups in $\mathrm{PGL}_4(\mathbb{C})$ have been classified by Blichfeldt \cite{Blichfeldt1917},
who has split them into the~following four classes:  {intransitive} groups, {transitive} groups, {imprimitive} groups,
{primitive} groups. In geometric language, these classes can be described as follows:
\begin{enumerate}[(I)]
\item {intransitive} groups are group that fix a point or leave a line invariant,
\item {transitive} groups are groups that are not intransitive,
\item {imprimitive} groups are transitive groups that
\begin{itemize}
\item[--] either leave a union of two skew lines invariant,
\item[--] or have an orbit of length $4$ (monomial subgroups),
\end{itemize}
\item {primitive} groups are transitive groups that are not imprimitive.
\end{enumerate}
Note that $\mathrm{PGL}_4(\mathbb{C})$ contains finitely many primitive finite subgroups up to conjugation.

Now, let us fix a finite subgroup $G\subset\mathrm{PGL}_4(\mathbb{C})$.
The main aim of this paper is to study \mbox{$G$-birational} transformations of $\mathbb{P}^3$ into other $G$-Mori fibre spaces.
If $\mathbb{P}^3$ is not $G$-birational to any other $G$-Mori fibre space, then $\mathbb{P}^3$ is said to be $G$-\emph{birationally rigid}.
It has been proven in~\cite{ChSh10a,ChSh09b,CheltsovShramov2017} that
\begin{center}
$\mathbb{P}^3$ is $G$-birationally rigid $\iff$ $G$ is primitive, $G\not\cong\mathfrak{A}_5$ and $G\not\cong\mathfrak{S}_5$.
\end{center}
For instance, if $G$ is an imprimitive subgroup such that $\mathbb{P}^3$ contains a $G$-orbit of length~$4$,
then $\mathbb{P}^3$ is not $G$-birationally rigid. In fact, this follows from

\begin{example}[{\cite{CheltsovShramov2017,Martello}}]
\label{example:Fano-Enriques}
Suppose that $G$ is imprimitive, $\mathbb{P}^3$ does not contain $G$-invariant unions of two skew lines,
and $\mathbb{P}^3$ contains a $G$-orbit $\Sigma_4$ of length $4$.
Let~$\mathcal{M}$ be the~linear system that consists of sextic surfaces in $\mathbb{P}^3$
singular along each line passing through two points in  $\Sigma_4$.
Then $\mathcal{M}$ defines a $G$-rational map $\psi\colon\mathbb{P}^3\dasharrow\mathbb{P}^{13}$.
Let $X_{24}=\overline{\mathrm{im}(\psi)}$.
Then
\begin{enumerate}[(i)]
\item the~induced map $\mathbb{P}^3\dasharrow X_{24}$ is $G$-birational,
\item $X_{24}\cong\mathbb{P}^1\times\mathbb{P}^1\times\mathbb{P}^1/\langle\tau\rangle$ for an involution  $\tau$ that fixes $8$ points \cite[$\S$ 6.3.2]{Bayle},
\item the~Fano threefold $X_{24}$ is a $G$-Mori fiber space over a point.
\end{enumerate}
\end{example}

Following \cite{AhmadinezhadOkada},
we define $\mathbb{P}^3$ to be \emph{$G$-solid} if $\mathbb{P}^3$ is not $G$-birational to conic bundles and del Pezzo fibrations.
In this case, all $G$-Mori fibre spaces that are $G$-birational to $\mathbb{P}^3$ are terminal Fano threefolds ---
they form a set $\mathcal{P}_G(\mathbb{P}^3)$, which we call the \emph{$G$-pliability} \cite{CortiMella}.
For example, if $\mathbb{P}^3$ is $G$-solid, then $\mathcal{P}_G(\mathbb{P}^3)=\{\mathbb{P}^3\}$ $\iff$ $\mathbb{P}^3$ is $G$-birationally rigid.

It natural to ask when is $\mathbb{P}^3$ $G$-solid? If $\mathbb{P}^3$ is $G$-solid, it follows from \cite{CheltsovShramov,CheltsovShramov2017} that
\begin{enumerate}
\item the~subgroup $G$ is transitive,
\item $\mathbb{P}^3$ does not contain $G$-invariant unions of two skew lines,
\item neither $G\cong\mathfrak{A}_5$ nor $G\cong\mathfrak{S}_5$.
\end{enumerate}
In fact, these conditions guarantee that $\mathbb{P}^3$ is $G$-solid provided that $|G|$ is sufficiently large.
Namely, if $G$ is transitive, $\mathbb{P}^3$ has no $G$-invariant unions of two skew lines,~and~\mbox{$|G|\geqslant 2^{17}3^4$},
then it follows from \cite{CheltsovShramov2017,CDK} that $\mathbb{P}^3$ is \mbox{$G$-solid}, $\mathbb{P}^3$ contains a unique $G$-orbit of length $4$,
and $\mathcal{P}_G(\mathbb{P}^3)=\{\mathbb{P}^3,X_{24}\}$,
where $X_{24}$ is the~Fano threefold from Example~\ref{example:Fano-Enriques}.

The goal of this paper is to prove the following result:

\begin{maintheorem*}
Let $G$ be an imprimitive finite subgroup in $\mathrm{PGL}_4(\mathbb{C})$ such that
$\mathbb{P}^3$ does not have $G$-invariant unions of two skew lines,
and $G$ is not conjugated to
\begin{itemize}
\item the~subgroup $G_{48,3}\cong\mumu^2_2.\mathfrak{A}_4\cong\mumu_4^2\rtimes\mumu_3$ of order $48$ generated by
$$
\begin{pmatrix}
-1 & 0 & 0 & 0\\
0 & 1 & 0 & 0\\
0 & 0 & -1 & 0\\
0 & 0 & 0 & 1
\end{pmatrix},
\begin{pmatrix}
1 & 0 & 0 & 0\\
0 & -1 & 0 & 0\\
0 & 0 & -1 & 0\\
0 & 0 & 0 & 1
\end{pmatrix},
\begin{pmatrix}
0 & 0 & 1 & 0\\
1 & 0 & 0 & 0\\
0 & 1 & 0 & 0\\
0 & 0 & 0 & 1
\end{pmatrix},
\begin{pmatrix}
0 & i & 0 & 0\\
1 & 0 & 0 & 0\\
0 & 0 & 0 & -i\\
0 & 0 & 1 & 0
\end{pmatrix};
$$

\item the~subgroup $G_{96,72}\cong\mumu_2^3.\mathfrak{A}_4\cong\mumu_4^2\rtimes\mumu_6$ of order $96$ generated by
$$
\begin{pmatrix}
-1 & 0 & 0 & 0\\
0 & 1 & 0 & 0\\
0 & 0 & 1 & 0\\
0 & 0 & 0 & 1
\end{pmatrix},
\begin{pmatrix}
1 & 0 & 0 & 0\\
0 & -1 & 0 & 0\\
0 & 0 & 1 & 0\\
0 & 0 & 0 & 1
\end{pmatrix},
\begin{pmatrix}
1 & 0 & 0 & 0\\
0 & 1 & 0 & 0\\
0 & 0 & -1 & 0\\
0 & 0 & 0 & 1
\end{pmatrix},
\begin{pmatrix}
0 & 0 & 1 & 0\\
1 & 0 & 0 & 0\\
0 & 1 & 0 & 0\\
0 & 0 & 0 & 1
\end{pmatrix},
\begin{pmatrix}
0 & i & 0 & 0\\
1 & 0 & 0 & 0\\
0 & 0 & 0 & i\\
0 & 0 & 1 & 0
\end{pmatrix};
$$

\item the~subgroup $G_{324,160}^\prime\cong\mumu_3^3\rtimes\mathfrak{A}_4$ of order $324$ generated by
$$
\begin{pmatrix}
e^{\frac{2\pi i}{3}} & 0 & 0 & 0\\
0 & 1 & 0 & 0\\
0 & 0 & 1 & 0\\
0 & 0 & 0 & 1
\end{pmatrix},
\begin{pmatrix}
1 & 0 & 0 & 0\\
0 & e^{\frac{2\pi i}{3}} & 0 & 0\\
0 & 0 & 1 & 0\\
0 & 0 & 0 & 1
\end{pmatrix},
\begin{pmatrix}
1 & 0 & 0 & 0\\
0 & 1 & 0 & 0\\
0 & 0 & e^{\frac{2\pi i}{3}} & 0\\
0 & 0 & 0 & 1
\end{pmatrix},
\begin{pmatrix}
0 & 0 & 1 & 0\\
1 & 0 & 0 & 0\\
0 & 1 & 0 & 0\\
0 & 0 & 0 & 1
\end{pmatrix},
\begin{pmatrix}
0 & -1 & 0 & 0\\
1 & 0 & 0 & 0\\
0 & 0 & 0 & -1\\
0 & 0 & 1 & 0
\end{pmatrix}.
$$
\end{itemize}
Then  $\mathbb{P}^3$ is \mbox{$G$-solid}, and $\mathcal{P}_G(\mathbb{P}^3)=\{\mathbb{P}^3,X_{24}\}$,
where $X_{24}$ is the~threefold from Example~\ref{example:Fano-Enriques}.
\end{maintheorem*}

In this paper, the notation $G_{a,b}$ or $G_{a,b}^\prime$ means that the GAP ID of these groups is [a,b].

If $G$ is conjugate to one of the~subgroups $G_{48,3}$, $G_{96,72}$, $G_{324,160}^\prime$, then $\mathbb{P}^3$ is not $G$-solid:

\begin{example}
\label{example:48-96-324-pencil}
Suppose that $G$ is one of the~groups $G_{48,3}$, $G_{96,72}$, $G_{324,160}^\prime$. Let
$$
\mathfrak{C}=\big\{\big(1+e^{\frac{2\pi i}{3}}\big)x_1^d+e^{\frac{2\pi i}{3}}x_2^d+x_3^d=x_0^d+e^{\frac{2\pi i}{3}} x_1^d-\big(1+e^{\frac{2\pi i}{3}}\big)x_2^d=0\big\}\subset\mathbb{P}^3,
$$
where
$$
d=\left\{\aligned
&2\ \text{if}\ G=G_{48,3}\ \text{or}\ G=G_{96,72},\\
&3\ \text{if}\ G=G_{324,160}^\prime.
\endaligned
\right.
$$
Then~$\mathfrak{C}$ is a smooth irreducible $G$-invariant curve, and there exists $G$-equivariant diagram
$$
\xymatrix{
&X\ar@{->}[ld]_{\vartheta}\ar@{->}[rd]^{\kappa}&\\%
\mathbb{P}^3&&\mathbb{P}^1}
$$
where $\vartheta$ is the~blow up of the~curve $\mathfrak{C}$,
and $\kappa$ is a $G$-Mori fibre space, which is a fibration into surfaces of degree $d$.
\end{example}

\begin{corollary}[{cf. \cite[Theorem~1.1]{CheltsovShramov2017}}]
\label{corollary:solid}
Let $G$ be an arbitrary finite subgroup in $\mathrm{PGL}_4(\mathbb{C})$.
Then $\mathbb{P}^3$ is $G$-solid if and only if the following conditions are satisfied:
\begin{itemize}
\item[(a)] $G$ does not fix a point,
\item[(b)] $G$ does not leave a pair of two skew lines invariant,
\item[(c)] $G$ is not isomorphic to $\mathfrak{A}_5$ or $\mathfrak{S}_5$,
\item[(d)] $G$ is not conjugate to $G_{48,3}$, $G_{96,72}$ or $G_{324,160}^\prime$.
\end{itemize}
\end{corollary}

This corollary describes all finite subgroups $G\subset\mathrm{PGL}_4(\mathbb{C})$ such that
the~projective space $\mathbb{P}^3$ is not $G$-birational to conic bundles and del Pezzo fibrations.
For the~projective plane $\mathbb{P}^2$, a similar problem has been solved in \cite{Sakovics}.

For the group $G_{324,160}^\prime$, we prove the following result.

\begin{theorem}
\label{theorem:324-160}
Suppose that $G=G_{324,160}^\prime$. Then $\mathbb{P}^3$, the~threefold $X_{24}$ from Example~\ref{example:Fano-Enriques},
and the~$G$-Mori fibre space $\kappa\colon X\to\mathbb{P}^1$ from Example~\ref{example:48-96-324-pencil}
are the~only $G$-Mori fiber~spaces that are $G$-birational to the~projective space $\mathbb{P}^3$.
\end{theorem}

We expect that a similar result holds also in the case when $G=G_{48,3}$ or $G=G_{96,72}$.
We~plan to prove this in a sequel to this paper together with Igor Krylov by combining our technique with the~methods developed in \cite{Pukhlikov1998,Co00,Krylov}.

\begin{remark}
\label{remark:intro-conclusion}
Our technique is not applicable in the case when the~group $G$ is intransitive.
In this case, the $G$-equivariant birational geometry of $\mathbb{P}^3$ has been studied in \cite{KreschTschinkel2021}
using the very powerful new technique recently developed in \cite{KontsevichPestunTschinkel,HassettKreschTschinkel,KreschTschinkel2020}.
\end{remark}

Using Main Theorem and Theorem~\ref{theorem:324-160}, one can construct examples of non-conjugate isomorphic finite subgroups in $\mathrm{Bir}(\mathbb{P}^3)$.
Let us present three such examples.

\begin{example}
\label{example:324-160-non-conjugate}
Let $G_{324,160}$ be the~subgroup in $\mathrm{PGL}_4(\mathbb{C})$ generated by
$$
\begin{pmatrix}
e^{\frac{2\pi i}{3}} & 0 & 0 & 0\\
0 & 1 & 0 & 0\\
0 & 0 & 1 & 0\\
0 & 0 & 0 & 1
\end{pmatrix},
\begin{pmatrix}
1 & 0 & 0 & 0\\
0 & e^{\frac{2\pi i}{3}} & 0 & 0\\
0 & 0 & 1 & 0\\
0 & 0 & 0 & 1
\end{pmatrix},
\begin{pmatrix}
1 & 0 & 0 & 0\\
0 & 1 & 0 & 0\\
0 & 0 & e^{\frac{2\pi i}{3}} & 0\\
0 & 0 & 0 & 1
\end{pmatrix},
\begin{pmatrix}
0 & 0 & 1 & 0\\
1 & 0 & 0 & 0\\
0 & 1 & 0 & 0\\
0 & 0 & 0 & 1
\end{pmatrix},
\begin{pmatrix}
0 & 1 & 0 & 0\\
1 & 0 & 0 & 0\\
0 & 0 & 0 & 1\\
0 & 0 & 1 & 0
\end{pmatrix}.
$$
Then $G_{324,160}\cong G_{324,160}^\prime$, $\mathbb{P}^3$ is $G_{324,160}$-solid by Main Theorem, but $\mathbb{P}^3$ is not $G_{324,160}^\prime$-solid.
Hence, the subgroups $G_{324,160}$  and $G_{324,160}^\prime$ are not conjugate in $\mathrm{Bir}(\mathbb{P}^3)$.
\end{example}

\begin{example}
\label{example:96-227-non-conjugate}
Let $G_{96,227}$ be the~subgroup in $\mathrm{PGL}_4(\mathbb{C})$ generated by
$$
\begin{pmatrix}
-1 & 0 & 0 & 0\\
0 & 1 & 0 & 0\\
0 & 0 & -1 & 0\\
0 & 0 & 0 & 1
\end{pmatrix},
\begin{pmatrix}
1 & 0 & 0 & 0\\
0 & -1 & 0 & 0\\
0 & 0 & -1 & 0\\
0 & 0 & 0 & 1
\end{pmatrix},
\begin{pmatrix}
0 & 0 & 0 & 1\\
1 & 0 & 0 & 0\\
0 & 1 & 0 & 0\\
0 & 0 & 1 & 0
\end{pmatrix},
\begin{pmatrix}
0 & 1 & 0 & 0\\
1 & 0 & 0 & 0\\
0 & 0 & 1 & 0\\
0 & 0 & 0 & 1
\end{pmatrix},
$$
and let $G_{96,227}^\prime$  be the~subgroup in $\mathrm{PGL}_4(\mathbb{C})$ generated by
$$
\begin{pmatrix}
-1 & 0 & 0 & 0\\
0 & 1 & 0 & 0\\
0 & 0 & -1 & 0\\
0 & 0 & 0 & 1
\end{pmatrix},
\begin{pmatrix}
1 & 0 & 0 & 0\\
0 & -1 & 0 & 0\\
0 & 0 & -1 & 0\\
0 & 0 & 0 & 1
\end{pmatrix},
\begin{pmatrix}
0 & 0 & 0 & -1\\
1 & 0 & 0 & 0\\
0 & 1 & 0 & 0\\
0 & 0 & 1 & 0
\end{pmatrix},
\begin{pmatrix}
0 & 1 & 0 & 0\\
1 & 0 & 0 & 0\\
0 & 0 & -1 & 0\\
0 & 0 & 0 & 1
\end{pmatrix}.
$$
Then $G_{96,227}\cong G_{96,227}^\prime\cong\mumu_2^2\rtimes\mathfrak{S}_4$,
and these two subgroups are not conjugate in $\mathrm{PGL}_4(\mathbb{C})$, because
$\mathbb{P}^3$ contains three $G_{96,227}$-orbits of length $4$ and only one $G_{96,227}^\prime$-orbit of length~$4$.
Thus, applying Main Theorem, we see that $G_{96,227}$ and $G_{96,227}^\prime$ are not conjugate in $\mathrm{Bir}(\mathbb{P}^3)$.
\end{example}

\begin{example}
\label{example:48-50-quadric-P1}
Let $G_{48,50}\cong\mumu^2_2\rtimes\mathfrak{A}_4\cong\mumu_2^4\rtimes\mumu_3$ be the subgroup in $\mathrm{PGL}_4(\mathbb{C})$
generated~by
$$
\begin{pmatrix}
-1 & 0 & 0 & 0\\
0 & 1 & 0 & 0\\
0 & 0 & -1 & 0\\
0 & 0 & 0 & 1
\end{pmatrix},
\begin{pmatrix}
1 & 0 & 0 & 0\\
0 & -1 & 0 & 0\\
0 & 0 & -1 & 0\\
0 & 0 & 0 & 1
\end{pmatrix},
\begin{pmatrix}
0 & 0 & 1 & 0\\
1 & 0 & 0 & 0\\
0 & 1 & 0 & 0\\
0 & 0 & 0 & 1
\end{pmatrix},
\begin{pmatrix}
0 & 1 & 0 & 0\\
1 & 0 & 0 & 0\\
0 & 0 & 0 & 1\\
0 & 0 & 1 & 0
\end{pmatrix}.
$$
Let $\mathcal{Q}_1=\{x_0^2+x_1^2+x_2^2+x_3^2=0\}\subset\mathbb{P}^3$.
Then $\mathcal{Q}_1$ is $G_{48,50}$-invariant,
which gives a faithful action of the group $G_{48,50}$ on $\mathcal{Q}_1\times\mathbb{P}^1$,
that induces an~embedding $\eta\colon G_{48,50}\hookrightarrow\mathrm{Bir}(\mathbb{P}^3)$.
Since $\mathbb{P}^3$ is $G_{48,50}$-solid, the subgroups $G_{48,50}$ and $\eta(G_{48,50})$ are not conjugate in~$\mathrm{Bir}(\mathbb{P}^3)$.
\end{example}

In this paper, we also find the generators of the group $\mathrm{Bir}^G(\mathbb{P}^3)$
for every imprimitive finite subgroup $G\subset\mathrm{PGL}_4(\mathbb{C})$ such that $\mathbb{P}^3$ is $G$-solid.
In particular, we show that this group is~finite provided that $G$ is not conjugate to $G_{48,50}$ or $G_{96,227}$ (see Corollary~\ref{corollary:final-large-groups}).
On~the other hand, if~$G=G_{48,50}$ or $G=G_{96,227}$, then $\mathrm{Bir}^G(\mathbb{P}^3)$ is infinite by Corollary~\ref{corollary:final}.
In these two cases, the group $\mathrm{Bir}^G(\mathbb{P}^3)$ is generated by the standard Cremona involution
$$
[x_0:x_1:x_2:x_3]\mapsto [x_1x_2x_3:x_0x_2x_3:x_0x_1x_3:x_0x_1x_2]
$$
and the finite subgroup $G_{576,8654}\cong(\mathfrak{A}_4\times\mathfrak{A}_4)\rtimes\mumu_2^2$ generated by
$$
\begin{pmatrix}
-1&0&0&0\\
0&1&0&0\\
0&0&1&0\\
0&0&0&1
\end{pmatrix},
\begin{pmatrix}
0&0&0&1\\
1&0&0&0\\
0&1&0&0\\
0&0&1&0
\end{pmatrix},
\begin{pmatrix}
0&1&0&0\\
1&0&0&0\\
0&0&1&0\\
0&0&0&1
\end{pmatrix},
\begin{pmatrix}
1&1&1&1\\
1&1&-1&-1\\
1&-1&1&-1\\
-1&1&1&-1
\end{pmatrix}.
$$

Let us describe the structure of this paper.
We will prove Main Theorem in Section~\ref{section:proof}.
In Section~\ref{section:subgroups}, we~will describe basic properties of finite monomial subgroups in $\mathrm{PGL}_4(\mathbb{C})$.
In Sections~\ref{section:P3-48},~\ref{section:P3-192} and \ref{section:P3-large},
we~will study $G$-equivariant geometry of the~projective space $\mathbb{P}^3$,
where $G$ is a finite subgroup  in $\mathrm{PGL}_4(\mathbb{C})$ that satisfies all conditions of Main Theorem.
In~Section~\ref{section:Fano-Enriques}, we will study $G$-equivariant geometry of the~threefold $X_{24}$ from Example~\ref{example:Fano-Enriques}.

\medskip
\textbf{Acknowledgments.}
We want to thank Hamid Abban, Michela Artebani, Igor Krylov, 	
Jennifer Paulhus, Yuri Prokhorov, Xavier Roulleau, Alessandra Sarti, Costya Shramov, Andrey Trepalin,
Yuri Tschinkel for helpful comments.
We want to thank Tim Dokchitser for his help with Magma computations and his online database \cite{Tim}.

\section{Irreducible monomial subgroups of degree four}
\label{section:subgroups}

Let $G$ be a finite transitive subgroup in $\mathrm{PGL}_4(\mathbb{C})$ such that $\mathbb{P}^3$ has a~$G$-orbit of length~$4$,
and let $P_1$, $P_2$, $P_3$, $P_4$ be the~four points of this $G$-orbit.
Choosing appropriate coordinates, we may assume that
\begin{center}
$P_1=[1:0:0:0]$, $P_2=[0:1:0:0]$, $P_3=[0:0:1:0]$, $P_4=[0:0:0:1]$.
\end{center}
Then the~$G$-action on the~set $\{P_1,P_2,P_3,P_4\}$ induces a group homomorphism $\upsilon\colon G\to\mathfrak{S}_4$.
Denote by $T$ the~kernel of the~homomorphism $\upsilon$. Suppose, in addition, that the following two conditions are satisfied:
\begin{itemize}
\item $G$ does not have fixed points in  $\mathbb{P}^3$,
\item $G$ does not leave a union of two skew lines in $\mathbb{P}^3$ invariant.
\end{itemize}
Then $T$ is not trivial, and either the~homomorphism $\upsilon$ is surjective, or its image is $\mathfrak{A}_4$.

Let $\mathbb{T}$ be the~torus in $\mathrm{PGL}_4(\mathbb{C})$ that consists of the~elements
given by the~diagonal matrices whose last entry is $1$. In the following, we will always shortcut
$$
(a_1,a_2,a_3)=\begin{pmatrix}
a_1 & 0 & 0 & 0\\
0 & a_2 & 0 & 0\\
0 & 0 & a_3 & 0\\
0 & 0 & 0 & 1
\end{pmatrix}.
$$
Note that we have $T\subset\mathbb{T}$.
Let $\mathbb{G}$ be the~normalizer of the~torus $\mathbb{T}$ in the~group $\mathrm{PGL}_4(\mathbb{C})$.
Then the~subset $\{P_1,P_2,P_3,P_4\}$ is $\mathbb{G}$-invariant, which gives an~epimorphism $\Upsilon\colon\mathbb{G}\to\mathfrak{S}_4$.
Since we have $G\subset\mathbb{G}$, we obtain the~following exact sequences of groups:
$$
\xymatrix{1\ar@{->}[rr] && T\ar@{_{(}->}[d]\ar@{->}[rr] && G\ar@{->}[rr]^{\upsilon}\ar@{_{(}->}[d] && \mathrm{im}(\upsilon)\ar@{_{(}->}[d]\ar@{->}[rr] && 1\\
1\ar@{->}[rr] && \mathbb{T}\ar@{->}[rr] && \mathbb{G}\ar@{->}[rr]^{\Upsilon} && \mathfrak{S}_4\ar@{->}[rr] && 1.}
$$
Note that $\mathbb{G}\cong\mathbb{T}\rtimes\mathfrak{S}_4$,
where we identify $\mathfrak{S}_4$ with the~subgroup in $\mathbb{G}$ generated by
$$
\tau=\begin{pmatrix}
0 & 0 & 0 & 1\\
1 & 0 & 0 & 0\\
0 & 1 & 0 & 0\\
0 & 0 & 1 & 0
\end{pmatrix}\
\text{and}\
\sigma=\begin{pmatrix}
0 & 1 & 0 & 0\\
1 & 0 & 0 & 0\\
0 & 0 & 1 & 0\\
0 & 0 & 0 & 1
\end{pmatrix}.
$$
The induced $\mathbb{G}$-action on $\mathbb{T}$ gives the~injective homomorphism
$$
\mathfrak{S}_4\cong\mathbb{G}/\mathbb{T}\rightarrow\mathrm{Aut}(\mathbb{T}),
$$
and the~corresponding action of the~group $\mathfrak{S}_4=\langle\tau,\sigma\rangle$ on $\mathbb{T}$ can be described as follows:
\begin{align*}
\tau &\colon (a_1,a_2,a_3)\longmapsto \left(\frac{a_2}{a_1},\frac{a_3}{a_1},\frac{1}{a_1}\right),\\
\sigma& \colon (a_1,a_2,a_3)\longmapsto (a_2,a_1,a_3).
\end{align*}
Clearly, if $\mathrm{im}(\upsilon)=\mathfrak{S}_4$, then $T$ is $\tau$-invariant and $\sigma$-invariant.

Let $h$ be an element in $T$ of maximal order $n\geqslant 1$.
Then the~order of every element in the~group $T$ divides $n$, hence $T\subseteq\mumu_n^3$.
Here, we identify $\mumu_n^3$ with the~subgroup
$$
\big\langle(\zeta_n,1,1),(1,\zeta_n,1),(1,1,\zeta_n)\big\rangle\subset\mathbb{T},
$$
where $\zeta_n=e^{\frac{2\pi i}{n}}$.

\begin{lemma}[{cf. \cite{Flannery}, \cite[Theorem 4.7]{DolgachevIskovskikh}, \cite[Corollary~7.3]{CDK}}]
\label{lemma:S4}
Suppose that $\mathrm{im}(\upsilon)=\mathfrak{S}_4$.
Then one of the~following assertions holds:
\begin{itemize}
\item[\textup(1)] $T=\mumu_n^3$;
\item[\textup(2)] $n$ is even and $T\cong\mumu_n^2\times\mumu_{\frac{n}{2}}$;
\item[\textup(3)] $n$ is divisible by $4$ and $T\cong\mumu_n^2\times\mumu_{\frac{n}{4}}$.
\end{itemize}
\end{lemma}

\begin{proof}
We have $h=(\zeta_n^a,\zeta_n^b,\zeta_n^c)$ for coprime non-negative integers $a$, $b$, $c$.
Applying cyclic permutation of order $3$~to~$h$,
we~see that  $(\zeta_n^b,\zeta_n^c,\zeta_n^a)$ and $(\zeta_n^c,\zeta_n^a,\zeta_n^b)$ are contained in $T$.
Hence, the~group $T$ contains $(\zeta_n,\zeta_n^\beta,\zeta_n^\gamma)$ for some non-negative integers $\beta$ and $\gamma$.
Then
$$
\left(\tau(\zeta_n,\zeta_n^\beta,\zeta_n^\gamma)\cdot(\zeta_n^{-\beta},\zeta_n^{-\gamma},\zeta_n^{-1})\right)^{-1}=(\zeta_n,\zeta_n,\zeta_n^2)\in T.
$$
Thus, we get $\tau(\zeta_n,\zeta_n,\zeta_n^2)=(1,\zeta_n,\zeta_n^{-1})\in T$ and so $(1,\zeta_n^{-1},\zeta_n)\in T$. Then
$$
(\zeta_n,\zeta_n,\zeta_n^2)\cdot(1,\zeta_n^{-1},\zeta_n)\cdot(\zeta_n^{-1},1,\zeta_n)=(1,1,\zeta_n^4)\in T,
$$
Now, we let $T^\prime=\langle (\zeta_n,1,\zeta_n^{-1}),(1,\zeta_n,\zeta_n^{-1}),(1,1,\zeta_n^4)\rangle$.
Then the~subgroup $T^\prime$ is $\mathfrak{S}_4$-invariant.
Moreover, we have  $T^\prime\subseteq T\subseteq \mumu_n^3$.
Furthermore, we have the~following possibilities:
\begin{itemize}
\item[\textup(1)] If $n$ is odd, then $T^\prime=\mumu_n^3$, hence $T=T^\prime=\mumu_n^3$.
\item[\textup(2)] If $n$ is divisible by $2$ but not by $4$, then $T^\prime\cong\mumu_n^2\times\mumu_{\frac{n}{2}}$.
\item[\textup(3)] If $n$ is divisible by $4$, then $T^\prime\cong\mumu_n^2\times\mumu_{\frac{n}{4}}$.
\end{itemize}
In the~case (2), if there exists $t\in T\setminus T^\prime$, then we have $\langle t, T^\prime\rangle=\mumu^3_n$, hence we are done.
Therefore, we may assume that we are in the~case (3).
As above, if there exists $t\in T\setminus T^\prime$, then either $\langle t, T^\prime\rangle\cong \mumu_n^2\times\mumu_{\frac{n}{2}}$ or $\langle t, T^\prime\rangle=\mumu_n^3$.
\end{proof}

\begin{corollary}
\label{corollary:odd}
Suppose that $\mathrm{im}(\upsilon)=\mathfrak{S}_4$, $T=\mumu_n^3$ and $n$ is odd.
Then $G$ is conjugated~to the~subgroup generated by
\begin{equation*}
(\zeta_n,1,1),(1,\zeta_n,1),(1,1,\zeta_n),
\begin{pmatrix}
0 & 0 & 0 & 1\\
1 & 0 & 0 & 0\\
0 & 1 & 0 & 0\\
0 & 0 & 1 & 0
\end{pmatrix},
\begin{pmatrix}
0 & 1 & 0 & 0\\
1 & 0 & 0 & 0\\
0 & 0 & 1 & 0\\
0 & 0 & 0 & 1
\end{pmatrix},
\end{equation*}
or $G$ is conjugated to the~subgroup generated by
$$
(\zeta_n,1,1),(1,\zeta_n,1),(1,1,\zeta_n),
\begin{pmatrix}
0 & 0 & 0 & -1\\
1 & 0 & 0 & 0\\
0 & 1 & 0 & 0\\
0 & 0 & 1 & 0
\end{pmatrix},
\begin{pmatrix}
0 & 1 & 0 & 0\\
1 & 0 & 0 & 0\\
0 & 0 & -1 & 0\\
0 & 0 & 0 & 1
\end{pmatrix}.
$$
In both cases, we have $G\cong T\rtimes \mathfrak{S}_4$.
\end{corollary}

\begin{proof}
Let $A$ and $B$ be some elements in the~group $G$ such that $\upsilon(A)=\tau$ and $\upsilon(B)=\sigma$.
If $A^4=B^2=(AB)^3=1$, then $\langle A,B\rangle\cong\mathfrak{S}_4$, hence $G\cong T\rtimes \mathfrak{S}_4$.
We have
\begin{equation*}
A=
\begin{pmatrix}
0 & 0 & 0 & 1\\
a_1 & 0 & 0 & 0\\
0 & a_2 & 0 & 0\\
0 & 0 & a_3 & 0
\end{pmatrix}\ \text{and}\ B=
\begin{pmatrix}
0 & b_2 & 0 & 0\\
b_1 & 0 & 0 & 0\\
0 & 0 & b_3 & 0\\
0 & 0 & 0 & 1
\end{pmatrix}
\end{equation*}
where each $a_i$ and $b_j$ are non-zero complex numbers.
Conjugating $G$ by an appropriate element of the~torus $\mathbb{T}$,
we can assume that $a_1=a_2=a_3=1$.

Since $\tau^4=\sigma^2=(\tau\sigma)^3=1$, we see that $B^2\in T$ and $(AB)^3\in T$, which gives
$$
\begin{cases}
b_1b_2=\zeta_n^\alpha, \\
b_3^2=\zeta_n^\beta, \\
\frac{b_2^3}{b_1b_3}=\zeta_n^\gamma
\end{cases}
$$
for some some numbers $\alpha$, $\beta$, $\gamma$ in $\{0,\ldots n-1\}$.
Hence, replacing $B$ with
$$
\big(\zeta_n^{\frac{-6\alpha+\beta+2\gamma}{8}},\zeta_n^{\frac{-2\alpha-\beta-2\gamma}{8}},\zeta_n^{-\frac{\beta}{2}}\big)B\in G,
$$
we may assume that $B^2=1$. Here, we consider division by $8$ and $2$ as division modulo~$n$.
In particular, we see that $G\cong T\rtimes \mathfrak{S}_4$ as claimed.

Now, we observe that $b_1b_2=1$, $b_3^2=1$, $b_2^3=b_1b_3$.
Then solving this system of equation, we obtain the following eight cases:
\begin{center}
\renewcommand\arraystretch{1.5}
\begin{tabular}{|c||c|c|c|}
\hline
 & \quad\quad $b_1$\quad\quad\quad & \quad\quad$b_2$\quad\quad\quad & \quad\quad$b_3$\quad\quad\quad\\
\hline
\hline
(i) & $1$& $1$& $1$\\
\hline
(ii) & $-1$& $-1$& $1$\\
\hline
(iii) & $i$& $-i$& $1$\\
\hline
(iv) & $-i$& $i$& $1$\\
\hline
(v) & $-\frac{\sqrt{2}}{2}+\frac{\sqrt{2}i}{2}$& $-\frac{\sqrt{2}}{2}-\frac{\sqrt{2}i}{2}$& $-1$\\
\hline
(vi) & $\frac{\sqrt{2}}{2}-\frac{\sqrt{2}i}{2}$   &  $\frac{\sqrt{2}}{2}+\frac{\sqrt{2}i}{2}$& $-1$\\
\hline
(vii) & $\frac{\sqrt{2}}{2}+\frac{\sqrt{2}i}{2}$   & $\frac{\sqrt{2}}{2}-\frac{\sqrt{2}i}{2}$& $-1$\\
\hline
(viii) & $-\frac{\sqrt{2}}{2}-\frac{\sqrt{2}i}{2}$ & $-\frac{\sqrt{2}}{2}+\frac{\sqrt{2}i}{2}$& $-1$\\
\hline
\end{tabular}
\end{center}
In case (i), we are done. In cases (ii), (iii) and (iv), we can conjugate $G$ to get the first group in the assertion of the lemma
using $(-1,1,-1)$, $(-i,-1,i)$, $(i,-1,-i)$, respectively.
Similarly, in cases (v), (vi), (vii), (viii), we can conjugate $G$ to get the second group using
\begin{multline*}
\Big(\frac{\sqrt{2}}{2}+\frac{\sqrt{2}i}{2},-i,-\frac{\sqrt{2}}{2}+\frac{\sqrt{2}i}{2}\Big), \Big(-\frac{\sqrt{2}}{2}-\frac{\sqrt{2}i}{2},-i,\frac{\sqrt{2}}{2}-\frac{\sqrt{2}i}{2}\Big),\\
\Big(-\frac{\sqrt{2}}{2}+\frac{\sqrt{2}i}{2},i,\frac{\sqrt{2}}{2}+\frac{\sqrt{2}i}{2}\Big), \Big(\frac{\sqrt{2}}{2}-\frac{\sqrt{2}i}{2},i,-\frac{\sqrt{2}}{2}-\frac{\sqrt{2}i}{2}\Big),
\end{multline*}
respectively. This completes the proof.
\end{proof}

Now, we describe the~possibilities for the~subgroup $T$ in the~case when $\mathrm{im}(\upsilon)=\mathfrak{A}_4$.
Keeping in mind our identification $\mathfrak{S}_4=\langle\tau,\sigma\rangle$, we see that $\mathfrak{A}_4=\langle\rho,\varsigma\rangle$ for
$$
\rho=\begin{pmatrix}
0 & 0 & 1 & 0\\
1 & 0 & 0 & 0\\
0 & 1 & 0 & 0\\
0 & 0 & 0 & 1
\end{pmatrix}\
\text{and}\
\varsigma=\begin{pmatrix}
0 & 1 & 0 & 0\\
1 & 0 & 0 & 0\\
0 & 0 & 0 & 1\\
0 & 0 & 1 & 0
\end{pmatrix}.
$$
Then $\rho$ and $\varsigma$ acts on $\mathbb{T}$ as follows:
\begin{align*}
\rho& \colon (a_1,a_2,a_3)\longmapsto (a_2,a_3,a_1),\\
\varsigma& \colon (a_1,a_2,a_3)\longmapsto\left(\frac{a_2}{a_3},\frac{a_1}{a_3},\frac{1}{a_3}\right).
\end{align*}
Clearly, our subgroup $T$ is $\rho$-invariant and $\varsigma$-invariant. Using this, we get

\begin{lemma}[{cf. \cite[Corollary~7.3]{CDK}}]
\label{lemma:A4}
Suppose $\mathrm{im}(\upsilon)=\mathfrak{A}_4$.
One of the~following~holds:
\begin{itemize}
\item[\textup(1)] $T=\mumu_n^3$;
\item[\textup(2)] $n$ is even and $T\cong\mumu_n^2\times\mumu_{\frac{n}{2}}$;
\item[\textup(3)] $n$ is divisible by $4$ and $T\cong\mumu_n^2\times\mumu_{\frac{n}{4}}$.
\end{itemize}
\end{lemma}

\begin{proof}
Arguing as in the~proof of Lemma~\ref{lemma:S4},
we may assume that $(\zeta_n,\zeta_n^\beta,\zeta_n^\gamma)\in T$ for some non-negative integers $\beta$ and $\gamma$,
where $n\geqslant 1$ is the~largest order of all elements in $T$, and $\zeta_n$ is a primitive $n$-th root of unity.
Then
$$
(\zeta_n,\zeta_n^\beta,\zeta_n^\gamma)\cdot\varsigma (\zeta_n,\zeta_n^\beta,\zeta_n^\gamma)\cdot\varsigma(\zeta_n^\beta,\zeta_n^\gamma,\zeta_n)\cdot\varsigma (\zeta_n^\gamma,\zeta_n,\zeta_n^\beta)=(\zeta_n, \zeta_n^{\beta}, \zeta_n^{-\beta-1})\in T.
$$
So, we have
$$
\big(\rho(\zeta_n, \zeta_n^{\beta}, \zeta_n^{-\beta-1})\cdot\varsigma\rho(\zeta_n, \zeta_n^{\beta}, \zeta_n^{-\beta-1})\big)^{-1}=(\zeta_n^2,\zeta_n^2,1)\in T
$$
and
$$
\Big((\zeta_n^2,\zeta_n^2,1)\cdot\big(\rho(\zeta_n^2,\zeta_n^2,1)\big)^{-1}\Big)^{-1}=(1,\zeta_n^{-2},\zeta_n^{2})\in T.
$$

If $\beta =2k$ for $k\in \mathbb{Z}_{\geq 0}$, then
$$
(\zeta_n, \zeta_n^{\beta}, \zeta_n^{-\beta-1})\cdot(1,\zeta_n^{-2},\zeta_n^{2})^k=(\zeta_n,1,\zeta^{-1}_n)\in T.
$$
Thereby, we see that $(1,\zeta_n,\zeta_n^{-1})$ and $(1,1,\zeta^{4}_n)$ are both contained in $T$.
Now, arguing as in the~end of the~proof of Lemma~\ref{lemma:S4}, we obtain the~required result.

Likewise, if $\beta =2k+1$, then
$$
(\zeta_n, \zeta_n^{\beta}, \zeta_n^{-\beta-1})\cdot(1,\zeta_n^{-2},\zeta_n^{2})^k=(\zeta_n,\zeta_n,\zeta^{-2}_n)\in T.
$$
Hence, we have
$$
(\rho((\zeta_n,\zeta_n,\zeta^{-2}_n)\cdot (1,\zeta_n^{-2},\zeta_n^{2})))^{-1}=(\zeta_n,1,\zeta^{-1}_n)\in T
$$
and we are done similarly to the previous case.
\end{proof}

Arguing as in the~proofs of Lemmas~\ref{lemma:S4} and \ref{lemma:A4}, we obtain the~following result:

\begin{lemma}
\label{lemma:P3-P3}
The group $G$ is conjugated in $\mathrm{PGL}_4(\mathbb{C})$ to a subgroup $\langle T,A,B\rangle\subset\mathrm{PGL}_4(\mathbb{C})$,
where $A$ and $B$ are elements in $\mathrm{PGL}_4(\mathbb{C})$ described as follows:
\begin{itemize}
\item if $\mathrm{im}(\upsilon)=\mathfrak{S}_4$, then
$$
A=\begin{pmatrix}
0 & 0 & 0 & \frac{a^2}{b}\\
1 & 0 & 0 & 0\\
0 & 1 & 0 & 0\\
0 & 0 & 1 & 0
\end{pmatrix}
\ \text{and}\
B=\begin{pmatrix}
0 & a & 0 & 0\\
a & 0 & 0 & 0\\
0 & 0 & b & 0\\
0 & 0 & 0 & 1
\end{pmatrix}
$$
for some complex numbers $a$ and $b$ such that $(a^2,a^2,b^2)\in T$;

\item if $\mathrm{im}(\upsilon)=\mathfrak{A}_4$, then
$$
A=\begin{pmatrix}
0 & 0 & 1 & 0\\
1 & 0 & 0 & 0\\
0 & 1 & 0 & 0\\
0 & 0 & 0 & 1
\end{pmatrix}
\ \text{and}\
B=\begin{pmatrix}
0 & a & 0 & 0\\
b & 0 & 0 & 0\\
0 & 0 & 0 & a\\
0 & 0 & 1 & 0
\end{pmatrix}
$$
for some complex numbers $a$ and $b$ such that $(a^2,b,1)\in T$ and $(b,b,1)\in T$.
\end{itemize}
\end{lemma}

\begin{proof}
First, we suppose that $\mathrm{im}(\upsilon)=\mathfrak{S}_4$.
Let $A$ and $B$ be elements in the~group $G$ such that $\upsilon(A)=\tau$ and $\upsilon(B)=\sigma$.
Conjugating $G$ by elements of $\mathbb{T}$, we may assume that
\begin{equation*}
A=
\begin{pmatrix}
0 & 0 & 0 & c\\
1 & 0 & 0 & 0\\
0 & 1 & 0 & 0\\
0 & 0 & 1 & 0
\end{pmatrix}\ \text{and}\ B=
\begin{pmatrix}
0 & a & 0 & 0\\
a & 0 & 0 & 0\\
0 & 0 & b & 0\\
0 & 0 & 0 & 1
\end{pmatrix}
\end{equation*}
for some non-zero complex numbers $a$, $b$ and $c$.
Then $B^2=(a^2,a^2,b^2)\in T$ and
$$
(AB)^3=\Big(1,\frac{a^2}{bc},1\Big)\in T.
$$
Now, using the~$\mathfrak{S}_4$-action on $T$, we see that $(1,1,\frac{a^2}{bc})\in T$.
Then, replacing $A\mapsto (1,1,\frac{a^2}{bc})A$, we obtain the~required assertion in the~case when $\mathrm{im}(\upsilon)=\mathfrak{S}_4$.

Now, we suppose that $\mathrm{im}(\upsilon)=\mathfrak{A}_4$.
As above, let $A$ and $B$ be some elements in $G$ such that $\upsilon(A)=\rho$ and $\upsilon(B)=\varsigma$.
Conjugating $G$ by elements in $\mathbb{T}$, we may assume that
\begin{equation*}
A=
\begin{pmatrix}
0 & 0 & 1 & 0\\
1 & 0 & 0 & 0\\
0 & c & 0 & 0\\
0 & 0 & 0 & 1
\end{pmatrix}\ \text{and}\ B=
\begin{pmatrix}
0 & a & 0 & 0\\
b & 0 & 0 & 0\\
0 & 0 & 0 & a\\
0 & 0 & 1 & 0
\end{pmatrix}
\end{equation*}
for some non-zero numbers $a$, $b$ and $c$.
Then $A^3=(c,c,c)\in T$, $B^2=(b,b,1)\in T$ and
$$
(AB)^3=\Big(1,\frac{a^2}{bc},1\Big)\in T.
$$
Now, using the~$\mathfrak{A}_4$-action on $T$, we get $(1,1,\frac{1}{c})\in T$.
Then, after replacing $A\mapsto (1,1,\frac{1}{c})A$, we obtain the~required assertion.
\end{proof}

\begin{corollary}
\label{corollary:P3-P3}
In the~assumption and notations of Lemma~\ref{lemma:P3-P3}, let $G^\prime=\langle T,A,B\rangle$,
and let $\iota$ be the~standard Cremona involution given by
$$
[x_0:x_1:x_2:x_3]\mapsto [x_1x_2x_3:x_0x_2x_3:x_0x_1x_3:x_0x_1x_2].
$$
Then $\iota G^\prime\iota=G^\prime$,
so $\iota G\iota$ is conjugated to $G$ in $\mathrm{PGL}_4(\mathbb{C})$.
\end{corollary}

\begin{proof}
Observe that $\iota T\iota=T$. Thus, to complete the~proof, it is enough to show that $\iota A\iota$ and $\iota B\iota$ are both contained in $G^\prime$.
If $\mathrm{im}(\upsilon)=\mathfrak{A}_4$, then $\iota A\iota=A\in G^\prime$ and
$$
\iota B\iota=\Big(\frac{1}{a^2},\frac{1}{b^2},\frac{1}{a^2}\Big)B,
$$
so that it is enough to show that $(a^2,b^2,a^2)\in T$.
In this case, it follows from Lemma~\ref{lemma:P3-P3} that $(a^2,b,1)\in T$ and $(b,b,1)\in T$,
so that, using the~$\mathfrak{A}_4$-action on $T$ described earlier, we see that $(b,1,a^2)$, $(1,b,b)$ and $(b,1,b)$ are contained in $T$ as well, which gives
$$
(a^2,b^2,a^2)=(b,1,a^2)(a^2,b,1)(1,b,b)(b,1,b)^{-1}\in T,
$$
which is exactly what we want. So, we may assume that $\mathrm{im}(\upsilon)=\mathfrak{S}_4$.
Then $\iota B\iota=B\in G^\prime$. Hence, to complete the~proof, it is enough to show that $\iota A\iota\in G^\prime$. But
$$
\iota A\iota=\Big(\frac{b^2}{a^4},1,1\Big)A,
$$
so it is enough to show that $(\frac{a^4}{b^2},1,1)\in T$.
But we have $(a^2,a^2,b^2)\in T$ by Lemma~\ref{lemma:P3-P3}.
Thus, using the~$\mathfrak{S}_4$-action on $T$, we see that $(1,\frac{b^2}{a^2},\frac{1}{a^2})\in T$ and $(b^2,a^2,a^2)\in T$,
so that
$$
(b^2,b^2,1)=\Big(1,\frac{b^2}{a^2},\frac{1}{a^2}\Big)(b^2,a^2,a^2)\in T,
$$
and, using the~$\mathfrak{S}_4$-action on $T$, we get $(b^2,1,b^2)\in T$.
On the~other hand, it follows from the~proof of Lemma~\ref{lemma:S4} that $T$ contains $(\zeta_n,1,\zeta_n^{-1})$, which implies that~\mbox{$(b^2,1,b^{-2})\in T$}.
Likewise, we get $(a^{-2},1,a^2)\in T$, so $(1,a^{-2},a^2)\in T$. Then
$$
(1,1,a^4b^2)=(a^2,a^2,b^2)(a^{-2},1,a^{2})(1,a^{-2},a^2)\in T,
$$
which implies that
$$
\Big(1,1,\frac{a^4}{b^2}\Big)=(1,1,a^4b^2)(b^2,1,b^{-2})(b^2,1,b^2)^{-1}\in T.
$$
Now, using the~$\mathfrak{S}_4$-action on $T$ one more time, we see that $(\frac{a^4}{b^2},1,1)\in T$ as required.
\end{proof}

\begin{corollary}
\label{corollary:P3-P3-explicit}
 There exist non-zero complex numbers $\lambda_1$,  $\lambda_2$ and  $\lambda_3$ such that $\iota G\iota=G$,
where $\iota$ is the~Cremona involution given by
$$
[x_0:x_1:x_2:x_3]\mapsto [\lambda_1 x_1x_2x_3:\lambda_2 x_0x_2x_3: \lambda_3 x_0x_1x_3:x_0x_1x_2].
$$
\end{corollary}

In the~remaining part of the~section, we classify all such groups $G$ when $n\in\{2,3\}$.
In this case, there exist precisely twelve possibilities for the~group $G$~up~to~conjugation,
which can be described as follows.
\begin{enumerate}[\normalfont(1)]
\item Let $G_{48,50}\cong \mumu^2_2\rtimes\mathfrak{A}_4\cong\mumu_2^4\rtimes\mumu_3$ be the~group generated by
$$
\big(-1,1,-1\big),\big(1,-1,-1\big),
\begin{pmatrix}
0 & 0 & 1 & 0\\
1 & 0 & 0 & 0\\
0 & 1 & 0 & 0\\
0 & 0 & 0 & 1
\end{pmatrix},
\begin{pmatrix}
0 & 1 & 0 & 0\\
1 & 0 & 0 & 0\\
0 & 0 & 0 & 1\\
0 & 0 & 1 & 0
\end{pmatrix}.
$$
\item Let $G_{48,3}\cong\mumu^2_2.\mathfrak{A}_4\cong\mumu_4^2\rtimes\mumu_3$ be the~group generated by
$$
\big(-1,1,-1\big),\big(1,-1,-1\big),
\begin{pmatrix}
0 & 0 & 1 & 0\\
1 & 0 & 0 & 0\\
0 & 1 & 0 & 0\\
0 & 0 & 0 & 1
\end{pmatrix},
\begin{pmatrix}
0 & i & 0 & 0\\
1 & 0 & 0 & 0\\
0 & 0 & 0 & -i\\
0 & 0 & 1 & 0
\end{pmatrix}.
$$

\item Let $G_{96,70}\cong\mumu_2^3\rtimes \mathfrak{A}_4\cong\mumu_2^4\rtimes\mumu_6$ be the~group generated by
$$
\big(-1,1,1\big),\big(1,-1,1\big),\big(1,1,-1\big),
\begin{pmatrix}
0 & 0 & 1 & 0\\
1 & 0 & 0 & 0\\
0 & 1 & 0 & 0\\
0 & 0 & 0 & 1
\end{pmatrix},
\begin{pmatrix}
0 & 1 & 0 & 0\\
1 & 0 & 0 & 0\\
0 & 0 & 0 & 1\\
0 & 0 & 1 & 0
\end{pmatrix}.
$$

\item Let $G_{96,72}\cong\mumu_2^3.\mathfrak{A}_4\cong\mumu_4^2\rtimes\mumu_6$  be the~group generated by
$$
\big(-1,1,1\big),\big(1,-1,1\big),\big(1,1,-1\big),
\begin{pmatrix}
0 & 0 & 1 & 0\\
1 & 0 & 0 & 0\\
0 & 1 & 0 & 0\\
0 & 0 & 0 & 1
\end{pmatrix},
\begin{pmatrix}
0 & i & 0 & 0\\
1 & 0 & 0 & 0\\
0 & 0 & 0 & i\\
0 & 0 & 1 & 0
\end{pmatrix}.
$$

\item Let $G_{96,227}\cong\mumu_2^2\rtimes\mathfrak{S}_4\cong \mumu_2^4\rtimes\mathfrak{S}_3$  be the~group generated by
$$
\big(-1,1,-1\big),\big(1,-1,-1\big),
\begin{pmatrix}
0 & 0 & 0 & 1\\
1 & 0 & 0 & 0\\
0 & 1 & 0 & 0\\
0 & 0 & 1 & 0
\end{pmatrix},
\begin{pmatrix}
0 & 1 & 0 & 0\\
1 & 0 & 0 & 0\\
0 & 0 & 1 & 0\\
0 & 0 & 0 & 1
\end{pmatrix}.
$$

\item Let $G_{96,227}^\prime\cong\mumu_2^2\rtimes\mathfrak{S}_4$  be the~group generated by
$$
\big(-1,1,-1\big),\big(1,-1,-1\big),
\begin{pmatrix}
0 & 0 & 0 & -1\\
1 & 0 & 0 & 0\\
0 & 1 & 0 & 0\\
0 & 0 & 1 & 0
\end{pmatrix},
\begin{pmatrix}
0 & 1 & 0 & 0\\
1 & 0 & 0 & 0\\
0 & 0 & -1 & 0\\
0 & 0 & 0 & 1
\end{pmatrix}.
$$

\item Let $G_{192,955}\cong\mumu_2^3\rtimes \mathfrak{S}_4\cong\mumu_2^4\rtimes\mathrm{D}_{12}$  be the~group generated by
$$
\big(-1,1,1\big),\big(1,-1,1\big),\big(1,1,-1\big),
\begin{pmatrix}
0 & 0 & 0 & 1\\
1 & 0 & 0 & 0\\
0 & 1 & 0 & 0\\
0 & 0 & 1 & 0
\end{pmatrix},
\begin{pmatrix}
0 & 1 & 0 & 0\\
1 & 0 & 0 & 0\\
0 & 0 & 1 & 0\\
0 & 0 & 0 & 1
\end{pmatrix}.
$$

\item Let $G_{192,185}\cong\mumu_2^3.\mathfrak{S}_4\cong\mumu_4^2\rtimes(\mumu_3\rtimes\mumu_4)$ be the~subgroup generated by
$$
\big(-1,1,1\big),\big(1,-1,1\big),\big(1,1,-1\big),
\begin{pmatrix}
0 & 0 & 0 & i\\
1 & 0 & 0 & 0\\
0 & 1 & 0 & 0\\
0 & 0 & 1 & 0
\end{pmatrix},
\begin{pmatrix}
0 & 1 & 0 & 0\\
1 & 0 & 0 & 0\\
0 & 0 & i & 0\\
0 & 0 & 0 & 1
\end{pmatrix}.
$$

\item Let $G_{324,160}\cong\mumu_3^3\rtimes\mathfrak{A}_4$  be the~group generated by
$$
\big(\zeta_3,1,1\big),\big(1,\zeta_3,1\big),\big(1,1,\zeta_3\big),
\begin{pmatrix}
0 & 0 & 1 & 0\\
1 & 0 & 0 & 0\\
0 & 1 & 0 & 0\\
0 & 0 & 0 & 1
\end{pmatrix},
\begin{pmatrix}
0 & 1 & 0 & 0\\
1 & 0 & 0 & 0\\
0 & 0 & 0 & 1\\
0 & 0 & 1 & 0
\end{pmatrix}.
$$

\item Let $G_{324,160}^\prime\cong\mumu_3^3\rtimes\mathfrak{A}_4$  be the~group generated by
$$
\big(\zeta_3,1,1\big),\big(1,\zeta_3,1\big),\big(1,1,\zeta_3\big),
\begin{pmatrix}
0 & 0 & 1 & 0\\
1 & 0 & 0 & 0\\
0 & 1 & 0 & 0\\
0 & 0 & 0 & 1
\end{pmatrix},
\begin{pmatrix}
0 & -1 & 0 & 0\\
1 & 0 & 0 & 0\\
0 & 0 & 0 & -1\\
0 & 0 & 1 & 0
\end{pmatrix}.
$$
\item Let $G_{648,704}\cong\mumu_3^3\rtimes\mathfrak{S}_4$ be the~group generated by
$$
\big(\zeta_3,1,1\big),\big(1,\zeta_3,1\big),\big(1,1,\zeta_3\big),
\begin{pmatrix}
0 & 0 & 0 & 1\\
1 & 0 & 0 & 0\\
0 & 1 & 0 & 0\\
0 & 0 & 1 & 0
\end{pmatrix},
\begin{pmatrix}
0 & 1 & 0 & 0\\
1 & 0 & 0 & 0\\
0 & 0 & 1 & 0\\
0 & 0 & 0 & 1
\end{pmatrix}.
$$
\item Let $G_{648,704}^\prime\cong\mumu_3^3\rtimes\mathfrak{S}_4$ be the~group generated by
$$
\big(\zeta_3,1,1\big),\big(1,\zeta_3,1\big),\big(1,1,\zeta_3\big),
\begin{pmatrix}
0 & 0 & 0 & 1\\
-1 & 0 & 0 & 0\\
0 & 1 & 0 & 0\\
0 & 0 & 1 & 0
\end{pmatrix},
\begin{pmatrix}
0 & 1 & 0 & 0\\
1 & 0 & 0 & 0\\
0 & 0 & 1 & 0\\
0 & 0 & 0 & -1
\end{pmatrix}.
$$
\end{enumerate}

We used Magma \cite{Magma} to identify the GAP ID's of these~groups.
For instance, to identify the~group $G_{648,704}$, we used the~following Magma code provided to us by Tim Dokchitser:
\begin{verbatim}
    K:=CyclotomicField(3);
    R<x>:=PolynomialRing(K);
    w:=Roots(x^2+x+1,K)[1,1];
    S:=[[[0,1,0,0],[1,0,0,0],[0,0,1,0],[0,0,0,1]],
        [[0,0,0,1],[1,0,0,0],[0,1,0,0],[0,0,1,0]],
        [[w,0,0,0],[0,1,0,0],[0,0,1,0],[0,0,0,1]],
        [[1,0,0,0],[0,w,0,0],[0,0,1,0],[0,0,0,1]],
        [[1,0,0,0],[0,1,0,0],[0,0,w,0],[0,0,0,1]]];
    G:=sub<GL(4,K)|[GL(4,K)|M: M in S]>;
    D:=[M: M in Center(G) | IsScalar(M)];
    GP:=quo<G|D>;
    IdentifyGroup(GP);
\end{verbatim}

We want to show that if $n\in\{2,3\}$, then $G$ is conjugated to a subgroup among
$G_{48,50}$, $G_{48,3}$, $G_{96,70}$, $G_{96,72}$, $G_{96,227}$, $G_{96,227}^\prime$,
$G_{192,955}$, $G_{192,185}$, $G_{324,160} $, $G_{324,160}^\prime$,~$G_{648,704}$,~$G_{648,704}^\prime$.

\begin{lemma}
\label{lemma:48}
Suppose $n=2$, $T\cong\mumu_2^2$ and $\mathrm{im}(\upsilon)=\mathfrak{A}_4$.
Then the~subgroup $G$ is~conjugated to one of the~subgroups $G_{48,50}$ or $G_{48,3}$.
\end{lemma}

\begin{proof}
Arguing as in the~proof of Lemma~\ref{lemma:A4}, we see that $T=\langle(-1,1,-1),(1,-1,-1)\rangle$.
Let $A$ and $B$ be some elements in the~group $G$ such that $\upsilon(A)=\rho$ and $\upsilon(B)=\varsigma$.
Then
$$
A=
\begin{pmatrix}
0 & 0 & a_3 & 0\\
a_1 & 0 & 0 & 0\\
0 & a_2 & 0 & 0\\
0 & 0 & 0 & 1
\end{pmatrix}\ \text{and}\
B=
\begin{pmatrix}
0 & b_2 & 0 & 0\\
b_1 & 0 & 0 & 0\\
0 & 0 & 0 & 1\\
0 & 0 & b_3 & 0
\end{pmatrix},
$$
where all $a_i$ and $b_j$ are some non-zero complex numbers.
Conjugating $G$ by an~appropriate element of the~torus $\mathbb{T}$, we may assume that $a_1=a_2=1$ and $b_3=b_1$.
Then $a_2=\pm 1$, because $A^3\in T$.
Replacing $A$ by $(1,-1,-1)A$ if necessary, we may assume that $a_2=1$.

Since $B^2\in T$ and $(AB)^3\in T$, we get $b_2=\pm 1$ and $b_1^2=b_2^3$.
If $b_2=1$, then $b_1=b_3=\pm 1$, which gives $G=G_{48,50}$.
Likewise, if $b_2=-1$, then $b_1=b_3=\pm i$, hence $G=G_{48,3}$.
\end{proof}

\begin{lemma}
\label{lemma:96}
Suppose $n=2$, $T=\mumu_2^3$ and $\mathrm{im}(\upsilon)=\mathfrak{A}_4$.
Then the~subgroup $G$ is~conjugated to one of the~subgroups $G_{96,70}$ or $G_{96,72}$.
\end{lemma}

\begin{proof}
The proof is essentially the~same as the~proof of Lemma~\ref{lemma:48}.
\end{proof}

\begin{lemma}
\label{lemma:96-2}
Suppose $n=2$, $T\cong\mumu_2^2$ and $\mathrm{im}(\upsilon)=\mathfrak{S}_4$.
Then the~group $G$ is~conjugated to one of the~subgroups $G_{96,227}$ or $G_{96,227}^\prime$.
\end{lemma}

\begin{proof}
Arguing as in the~proof of Lemma~\ref{lemma:A4}, we see that $T=\langle(-1,1,-1),(1,-1,-1)\rangle$.
Let $A$ and $B$ be some elements in $G$ such that $\upsilon(A)=\tau$ and $\upsilon(B)=\sigma$.
Then, arguing as in the~proof of Corollary~\ref{corollary:odd}, we can assume that
$$
A=
\begin{pmatrix}
0 & 0 & 0 & 1\\
1 & 0 & 0 & 0\\
0 & 1 & 0 & 0\\
0 & 0 & 1 & 0
\end{pmatrix}
\text{and}\
B=
\begin{pmatrix}
0 & b_2 & 0 & 0\\
b_1 & 0 & 0 & 0\\
0 & 0 & b_3 & 0\\
0 & 0 & 0 & 1
\end{pmatrix}
$$
for some non-zero complex numbers $b_1$, $b_2$ and $b_3$.
Since $B^2\in T$ and $(AB)^3\in T$, we get
$$
\begin{cases}
b_2^3=b_1b_3,\\
b_3^2=1,\\
b_1b_2=\pm 1.
\end{cases}
$$
This gives $b_2^8=1$. If $b_2$ is a primitive eighth root of unity, we get $b_1=\mp b_2^3$ and $b_3=\mp 1$, which gives $(b_2^3,b_2^2,b_2)^{-1}G(b_2^3,b_2^2,b_2)=G_{96,227}^\prime$.
If $b_2^4=1$, then $G$ is conjugate to $G_{96,227}$.
\end{proof}

The~subgroups  $G_{96,227}$ and $G_{96,227}^\prime$ are not conjugated in $\mathrm{PGL}_4(\mathbb{C})$,
because $\mathbb{P}^3$ contains three $G_{96,227}$-orbits of length $4$
and only one $G_{96,227}^\prime$-orbit of length $4$.

\begin{lemma}
\label{lemma:192}
Suppose $n=2$, $T=\mumu_2^3$ and $\mathrm{im}(\upsilon)=\mathfrak{S}_4$.
Then the~group $G$ is~conjugated to one of the~subgroups $G_{192,955}$ or $G_{192,185}$.
\end{lemma}

\begin{proof}
Arguing as in the~proof of Lemma~\ref{lemma:96-2}, we may assume that
$$
G=\langle(-1,1,1),(1,-1,1),(1,1,-1),A,B\rangle,
$$
where
$$
A=
\begin{pmatrix}
0 & 0 & 0 & 1\\
1 & 0 & 0 & 0\\
0 & 1 & 0 & 0\\
0 & 0 & 1 & 0
\end{pmatrix}
\text{and}\
B=
\begin{pmatrix}
0 & b_2 & 0 & 0\\
b_1 & 0 & 0 & 0\\
0 & 0 & b_3 & 0\\
0 & 0 & 0 & 1
\end{pmatrix}
$$
for some non-zero complex numbers $b_1$, $b_2$, $b_3$ such that
$b_1b_2=\pm 1$, $b_3^2=\pm 1$, $b_2^3=\pm b_1b_3$.
This equations give us $b_2^8=1$. Now,  arguing as in the~end of the~proof of Lemma~\ref{lemma:96-2},
we~see that the~subgroup $G$ is conjugated either to $G_{192,955}$ or to $G_{192,185}$.
\end{proof}

\begin{lemma}
\label{lemma:324}
Suppose $n=3$, $T=\mumu_2^3$ and $\mathrm{im}(\upsilon)=\mathfrak{A}_4$.
Then the~group $G$ is~conjugated to one of the~subgroups $G_{324,160}$ or $G_{324,160}^\prime$.
\end{lemma}

\begin{proof}
Arguing as in the~proof of Lemma~\ref{lemma:48}, we may assume that
$$
G=\langle(\zeta_3,1,1),(1,\zeta_3,1),(1,1,\zeta_3),A,B\rangle
$$
for
$$
A=
\begin{pmatrix}
0 & 0 & 1 & 0\\
1 & 0 & 0 & 0\\
0 & r & 0 & 0\\
0 & 0 & 0 & 1
\end{pmatrix}\ \text{and}\
B=
\begin{pmatrix}
0 & a & 0 & 0\\
b & 0 & 0 & 0\\
0 & 0 & 0 & a\\
0 & 0 & 1 & 0
\end{pmatrix},
$$
where $r$, $a$ and $b$ are some non-zero complex numbers. Then
$$
A^3=\begin{pmatrix}
r & 0 & 0 & 0\\
0 & r & 0 & 0\\
0 & 0 & r & 0\\
0 & 0 & 0 & 1
\end{pmatrix},
B^2=\begin{pmatrix}
b & 0 & 0 & 0\\
0 & b & 0 & 0\\
0 & 0 & 1 & 0\\
0 & 0 & 0 & 1
\end{pmatrix},
(AB)^3=\begin{pmatrix}
1 & 0 & 0 & 0\\
0 & \frac{a^2}{rb} & 0 & 0\\
0 & 0 & 1 & 0\\
0 & 0 & 0 & 1
\end{pmatrix}.
$$
Since $A^3\in T$, $B^2\in T$, $(AB)^3\in T$, we get $r=\zeta_3^\alpha$ and $b=\zeta_3^\beta$ for some $\alpha$ and $\beta$ in $\{0,1,2\}$.
Replacing $A\mapsto(1,1,\zeta_3^{-\alpha})A$ and $B\mapsto(1,\zeta_3^{-\beta},1)B$,
we may assume that $r=1$ and $b=1$.
Then $a=\zeta_6^\gamma$ for $\gamma\in\{0,1,2,3,4,5\}$.
Replacing $B\mapsto(\zeta_3^{\delta},1,\zeta_3^{\delta})B$~for~some $\delta\in\{0,1,2\}$,
we~may assume $a\in\{\pm 1,\zeta_6\}$.
If $a=1$, then $G=G_{324,160} $.
If~$a\ne 1$,~then~$G=G_{324,160}^\prime$.
\end{proof}

It follows from Example~\ref{example:48-96-324-pencil} that
the subgroups $G_{324,160} $ and $G_{324,160}^\prime$ are not conjugate,
because $\mathbb{P}^3$ does not contain $G_{324,160} $-invariant pencils of cubic surfaces.

If $n=3$, $T=\mumu_3^3$ and $\mathrm{im}(\upsilon)=\mathfrak{S}_4$, it follows from Corollary~\ref{corollary:odd}
that $G$ is~conjugated to one of the subgroups $G_{648,704}$ or $G_{648,704}^\prime$.
Note that these subgroups are not conjugated, because the~group $G_{648,704}$ leaves invariant the~Fermat cubic surface,
but one can check that there exists no $G_{648,704}^\prime$-invariant cubic surface in $\mathbb{P}^3$.

\section{Equivariant geometry of projective space: group of order 48}
\label{section:P3-48}

Let $G$ be the~subgroup in $\mathrm{PGL}_4(\mathbb{C})$ generated by
$$
M=\begin{pmatrix}
-1 & 0 & 0 & 0\\
0 & 1 & 0 & 0\\
0 & 0 & -1 & 0\\
0 & 0 & 0 & 1
\end{pmatrix},
N=\begin{pmatrix}
-1 & 0 & 0 & 0\\
0 & -1 & 0 & 0\\
0 & 0 & 1 & 0\\
0 & 0 & 0 & 1
\end{pmatrix},
A=\begin{pmatrix}
0 & 0 & 1 & 0\\
1 & 0 & 0 & 0\\
0 & 1 & 0 & 0\\
0 & 0 & 0 & 1
\end{pmatrix},
B=\begin{pmatrix}
0 & 1 & 0 & 0\\
1 & 0 & 0 & 0\\
0 & 0 & 0 & 1\\
0 & 0 & 1 & 0
\end{pmatrix},
$$
and let $\mathbb{H}\cong\mumu_2^4$ be the~normal subgroup of the~group $G$ generated by $M$, $N$, $B$, $ABA^2$.
Then~$G$ is the~subgroup $G_{48,50}\cong\mumu_2^4\rtimes\mumu_3$ that has been introduced in Section~\ref{section:subgroups}.

\begin{remark}
\label{remark:48-Tim-subgroups}
The subgroup lattice of $G$ is described in \cite{Tim}. Let us present this description.
The subgroup $\mathbb{H}$ is the~unique subgroup in $G$ that is isomorphic to~$\mumu_2^4$. It is normal.
Similarly, the~group $G$ contains $16$ subgroups that are isomorphic to $\mumu_3$, which are all conjugated by the Sylow theorem.
Up to conjugation, the~group $G$ contains $5$ subgroups isomorphic to $\mumu_2$, which are all contained in the~subgroup $\mathbb{H}$, hence $15$ subgroups in total.
Finally, up to conjugation, the~group $G$ contains exactly $5$, $5$, $15$ subgroups that are isomorphic to $\mathfrak{A}_4$, $\mumu_2^3$, $\mumu_2^2$, respectively.
Their generators can be described as follows:
\begin{center}
\renewcommand\arraystretch{1.4}\hspace*{-0.6cm}
\begin{tabular}{|c||c|}
\hline
$\mathfrak{A}_4$ &  $\langle A,B,ABA^2\rangle$, $\langle A,M,N\rangle$, $\langle A,ABA^2N,BMN\rangle$, $\langle A,ABA^2M,BN\rangle$, $\langle A,ABA^2MN,BM\rangle$\\
\hline
$\mumu_2^3$ &  $\langle ABA^2, B , N\rangle$, $\langle ABA^2, BN, M \rangle$,  $\langle ABA^2, BM, MN \rangle$, $\langle ABA^2, BM, N\rangle$, $\langle B, M, N\rangle$\\
\hline
 & $\langle M, N \rangle$, $\langle B, N \rangle$, $\langle B, M \rangle$, $\langle B, MN \rangle$, $\langle BN, M \rangle$, $\langle BN, MN \rangle$,\\
$\mumu_2^2$& $\langle BM, N \rangle$, $\langle B,ABA^2 \rangle$, $\langle ABA^2, BMN \rangle$, $\langle ABA^2, BM \rangle$, $\langle ABA^2, BN \rangle$, \\
& $\langle ABA^2N, BM \rangle$, $\langle ABA^2N, BMN \rangle$, $\langle ABA^2M, BN \rangle$, $\langle ABA^2MN, BM \rangle$\\
\hline
\end{tabular}
\end{center}
\end{remark}

Now, let us the action of $G$ on smooth curves of small genus.

\begin{lemma}[{\cite{Breuer,LMFDB}}]
\label{lemma:48-curves}
Let $C$ be a smooth curve of genus $g\leqslant 19$.
\begin{enumerate}[\normalfont(1)]
\item If $\mathbb{H}$ acts faithfully on $C$, then $g\geqslant 5$ and $C$ is not hyperelliptic.
\item Suppose that $G$ acts faithfully on $C$. Then the~$G$-orbits in $C$ are of lengths $16$, $24$,~$48$.
Let $a_{16}$ and $a_{24}$ be the~number of $G$-orbits in $C$ of length $16$ and $24$, respectively.
Then~$g\in\{9,13,17\}$, and the~possible values of $a_{16}$ and $a_{24}$ are given in the~table
\begin{center}
\renewcommand\arraystretch{1.4}
\begin{tabular}{|c||c|c|c|c|c|c|c|}
\hline
$g$      & $9$ & $13$ &$13$ &$13$ &$17$ &$17$&$17$\\
\hline
\hline
$a_{16}$ & $2$ & $0$  & $0$& $3$ & $1$ & $1$ & $4$\\
\hline
$a_{24}$ & $2$ & $1$  & $5$& $1$ & $0$ & $4$ & $0$\\
\hline
\end{tabular}
\end{center}
\end{enumerate}
\end{lemma}

\begin{proof}
By \cite[Lemma 2.3]{CheltsovShramov2017}, the~group $\mathbb{H}$ cannot act faithfully on rational or elliptic curve.
Moreover, if $\mathbb{H}$ acts faithfully on $C$ and the curve $C$ is hyperelliptic, then the~canonical morphism
$C\to\mathbb{P}^1$ is $\mathbb{H}$-equivariant, which is impossible, since neither $\mumu^4$ nor $\mumu_2^2$ can act faithfully on a~rational curve.
Thus, assertion (1) follows from \cite{KuribayashiKuribayashi}.

Suppose $G$ acts faithfully on $C$. Then $g\geqslant 5$ by (1),
and $g\ne 5$ by \cite[Proposition 3]{KuribayashiKimura}, since $G$ does not contain elements of order $4$.
Thus, we conclude that $g>5$.

By Remark~\ref{remark:48-Tim-subgroups}, the~$G$-orbits in $C$ are of lengths $16$, $24$,~$48$,
because the~stabilizer in the~group $G$ of a point in $C$ is cyclic.
Let $\widehat{C}= C/G$, and let $\hat{g}$ be the~genus of the~curve~$\widehat{C}$.
Then $2g-2=48(2\hat{g}-2)+32a_{16}+24a_{24}$ by the Hurwitz's formula.
This implies (2).
\end{proof}

Let $\mathcal{Q}_1=\{x_0^2+x_1^2+x_2^2+x_3^2=0\}$. Then $\mathcal{Q}_1$ is the~unique $G$-invariant quadric in $\mathbb{P}^3$.~Let
\begin{align*}
\mathcal{Q}_2&=\big\{x_0^2+x_1^2=x_2^2+x_3^2\big\}, &\mathcal{Q}_3&=\big\{x_0^2-x_1^2=x_2^2-x_3^2\big\}, &\mathcal{Q}_4&=\big\{x_0^2-x_1^2=x_3^2-x_2^2\big\},\\
\mathcal{Q}_5&=\big\{x_0x_2+x_1x_3=0\big\}, &\mathcal{Q}_6&=\big\{x_0x_3+x_1x_2=0\big\}, &\mathcal{Q}_7&=\big\{x_0x_1+x_2x_3=0\big\},\\
\mathcal{Q}_8&=\big\{x_0x_2=x_1x_3\big\}, &\mathcal{Q}_9&=\big\{x_0x_3=x_1x_2\big\}, &\mathcal{Q}_9&=\big\{x_0x_1=x_2x_3\big\}.
\end{align*}
Then $\mathcal{Q}_1$, $\mathcal{Q}_2$, $\mathcal{Q}_3$, $\mathcal{Q}_4$, $\mathcal{Q}_5$, $\mathcal{Q}_6$, $\mathcal{Q}_7$,
$\mathcal{Q}_8$, $\mathcal{Q}_9$, $\mathcal{Q}_{10}$ are all $\mathbb{H}$-invariant quadric surfaces in~$\mathbb{P}^3$.
Observe that these quadric surfaces are smooth, and $\mathbb{H}$ acts faithfully on each of them.
These are the ten \emph{fundamental quadrics} in \cite{Hudson}.

\begin{lemma}
\label{lemma:alpha-quadric-Heisenberg}
Let $S$ be an $\mathbb{H}$-invariant quadric surface in $\mathbb{P}^3$. Then $\alpha_{\mathbb{H}}(S)=1$.
\end{lemma}

\begin{proof}
Fix an isomorphism $S\cong\mathbb{P}^1\times\mathbb{P}^1$. Observe that $S$ does not have $\mathbb{H}$-fixed points,
and the~surface $S$ does not contain $\mathbb{H}$-invariant curves of degree $(1,0)$, $(0,1)$ or $(1,1)$.
Indeed, this follows from the fact that $\mathbb{P}^3$ does not contain $\mathbb{H}$-fixed points,
and it contains neither $\mathbb{H}$-invariant lines nor $\mathbb{H}$-invariant planes.

Note that $|-K_{S}|$ has $\mathbb{H}$-invariant curves, these are the restrictions of other $\mathbb{H}$-invariant
quadric surfaces in $\mathbb{P}^3$ on $S$. This shows that $\alpha_{\mathbb{H}}(S)\leqslant 1$.

Suppose that $\alpha_{\mathbb{H}}(S)<1$.
Then $S$ contains an~$\mathbb{H}$-invariant effective $\mathbb{Q}$-divisor $D$ such that $D\sim_{\mathbb{Q}}-K_S$, and
$(S,\lambda D)$ is not log canonical for some rational number \mbox{$\lambda<1$}.
Since~the surface $S$ does not contains $\mathbb{H}$-invariant curves of degree $(1,0)$, $(0,1)$ or $(1,1)$,
the~locus $\mathrm{Nklt}(S,\lambda D)$ is zero-dimensional.
Applying the~Koll\'ar--Shokurov connectedness theorem  \cite[Corollary~5.49]{KoMo98},
we see that $\mathrm{Nklt}(S,\lambda D)$ is a point, which must be $\mathbb{H}$-fixed.
But $S$ does not contain $\mathbb{H}$-fixed points. Contradiction.
\end{proof}

Observe also that $G$ acts naturally on the~set
$$
\big\{\mathcal{Q}_1,\mathcal{Q}_2,\mathcal{Q}_3,\mathcal{Q}_4,\mathcal{Q}_5,\mathcal{Q}_6,\mathcal{Q}_7,\mathcal{Q}_8,\mathcal{Q}_9,\mathcal{Q}_{10}\big\},
$$
and it splits this set into four $G$-orbits:
$\{\mathcal{Q}_1\}$, $\{\mathcal{Q}_2,\mathcal{Q}_3,\mathcal{Q}_4\}$, $\{\mathcal{Q}_5,\mathcal{Q}_6,\mathcal{Q}_7\}$, $\{\mathcal{Q}_8,\mathcal{Q}_9,\mathcal{Q}_{10}\}$.

\begin{remark}
\label{remark:48-30-lines}
Any two distinct $\mathbb{H}$-invariant quadrics in $\mathbb{P}^3$ intersect by a quadruple of lines.
By \cite[Lemma~2.17]{CheltsovShramov2017}, this gives $30$ lines, which can be characterized as follows: for every line among these $30$ lines,
there is an~element $g\in\mathbb{H}$ such that $g$ pointwise fixes this~line.
These lines contains all $G$-orbits of length $24$. See Remark~\ref{remark:48-orbits-24} for more details.
\end{remark}

Let us describe $G$-orbits in $\mathbb{P}^3$.
All $G$-orbits of length~$24$ are described in Remark~\ref{remark:48-30-lines}.
To describe the remaining $G$-orbits in $\mathbb{P}^3$, we let
\begin{align*}
\Sigma_{4}&=\mathrm{Orb}_G\big([1:0:0:0]\big),\\
\Sigma_{4}^{\prime}&=\mathrm{Orb}_G\big([1:1:1:-1]\big),\\
\Sigma_{4}^{\prime\prime}&=\mathrm{Orb}_G\big([1:1:1:1]\big),\\
\Sigma_{12}&=\mathrm{Orb}_G\big([0:0:1:1]\big),\\
\Sigma_{12}^\prime&=\mathrm{Orb}_G\big([0:0:i:1]\big),\\
\Sigma_{12}^{\prime\prime}&=\mathrm{Orb}_G\big([i:i:1:1]\big),\\
\Sigma_{12}^{\prime\prime\prime}&=\mathrm{Orb}_G\big([-i:i:1:1]\big),\\
\Sigma_{16}&=\mathrm{Orb}_G\big([-1+\sqrt{3}i:-1-\sqrt{3}i:2:0]\big),\\
\Sigma_{16}^\prime&=\mathrm{Orb}_G\big([-1-\sqrt{3}i:-1+\sqrt{3}i:2:0]\big),\\
\Sigma_{16}^t&=\mathrm{Orb}_G\big([1:1:1:t]\big)\ \text{for}\ t\in\mathbb{C}\setminus\{\pm 1\}.
\end{align*}
Then $\Sigma_{4}$, $\Sigma_{4}^{\prime}$, $\Sigma_{4}^{\prime\prime}$, $\Sigma_{12}$, $\Sigma_{12}^\prime$,
$\Sigma_{12}^{\prime\prime}$, $\Sigma_{12}^{\prime\prime\prime}$, $\Sigma_{16}$, $\Sigma_{16}^\prime$
are $G$-orbits of length $4$, $4$, $4$, $12$, $12$, $12$, $12$, $16$, $16$, respectively.
Similarly, $\Sigma_{16}^t$ is a $G$-orbit of length $16$ for every $t\in\mathbb{C}\setminus\{\pm 1\}$.

\begin{lemma}
\label{lemma:48-orbits}
Let $\Sigma$ be a $G$-orbit in $\mathbb{P}^3$ such that $|\Sigma|<24$.
Then
\begin{itemize}
\item either $\Sigma$ is one of the~$G$-orbits $\Sigma_{4}$, $\Sigma_{4}^{\prime}$, $\Sigma_{4}^{\prime\prime}$, $\Sigma_{12}$, $\Sigma_{12}^\prime$,
$\Sigma_{12}^{\prime\prime}$, $\Sigma_{12}^{\prime\prime\prime}$, $\Sigma_{16}$, $\Sigma_{16}^\prime$,
\item or $\Sigma=\Sigma_{16}^t$ for some $t\in\mathbb{C}\setminus\{\pm 1\}$.
\end{itemize}
\end{lemma}

\begin{proof}
Let $\Gamma$ be the~stabilizer of a point in $\Sigma$. Then $|\Gamma|>2$.
But $\mathbb{P}^3$ has no $\mathbb{H}$-fixed~points,
so that $\Gamma$ is isomorphic to $\mathfrak{A}_4$, $\mumu_2^3$, $\mumu_2^2$ or $\mumu_3$ by Remark~\ref{remark:48-Tim-subgroups}.
Then $|\Sigma|\in\{4,6,12,16\}$.

If $\Gamma\cong\mumu_2^3$, then $\Gamma\subset\mathbb{H}$, hence $\mathbb{P}^3$ contains an $\mathbb{H}$-orbit of length $2$,
which is impossible, since $\mathbb{P}^3$ does not contains $\mathbb{H}$-invariant lines.
Hence, $\Gamma$ is isomorphic to one of the~following three groups: $\mathfrak{A}_4$, $\mumu_2^2$, $\mumu_3$.
Then $|\Sigma|\in\{4,12,16\}$.

Suppose that $\Gamma\cong\mumu_3$.
By Remark~\ref{remark:48-Tim-subgroups}, we may assume that $\Gamma=\langle A\rangle$.
Then the~$\Gamma$-fixed points in $\mathbb{P}^3$ are the following
$$
\big[1-\sqrt{3}i:1+\sqrt{3}i:2:0\big],\big[1+\sqrt{3}i:1-\sqrt{3}i:2:0\big], [0:0:0:1], [1:1:1:t]
$$
for any $t\in\mathbb{C}$.
Since the~stabilizers of the~points $[0:0:0:1]$, $[1:1:1:1]$, $[1:1:1:-1]$ are larger than $\Gamma$,
either $\Sigma$ is one of the~orbits $\Sigma_{16}$, $\Sigma_{16}^\prime$,
or $\Sigma=\Sigma_{16}^t$ for some $t\in\mathbb{C}\setminus\{\pm 1\}$.

Now, suppose that $\Gamma\cong\mathfrak{A}_4$.
By Remark~\ref{remark:48-Tim-subgroups}, the~group $G$ contains exactly five subgroups isomorphic to $\mathfrak{A}_4$ up to conjugation.
Three of these groups are $\langle A,B\rangle$, $\langle A,NB\rangle$, $\langle MA,B\rangle$.
If $\Gamma$ is one of these subgroups, then $\Sigma$ is one of the~$G$-orbits $\Sigma_{4}$, $\Sigma_{4}^{\prime}$, $\Sigma_{4}^{\prime\prime}$, respectively.
The~remaining subgroups conjugated to $\mathfrak{A}_4$ are the~groups  $\langle A,MB\rangle$ and $\langle ABA,BNM\rangle$.
One can check that both of them do not have fixed points in $\mathbb{P}^3$.

Finally, we suppose that $\Gamma\cong\mumu_2^2$.
By Remark~\ref{remark:48-Tim-subgroups},  the~group $G$ contains $15$ subgroups isomorphic to $\mumu_2^2$ up to conjugation.
Five subgroups among them are normal --- they are contained in the~subgroups of $G$ isomorphic to  $\mathfrak{A}_4$.
The fixed points of three of them are contained in the~subset $\Sigma_{4}\cup\Sigma_{4}^{\prime}\cup\Sigma_{4}^{\prime\prime}$,
and the~remaining two normal subgroups do not fix any point in  $\mathbb{P}^3$ ---
they leave invariant rulings of the~quadric $\mathcal{Q}_1\cong\mathbb{P}^1\times\mathbb{P}^1$.

To complete the~proof, we may assume that $\Gamma$ is not a normal subgroup of the~group~$G$.
Up to conjugation, there are ten such subgroups in $G$ by Remark~\ref{remark:48-Tim-subgroups}.
Four of them fix a~point in the~subset $\Sigma_{12}\cup\Sigma_{12}^\prime\cup\Sigma_{12}^{\prime\prime}\cup\Sigma_{12}^{\prime\prime\prime}$.
Up to conjugation, these are the~subgroups
$$
\big\langle B,N\big\rangle, \big\langle BM,N\big\rangle, \big\langle ABA^2,BMN\big\rangle, \big\langle ABA^2M, BN\big\rangle
$$
respectively. If $\Gamma$ is one of them, then $\Sigma$ is one of the~$G$-orbits $\Sigma_{12}$, $\Sigma_{12}^\prime$, $\Sigma_{12}^{\prime\prime}$, $\Sigma_{12}^{\prime\prime\prime}$.

The remaining $6$ subgroups in $G$ that are isomorphic to $\mumu_2^2$ are described in Remark~\ref{remark:48-Tim-subgroups}.
For instance, take the~subgroup $\langle B,MN\rangle\cong\mumu_2^2$.
This group does not fix any point in~$\mathbb{P}^3$, but this subgroup leaves invariant rulings of the~quadric $\mathcal{Q}_8\cong\mathbb{P}^1\times\mathbb{P}^1$.
To be precise, for every $[a:b]\in\mathbb{P}^1$, the~group $\langle B,MN\rangle$ leaves invariant the~line
$$
\big\{ax_0+bx_3=ax_1+bx_2=0\big\}\subset\mathcal{Q}_8.
$$
Moreover, these are all  $\langle B,MN\rangle$-invariant lines in $\mathbb{P}^3$.
Similarly, one can also check that each of the~remaining non-normal subgroups in $G$ isomorphic to $\mumu_2^2$ fixes no point in $\mathbb{P}^3$,
but it leaves invariant infinitely many lines that are contained in one of the~$\mathbb{H}$-invariant quadrics
$\mathcal{Q}_2$, $\mathcal{Q}_3$, $\mathcal{Q}_{4}$, $\mathcal{Q}_5$, $\mathcal{Q}_6$, $\mathcal{Q}_{7}$,
$\mathcal{Q}_8$, $\mathcal{Q}_9$, $\mathcal{Q}_{10}$.
This completes the~proof of the~lemma.
\end{proof}

Let us describe the normalizer of the subgroup $G$ in the group $\mathrm{PGL}_4(\mathbb{C})$.
To start with, recall from Section~\ref{section:subgroups} that $G\triangleleft G_{96,227}\cong \mumu_2^4\rtimes\mathfrak{S}_3$, where $G_{96,227}$ is generated by
$$
M,N,A^\prime=\begin{pmatrix}
0 & 0 & 0 & 1\\
1 & 0 & 0 & 0\\
0 & 1 & 0 & 0\\
0 & 0 & 1 & 0
\end{pmatrix},
B^\prime=\begin{pmatrix}
0 & 1 & 0 & 0\\
1 & 0 & 0 & 0\\
0 & 0 & 1 & 0\\
0 & 0 & 0 & 1
\end{pmatrix}.
$$
Similarly, we have $G\triangleleft G_{96,70}$ and $G\triangleleft G_{192,955}$, where $G_{96,70}\cong\mumu_2^4\rtimes\mumu_6$  is generated by
$$
M,N,A,B,L=
\begin{pmatrix}
-1 & 0 & 0 & 0\\
0 & 1 & 0 & 0\\
0 & 0 & 1 & 0\\
0 & 0 & 0 & 1
\end{pmatrix},
$$
and $G_{192,955}=\langle M,N,A^\prime,B^\prime,L\rangle\cong\mumu_4^2\rtimes\mathrm{D}_{12}$.
Let $G_{144,184}$ be the~subgroup generated~by
$$
M,N,A,B,
R=\frac{1}{2}\begin{pmatrix}
1 & 1 & 1 & 1\\
1 & 1 & -1 & -1\\
1 &-1 & 1 & -1\\
-1 & 1 & 1 & -1
\end{pmatrix},
$$
and let $G_{288,1025}=\langle M,N,A^\prime,B^\prime,R\rangle$.
Then $G\triangleleft G_{144,184}\cong\mathfrak{A}_4\times\mathfrak{A}_4$ and $G\triangleleft G_{288,1025}\cong\mathfrak{A}_4\wr\mumu_2$.

Let $G_{576,8654}=\langle M,N,L,A^\prime,B^\prime,R\rangle$.
Then $G_{576,8654}\cong\mumu_2^4\rtimes(\mumu_3^2\rtimes\mumu_2^2)\cong(\mathfrak{A}_4\times\mathfrak{A}_4)\rtimes\mumu_2^2$.

\begin{lemma}
\label{lemma:48-50-normalizer}
The group $G_{576,8654}$ is the normalizer of the group $G$ in $\mathrm{PGL}_4(\mathbb{C})$.
\end{lemma}

\begin{proof}
Let $\Gamma$ be the normalizer of the subgroup $G$ in $\mathrm{PGL}_4(\mathbb{C})$.
Observe that $G\triangleleft G_{576,8654}$. Thus, we have $G_{576,8654}\subset\Gamma$.
Let us show that $\Gamma\subset G_{576,8654}$.

Take any element $g\in \Gamma$.
Since $\Sigma_{4}$, $\Sigma_{4}^{\prime}$, $\Sigma_{4}^{\prime\prime}$ are the only $G$-orbits of length four in $\mathbb{P}^3$,
we~see that $g$ must permutes these $G$-orbits.
Therefore, swapping $g$ with $g\circ R$ or $g\circ R^2$, we~may assume that $\Sigma_{4}$ is $g$-invariant.
So, composing $g$ with a suitable element in $G_{192,955}$, we may assume that $g$ fixes every point in $\Sigma_4$.
Then
$$
g=\begin{pmatrix}
t_1 & 0 & 0 & 0\\
0 & t_2 & 0 & 0\\
0 & 0 & t_3 & 0\\
0 & 0 & 0 & 1
\end{pmatrix}
$$
for some non-zero complex numbers $t_1$, $t_2$, $t_3$.
Recall that $\mathcal{Q}_1$ is the unique $G$-invariant quadric surface in $\mathbb{P}^3$,
so that $\mathcal{Q}_1$ is $g$-invariant. This gives us $t_1=\pm 1$, $t_2=\pm 1$, $t_3=\pm 1$,
so that $g\in G_{192,955}\subset G_{576,8654}$. This shows that $\Gamma\subset G_{576,8654}$.

We can also argue as follows.
Let $\mathfrak{N}\subset\mathrm{PGL}_4(\mathbb{C})$ be the normalizer of the group $\mathbb{H}\cong\mumu^4$.
Then it follows from \cite[\S 123]{Blichfeldt1917} or \cite{Nieto} that there exists an~exact sequence of groups
$$
1\longrightarrow \mathbb{H}\longrightarrow \mathfrak{N}\longrightarrow \mathfrak{S}_{6}\longrightarrow 1.
$$
But $\mathbb{H}$ is a normal subgroup in $G$, $G_{96,70}$, $G_{96,227}$, $G_{192,955}$, $G_{144,184}$, $G_{288,1025}$, $G_{576,8654}$,
and the images of these groups in $\mathfrak{S}_{6}$ are isomorphic to $\mumu_3$, $\mumu_6$, $\mathfrak{S}_3$,
$\mathrm{D}_{12}$, $\mumu_3\times\mumu_3$, $\mumu_3\wr\mumu_2$,~$\mumu_3^2\rtimes\mumu_2^2$.
Using this, it is not difficult to see that $G_{576,8654}$
is the normalizer of the group $G$.
\end{proof}

Let $\widehat{\mathbb{H}}$, $\widehat{G}$, $\widehat{G}_{96,227}$ and $\widehat{G}_{144,184}$ be the~subgroups in $\mathrm{GL}_4(\mathbb{C})$ defined as follows:
\begin{align*}
\widehat{\mathbb{H}}&=\big\langle M,N,B,ABA^2\big\rangle,\\
\widehat{G}&=\big\langle M,N,A,B\big\rangle,\\
\widehat{G}_{96,227}&=\big\langle M,N,A^\prime,B^\prime\big\rangle,\\
\widehat{G}_{144,184}&=\big\langle M,N,A,B,R\big\rangle,
\end{align*}
where we consider $M$, $N$, $A$, $B$, $A^\prime$, $B^\prime$, $R$ as elements of $\mathrm{GL}_4(\mathbb{C})$.
These groups are mapped to the~groups $\mathbb{H}$, $G$, $G_{96,227}$ and $G_{144,184}$
via the~natural projection $\mathrm{GL}_4(\mathbb{C})\to\mathrm{PGL}_4(\mathbb{C})$,
and their GAP ID's are  [32,49], [96,204], [192,1493] and [288,860], respectively.

Note that the groups $\widehat{\mathbb{H}}$, $\widehat{G}$, $\widehat{G}_{96,227}$, and $\widehat{G}_{144,184}$ act naturally linearly on
$H^0(\mathbb{P}^3,\mathcal{O}_{\mathbb{P}^3}(1))$.
The corresponding linear representations are irreducible and can be identified by GAP.

\begin{lemma}
\label{lemma:48-representations}
Let $\mathbb{V}$ be the~vector subspace in $H^0(\mathbb{P}^3,\mathcal{O}_{\mathbb{P}^3}(4))$ consisting of all $\widehat{\mathbb{H}}$-invariants.
Then the~vector space $\mathbb{V}$ is five-dimensional.
Furthermore, it contains all one-dimensional subrepresentations in the vector space $H^0(\mathbb{P}^3,\mathcal{O}_{\mathbb{P}^3}(4))$
of the groups $\widehat{G}$, $\widehat{G}_{96,227}$ and $\widehat{G}_{144,184}$.
Moreover, the~following assertions hold:
\begin{enumerate}[\normalfont(i)]
\item as a~$\widehat{G}$-representation, the~vector space $\mathbb{V}$
splits as a sum of $3$ trivial representations and $2$ non-isomorphic one-dimensional non-trivial representations;
\item as a~$\widehat{G}_{96,227}$-representation the~space $\mathbb{V}$
splits as a sum of $3$ trivial representations and $1$ two-dimensional irreducible representations;
\item as a~$\widehat{G}_{144,184}$-representation the~space $\mathbb{V}$
splits as a sum of $5$ distinct non-isomorphic one-dimensional representations.
\end{enumerate}
\end{lemma}

\begin{proof}
We used GAP to verify all assertions.
\end{proof}

By Lemma~\ref{lemma:48-representations}, $\mathbb{P}^3$ contains exactly five $G_{144,184}$-invariant quartic surfaces \cite[(2.20)]{CheltsovShramov2017}.
To describe their defining equations, let
\begin{align*}
f_1&=x_0^2+x_1^2+x_2^2+x_3^2,\\
f_2&=2\big(x_0^2x_1^2+x_0^2x_2^2+x_0^2x_3^2+x_1^2x_2^2+x_1^2x_3^2+x_2^2x_3^2\big)-\big(x_0^4+x_1^4+x_2^4+x_3^4\big)+8\sqrt{3}ix_0x_1x_2x_3,\\
f_3&=2\big(x_0^2x_1^2+x_0^2x_2^2+x_0^2x_3^2+x_1^2x_2^2+x_1^2x_3^2+x_2^2x_3^2\big)-\big(x_0^4+x_1^4+x_2^4+x_3^4\big)-8\sqrt{3}ix_0x_1x_2x_3,\\
f_4&=(-1+\sqrt{3}i)\big(x_0^2x_2^2-x_0^2x_3^2-x_1^2x_2^2+x_1^2x_3^2\big)-2\big(x_0^2x_1^2-x_0^2x_2^2-x_1^2x_3^2+x_2^2x_3^2\big),\\
f_5&=(-1-\sqrt{3}i)\big(x_0^2x_2^2-x_0^2x_3^2-x_1^2x_2^2+x_1^2x_3^2\big)-2\big(x_0^2x_1^2-x_0^2x_2^2-x_1^2x_3^2+x_2^2x_3^2\big),
\end{align*}
and let $S_2=\{f_2=0\}$, $S_3=\{f_3=0\}$, $S_4=\{f_4=0\}$, $S_5=\{f_5=0\}$.
Then
\begin{itemize}
\item $2\mathcal{Q}_1$, $S_2$, $S_3$, $S_4$, $S_5$ are $G_{144,184}$-invariant quartic surfaces;
\item the surfaces $S_2$, $S_3$, $S_4$, $S_5$ are irreducible;
\item one has $\mathrm{Sing}(S_2)=\mathrm{Sing}(S_3)=\Sigma_{12}$ and $\mathrm{Sing}(S_4)=\mathrm{Sing}(S_5)=\Sigma_{4}\cup\Sigma_{4}^\prime\cup\Sigma_4^{\prime\prime}$.
\end{itemize}
By \cite[Lemma~3.12]{CheltsovShramov2017}, singularities of the surfaces $S_2$, $S_3$, $S_4$, $S_5$ are ordinary double points.

The polynomials $f_1^2$, $f_2$ and $f_3$ generate a~three-dimensional vector space
that contains all~$\widehat{G}$-invariant elements in $H^0(\mathbb{P}^3,\mathcal{O}_{\mathbb{P}^3}(4))$.
Consider the following basis of this~space:
\begin{center}
$x_0x_1x_2x_3$, $x_0^2x_1^2+x_0^2x_2^2+x_0^2x_3^2+x_1^2x_2^2+x_1^2x_3^2+x_2^2x_3^2$, $x_0^4+x_1^4+x_2^4+x_3^4$.
\end{center}
Using this basis, let us define the~net $\mathcal{M}_{4}$ consisting of quartic surfaces in $\mathbb{P}^3$ given by
\begin{equation}
\label{equation:48-net}
ax_0x_1x_2x_3+b\big(x_0^2x_1^2+x_0^2x_2^2+x_0^2x_3^2+x_1^2x_2^2+x_1^2x_3^2+x_2^2x_3^2\big)+c\big(x_0^4+x_1^4+x_2^4+x_3^4)=0,
\end{equation}
for $[a:b:c]\in\mathbb{P}^2$.
Then every surface in the~net $\mathcal{M}_{4}$ is $G$-invariant and $G_{96,227}$-invariant.
For $[a:b:c]=[1:0:0]$, we get the~surface
$$
\mathcal{T}=\{x_0x_1x_2x_3=0\}\in\mathcal{M}_{4}.
$$
Similarly, for $[a:b:c]=[-8:-2:1]$, we get another reducible surface
$$
\mathcal{T}^\prime=\{(x_0+x_1+x_2-x_3)(x_0+x_1-x_2+x_3)(x_0-x_1+x_2+x_3)(x_0-x_1-x_2-x_3)=0\}.
$$
Likewise, for $[a:b:c]=[8:-2:1]$, we get the~reducible surface
$$
\mathcal{T}^{\prime\prime}=\{(x_0+x_1+x_2+x_3)(x_0-x_1-x_2+x_3)(x_0+x_1-x_2-x_3)(x_0-x_1+x_2-x_3)=0\}.
$$
Finally, for $[a:b:c]=[0:2:1]$, we get the~non-reduced surface $2\mathcal{Q}_1\in\mathcal{M}_{4}$.

\begin{lemma}
\label{lemma:48-net}
The following assertion holds:
\begin{enumerate}[\normalfont(i)]
\item the~base locus of the~net $\mathcal{M}_{4}$ is the~set $\Sigma_{16}\cup\Sigma_{16}^\prime$,
\item the~only reducible surfaces in $\mathcal{M}_{4}$ are $\mathcal{T}$, $\mathcal{T}^\prime$, $\mathcal{T}^{\prime\prime}$, $2\mathcal{Q}_1$,
\item every irreducible surface in $\mathcal{M}_{4}$ has at most isolated ordinary double points,
\item if $S$ is a surface given by \eqref{equation:48-net}, then $S$ is singular if and only if
$$
c(b+2c)(b-2c)(a+2b-4c)(a-2b+4c)(a-6b-4c)(a+6b+4c)(a^2c+4b^3-12b^2c+16c^3)=0,
$$

\item if $S$ is an irreducible singular surface given by \eqref{equation:48-net}, then
\begin{itemize}
\item either $a^2c+4b^3-12b^2c+16c^3\ne 0$, and $\mathrm{Sing}(S)$ is described in Table~\ref{table:48-net},
\item or $a^2c+4b^3-12b^2c+16c^3=0$, and $\mathrm{Sing}(S)=\Sigma_{16}^t$ for $t\in\mathbb{C}\setminus\{\pm 1,\pm\sqrt{3}i\}$
which is uniquely determined by $[a:b:c]=[2t^3+6t:-t^2-1:1]$.
\end{itemize}
\end{enumerate}
\end{lemma}

\begin{proof}
Assertion (i) is easy to check. Assertion (ii) follows from Remark~\ref{lemma:48-orbits} and
the fact that $\mathcal{Q}_1$ is the~only $G$-invariant quadric surface in $\mathbb{P}^3$.

Assertion (iv) has been proved in \cite{Nieto}, see \cite[Proposition~3.1]{Bouyer},
\cite[Theorem~10.3.18]{Dolgachev}, \cite[Proposition 2.1]{Eklund}, \cite[Lemma~2.21]{Gonzalez-Dorrego}.

Both assertions (iii) and (v) are easy to verify.
\end{proof}

\begin{corollary}
\label{corollary:48-net}
Let $S$ be an irreducible surface in $\mathcal{M}_{4}$, and let $\pi\colon\widetilde{S}\to S$ be
its minimal resolution of singularities. Then $\widetilde{S}$ is a~K3 surface,
the~action of the~group $G$ lifts to $\widetilde{S}$, all~$G$-orbits in $\widetilde{S}$ are of length $16$, $24$ or $48$,
and $\widetilde{S}$ contains exactly $6$ orbits of length $16$.
In particular, if $S$ contains a $G$-orbit of length $<16$, then $S$ is singular at this $G$-orbit.
\end{corollary}

\begin{proof}
All assertions follow from Lemma~\ref{lemma:48-net} and explicit computations,
and the~assertion about $G$-orbits follows from \cite[Theorem~3]{Xiao}, since the~$G$-action on $\widetilde{S}$ is symplectic \cite{Hashimoto}.
\end{proof}

\begin{table}
\caption{Singular locus of an irreducible quartic surface $S\subset\mathbb{P}^3$ such that the~surface $S$ is given by \eqref{equation:48-net} with $a^2c+4b^3-12b^2c+16c^3\ne 0$.}
\label{table:48-net}
\begin{center}
\renewcommand\arraystretch{1.5}
\begin{tabular}{|c|c|c|}
  \hline
Condition on $[a:b:c]$ & Additional conditions on $[a:b:c]$ & $\mathrm{Sing}(S)$ \\
\hline
\hline
\multirow{5}{*}{ $c=0$} & $[a:b:0]\ne[\pm 6:1:0],[\pm 2:1:0]$  & $\Sigma_4$\\\cline{2-3}
                       & $[a:b:0]=[6:1:0]$  & $\Sigma_4\cup\Sigma_4^\prime$ \\\cline{2-3}
                       & $[a:b:0]=[-6:1:0]$ & $\Sigma_4\cup\Sigma_4^{\prime\prime}$\\\cline{2-3}
                       & $[a:b:0]=[2:1:0]$ & $\Sigma_4\cup\Sigma_{12}^{\prime\prime\prime}$\\\cline{2-3}
                       & $[a:b:0]=[-2:1:0]$ & $\Sigma_4\cup\Sigma_{12}^{\prime\prime}$\\
 \hline
$c\ne 0$ and $b+2c=0$ & $[a:b:c]\ne[\pm 8:-2:1]$ & $\Sigma_{12}$\\\cline{2-3}
\hline
\multirow{3}{*}{ $c\ne 0$ and $b-2c=0$} & $[a:b:c]\ne[\pm 16:2:1]$ & $\Sigma_{12}^\prime$\\\cline{2-3}
                                       & $[a:b:c]=[16:2:1]$  & $\Sigma_{12}^\prime\cup\Sigma_4^\prime$ \\\cline{2-3}
                                        & $[a:b:c]=[-16:2:1]$  & $\Sigma_{12}^\prime\cup\Sigma_4^{\prime\prime}$ \\

\hline
\multirow{2}{*}{ $c\ne 0$ and $a+2b-4c=0$}& $[a:b:c]\ne[4:0:1]$ & $\Sigma_{12}^{\prime\prime}$\\\cline{2-3}
                                       & $[a:b:c]=[4:0:1]$  & $\Sigma_{12}^{\prime\prime}\cup\Sigma_{4}^{\prime}$ \\
\hline
\multirow{2}{*}{ $c\ne 0$ and $a-2b+4c=0$}& $[a:b:c]\ne[-4:0:1]$ & $\Sigma_{12}^{\prime\prime\prime}$\\\cline{2-3}
                                       & $[a:b:c]=[-4:0:1]$  & $\Sigma_{12}^{\prime\prime\prime}\cup\Sigma_{4}^{\prime\prime}$ \\
 \hline
\multirow{2}{*}
{$c\ne 0$ and $a-6b-4c=0$}& $[a:b:c]\ne[0:-2:3],[16:2:1],[4:0:1]$   & $\Sigma_4^\prime$\\\cline{2-3}
                       & $[a:b:c]=[0:-2:3]$  & $\Sigma_4^\prime\cup\Sigma_4^{\prime\prime}$ \\
\hline
$c\ne 0$ and $a+6b+4c=0$ & $[a:b:c]\ne [-16:2:1]$,[-4:0:1]$,[0:-2:3]$ & $\Sigma_4^{\prime\prime}$\\
\hline
\end{tabular}
\end{center}
\end{table}

Now, let us describe all $G$-irreducible curves in $\mathbb{P}^3$ that are unions of at most $15$ lines.
Let $\mathcal{L}_6=\mathrm{Sing}(\mathcal{T})$,
$\mathcal{L}_6^\prime=\mathrm{Sing}(\mathcal{T}^\prime)$,
$\mathcal{L}_6^{\prime\prime}=\mathrm{Sing}(\mathcal{T}^{\prime\prime})$.
Then $\mathcal{L}_6$, $\mathcal{L}_6^\prime$, $\mathcal{L}_6^{\prime\prime}$
are $G$-irreducible curves, and each of them is a union of six lines.
We have
\begin{align*}
\Sigma_4&=\mathrm{Sing}\big(\mathcal{L}_6\big),\\
\Sigma_4^\prime&=\mathrm{Sing}\big(\mathcal{L}_6^\prime\big),\\
\Sigma_4^{\prime\prime}&=\mathrm{Sing}\big(\mathcal{L}_6^{\prime\prime}\big).
\end{align*}
Observe also that
$$
\Sigma_{12}=\mathcal{L}_6\cap\mathcal{L}_6^\prime=\mathcal{L}_6\cap\mathcal{L}_6^{\prime\prime}=\mathcal{L}_6^\prime\cap\mathcal{L}_6^{\prime\prime},
$$
so that the surfaces $\mathcal{T}$, $\mathcal{T}^\prime$, $\mathcal{T}^{\prime\prime}$ form a configuration which is known as a \emph{desmic system},
see \cite[\S~IV]{Lewis} and \cite[\S~3.19]{Lord}. Note also that
\begin{align*}
\mathcal{L}_6\cap\mathcal{Q}_1&=\Sigma_{12}^\prime,\\
\mathcal{L}_6^\prime\cap\mathcal{Q}_1&=\Sigma_{12}^{\prime\prime\prime},\\
\mathcal{L}_6^{\prime\prime}\cap\mathcal{Q}_1&=\Sigma_{12}^{\prime\prime}.
\end{align*}

Now, we let $\mathcal{L}_4$ be the~$G$-irreducible curve in $\mathbb{P}^3$ whose irreducible component is the~line
$$
\big\{2x_0+(1+\sqrt{3}i)x_2-(1-\sqrt{3}i)x_3=2x_1+(1-\sqrt{3}i)x_2+(1+\sqrt{3}i)x_3=0\big\},
$$
let $\mathcal{L}_4^\prime$ be the~$G$-irreducible curve  in $\mathbb{P}^3$ whose irreducible component is the~line
$$
\big\{2x_0+(1-\sqrt{3}i)x_2-(1+\sqrt{3}i)x_3=2x_1+(1+\sqrt{3}i)x_2+(1-\sqrt{3}i)x_3=0\big\},
$$
let $\mathcal{L}_4^{\prime\prime}$ be the~$G$-irreducible curve  in $\mathbb{P}^3$ whose irreducible component is the~line
$$
\big\{2x_0-(1-\sqrt{3}i)x_2+(1+\sqrt{3}i)x_3=2x_1+(1+\sqrt{3}i)x_2+(1-\sqrt{3}i)x_3=0\big\},
$$
let $\mathcal{L}_4^{\prime\prime\prime}$ be the~$G$-irreducible curve  in $\mathbb{P}^3$ whose irreducible component is the~line
$$
\big\{2x_0-(1+\sqrt{3}i)x_2+(1-\sqrt{3}i)x_3=2x_1+(1-\sqrt{3}i)x_2+(1+\sqrt{3}i)x_3=0\big\},
$$
let $\mathcal{L}_6^{\prime\prime\prime}$ be the~$G$-irreducible curve  in $\mathbb{P}^3$ whose irreducible component is the~line
$$
\big\{x_0+ix_2=x_1+ix_3=0\big\},
$$
and let $\mathcal{L}_6^{\prime\prime\prime\prime}$ be the~$G$-irreducible curve in  in $\mathbb{P}^3$ whose irreducible component is the~line
$$
\big\{x_0+ix_3=x_1+ix_2=0\big\}.
$$
Then $\mathcal{L}_4$, $\mathcal{L}_4^\prime$, $\mathcal{L}_4^{\prime\prime}$, $\mathcal{L}_4^{\prime\prime\prime}$
consist of $4$ disjoint lines,
$\mathcal{L}_6^{\prime\prime\prime}$ and $\mathcal{L}_6^{\prime\prime\prime\prime}$ consist of $6$ disjoint lines,
and all these six $G$-irreducible reducible curves are contained in the~quadric $\mathcal{Q}_1$.

\begin{remark}
\label{remark:48-orbits-24}
Since $\mathbb{H}$ contains all elements of order $2$ in~$G$,
it follows from \cite[\S~2]{CheltsovShramov2017} that all $G$-orbits of length $24$ in the~space $\mathbb{P}^3$
are contained in the~union
$\mathcal{L}_6\cup\mathcal{L}_6^{\prime}\cup\mathcal{L}_6^{\prime\prime}\cup\mathcal{L}_6^{\prime\prime\prime}\cup\mathcal{L}_6^{\prime\prime\prime\prime}$.
Vice versa, if $P$ is a point in $\mathcal{L}_6\cup\mathcal{L}_6^{\prime}\cup\mathcal{L}_6^{\prime\prime}\cup\mathcal{L}_6^{\prime\prime\prime}\cup\mathcal{L}_6^{\prime\prime\prime\prime}$, then either its $G$-orbit has length~$24$,
or the~point $P$ is contained in the~union $\Sigma_{4}\cup\Sigma_{4}^{\prime}\cup\Sigma_{4}^{\prime\prime}\cup\Sigma_{12}\cup\Sigma_{12}^{\prime}\cup\Sigma_{12}^{\prime\prime}$.
\end{remark}

Let us present the intersections of the~curves $\mathcal{L}_4$, $\mathcal{L}_4^\prime$,
$\mathcal{L}_4^{\prime\prime}$, $\mathcal{L}_4^{\prime\prime\prime}$,
$\mathcal{L}_6$, $\mathcal{L}_6^{\prime}$, $\mathcal{L}_6^{\prime\prime}$, $\mathcal{L}_6^{\prime\prime\prime}$,
$\mathcal{L}_6^{\prime\prime\prime\prime}$.
\begin{center}
\renewcommand\arraystretch{1.4}
\begin{tabular}{|c||c|c|c|c|c|c|}
  \hline
$\cap$ & $\mathcal{L}_4$ & $\mathcal{L}_4^\prime$ & $\mathcal{L}_4^{\prime\prime}$ & $\mathcal{L}_4^{\prime\prime\prime}$ & $\mathcal{L}_6^{\prime\prime\prime}$ & $\mathcal{L}_6^{\prime\prime\prime\prime}$  \\
\hline
\hline
$\mathcal{L}_4$ & $\mathcal{L}_4$ & $\varnothing$ & $\Sigma_{16}^{\sqrt{3}i}$& $\Sigma_{16}^\prime$& $\varnothing$& $\mathcal{L}_4\cap\mathcal{L}_6^{\prime\prime\prime\prime}$ \\
\hline
$\mathcal{L}_4^\prime$ & $\varnothing$& $\mathcal{L}_4^\prime$ & $\Sigma_{16}$& $\Sigma_{16}^{-\sqrt{3}i}$& $\varnothing$& $\mathcal{L}_4^\prime\cap\mathcal{L}_6^{\prime\prime\prime\prime}$ \\
\hline
$\mathcal{L}_4^{\prime\prime}$ & $\Sigma_{16}^{\sqrt{3}i}$& $\Sigma_{16}$& $\mathcal{L}_4^{\prime\prime}$ & $\varnothing$ & $\mathcal{L}_4^{\prime\prime}\cap\mathcal{L}_6^{\prime\prime\prime}$ & $\varnothing$\\
\hline
$\mathcal{L}_4^{\prime\prime\prime}$ & $\Sigma_{16}^\prime$ &$\Sigma_{16}^{-\sqrt{3}i}$ & $\varnothing$ &$\mathcal{L}_4^{\prime\prime\prime}$ & $\mathcal{L}_4^{\prime\prime\prime}\cap\mathcal{L}_6^{\prime\prime\prime}$& $\varnothing$\\
\hline
$\mathcal{L}_6^{\prime\prime\prime}$ & $\varnothing$& $\varnothing$&  $\mathcal{L}_4^{\prime\prime}\cap\mathcal{L}_6^{\prime\prime\prime}$& $\mathcal{L}_4^{\prime\prime\prime}\cap\mathcal{L}_6^{\prime\prime\prime}$& $\mathcal{L}_6^{\prime\prime\prime}$ & $\Sigma_{12}^\prime\cup\Sigma_{12}^{\prime\prime}\cup\Sigma_{12}^{\prime\prime\prime}$\\
\hline
$\mathcal{L}_6^{\prime\prime\prime\prime}$ & $\mathcal{L}_4\cap\mathcal{L}_6^{\prime\prime\prime\prime}$ & $\mathcal{L}_4^\prime\cap\mathcal{L}_6^{\prime\prime\prime\prime}$ & $\varnothing$& $\varnothing$& $\Sigma_{12}^\prime\cup\Sigma_{12}^{\prime\prime}\cup\Sigma_{12}^{\prime\prime\prime}$ & $\mathcal{L}_6^{\prime\prime\prime\prime}$\\
\hline
\end{tabular}
\end{center}
\noindent
where the~intersections $\mathcal{L}_4\cap\mathcal{L}_6^{\prime\prime\prime\prime}$, $\mathcal{L}_4^\prime\cap\mathcal{L}_6^{\prime\prime\prime\prime}$,
$\mathcal{L}_4^{\prime\prime}\cap\mathcal{L}_6^{\prime\prime\prime}$, $\mathcal{L}_4^{\prime\prime\prime}\cap\mathcal{L}_6^{\prime\prime\prime}$
are $G$-orbits of length $24$.

\begin{lemma}
\label{label:48-lines}
Let $C$ be a $G$-irreducible curve in $\mathbb{P}^3$ such that $C$ is a union of $d\leqslant 15$~lines.
Then either $C$ is one of the~curves $\mathcal{L}_4$, $\mathcal{L}_4^\prime$,
$\mathcal{L}_4^{\prime\prime}$, $\mathcal{L}_4^{\prime\prime\prime}$,
$\mathcal{L}_6$, $\mathcal{L}_6^{\prime}$, $\mathcal{L}_6^{\prime\prime}$, $\mathcal{L}_6^{\prime\prime\prime}$,
$\mathcal{L}_6^{\prime\prime\prime\prime}$,
or $C$ is a disjoint union of $12$ lines,
and there exists an~$\mathbb{H}$-invariant quadric surface that contains at least four irreducible components of the~curve $C$.
\end{lemma}

\begin{proof}
Let $\ell$ be an~irreducible component of the curve $C$, let $\Gamma=\mathrm{Stab}_{G}(\ell)$. Then $|\Gamma|\geqslant 4$.
Since $\mathbb{P}^3$ contains no $\mathbb{H}$-invariant lines, one has $\Gamma\cong\mathfrak{A}_4$ or $\Gamma\cong\mumu_2^3$ or $\Gamma\cong\mumu_2^3$ by Remark~\ref{remark:48-Tim-subgroups}.
Therefore, we see that $d\in\{4,6,12\}$.

By Remark~\ref{remark:48-Tim-subgroups}, the group $G$ contains five subgroups isomorphic to $\mathfrak{A}_4$ up to conjugation.
We explicitly described the generators of these subgroups in the~proof of Lemma~\ref{lemma:48-orbits}.
Three of them are stabilizers of a point in the~$G$-orbits $\Sigma_{4}$, $\Sigma_{4}^\prime$, $\Sigma_{4}^{\prime\prime}$,
and none of them leaves a line in $\mathbb{P}^3$ invariant, hence $\Gamma$ is not one of them.
If $\Gamma$ is one of the~two remaining subgroups in $G$ isomorphic to $\mathfrak{A}_4$,
then $\Gamma$ leaves invariant exactly two lines in $\mathbb{P}^3$ --- these are either  components of the~curves
$\mathcal{L}_4$ and $\mathcal{L}_4^\prime$, or components of the~curves $\mathcal{L}_4^{\prime\prime}$ and $\mathcal{L}_4^{\prime\prime\prime}$.
Thus, if $\Gamma\cong\mathfrak{A}_4$, then $C$ is one of the~curves $\mathcal{L}_4$, $\mathcal{L}_4^\prime$,
$\mathcal{L}_4^{\prime\prime}$, $\mathcal{L}_4^{\prime\prime\prime}$.

Now, we suppose that $\Gamma\cong\mumu_2^3$.
Up to conjugation, the~group $G$ contains exactly five subgroups isomorphic to $\mumu_2^3$.
Their generators are explicitly described in Remark~\ref{remark:48-Tim-subgroups}.
For instance, consider the~subgroup $\langle B,M,N\rangle$.
This subgroup leaves invariant exactly two lines in $\mathbb{P}^3$ --- the~lines $\{x_0=x_1=0\}$ and $\{x_2=x_3=0\}$,
which are irreducible components of the~curve $\mathcal{L}_6$.
Therefore, if $\Gamma$ is conjugated to $\langle B,M,N\rangle$, one has $C=\mathcal{L}_6$.
Similarly, if $\Gamma$ is conjugated to one of the~remaining four subgroups isomorphic to $\mumu_2^3$,
then $C$ is one of the~curves $\mathcal{L}_6^{\prime}$, $\mathcal{L}_6^{\prime\prime}$,
$\mathcal{L}_6^{\prime\prime\prime}$ or $\mathcal{L}_6^{\prime\prime\prime\prime}$.

Hence, we may assume that $\Gamma\cong\mumu_2^2$ and $d=12$.
Arguing as in the proof of Lemma~\ref{lemma:48-orbits},
we see that $\Gamma$ fixes no points in $\mathbb{P}^3$.
Up to conjugation, there are eight possibilities for $\Gamma$, which are described in Remark~\ref{remark:48-Tim-subgroups}.
In each case, $\Gamma$-invariant lines
span one of the~quadric surfaces $\mathcal{Q}_1$, $\mathcal{Q}_2$, $\mathcal{Q}_3$, $\mathcal{Q}_{4}$, $\mathcal{Q}_5$, $\mathcal{Q}_6$, $\mathcal{Q}_{7}$,
$\mathcal{Q}_8$, $\mathcal{Q}_9$, $\mathcal{Q}_{10}$. Thus, we conclude that
$$
\ell\subset\bigcup_{i=1}^{10}\mathcal{Q}_i.
$$
Moreover, explicit computations show that the curve $C$ is a disjoint union of $12$ lines,
and either $C\subset\mathcal{Q}_1$, or one of the~following three possibilities hold:
\begin{enumerate}[(i)]
\item $C\subset\mathcal{Q}_2\cup\mathcal{Q}_3\cup\mathcal{Q}_{4}$,
and each quadric $\mathcal{Q}_2$, $\mathcal{Q}_3$, $\mathcal{Q}_{4}$ contains $4$ components of $C$;
\item $C\subset\mathcal{Q}_5\cup\mathcal{Q}_6\cup\mathcal{Q}_{7}$,
and each quadric $\mathcal{Q}_5$, $\mathcal{Q}_6$, $\mathcal{Q}_{7}$ contains $4$ components of $C$;
\item $C\subset\mathcal{Q}_8\cup\mathcal{Q}_9\cup\mathcal{Q}_{10}$,
and each quadric $\mathcal{Q}_8$, $\mathcal{Q}_9$, $\mathcal{Q}_{10}$ contains $4$ components of $C$.
\end{enumerate}
This completes the proof of the lemma.
\end{proof}

Now, let us prove one auxiliary results that will be used later.

\begin{lemma}
\label{lemma:48-sextic-surfaces}
Let $\mathcal{M}_6$ be the~linear system that is generated by the~sextic surfaces
$$
3\mathcal{Q}_1, \mathcal{Q}_1+S_2, \mathcal{Q}_1+S_3, \big\{x_0^6+x_1^6+x_2^6+x_3^6=0\big\}.
$$
Then $\mathcal{M}_6$ is three-dimensional, its base locus is $\mathcal{L}_6^{\prime\prime\prime}\cup\mathcal{L}_6^{\prime\prime\prime\prime}$,
and $\mathcal{M}_6\vert_{\mathcal{Q}_1}=\mathcal{L}_6^{\prime\prime\prime}+\mathcal{L}_6^{\prime\prime\prime\prime}$.
If~$S$~is~a~$G$-invariant sextic surface in $\mathbb{P}^3$, then $S\in\mathcal{M}_6$ or $S=\mathcal{Q}_1+S_4$ or $S=\mathcal{Q}_1+S_5$.
\end{lemma}

\begin{proof}
All assertions about $\mathcal{M}_6$ are easy and can be checked using explicit computations.
Arguing as in the~proof of Lemma~\ref{lemma:48-representations},
we obtain the remaining assertion.
\end{proof}

Now, let us describe all $G$-irreducible curves in $\mathbb{P}^3$ that consist of $4$ irreducible conics.
Let $\mathcal{C}_8^1$ be the~$G$-irreducible curve in $\mathbb{P}^3$ whose irreducible component is the~conic
$$
\big\{x_0=x_1^2+x_2^2+x_3^2=0\big\},
$$
let $\mathcal{C}_8^{2}$ be the~$G$-irreducible curve whose irreducible component is the~conic
$$
\big\{x_0=2x_1^2-(1-\sqrt{3}i)x_2^2-(1+\sqrt{3}i)x_3^2=0\big\},
$$
and let $\mathcal{C}_8^{3}$ be the~$G$-irreducible curve whose irreducible component is the~conic
$$
\big\{x_0=2x_1^2-(1+\sqrt{3}i)x_2^2-(1-\sqrt{3}i)x_3^2=0\big\}.
$$
Then $\mathcal{C}_8^1$, $\mathcal{C}_8^{2}$ and $\mathcal{C}_8^{3}$ are union of $4$ irreducible conics that are contained in the~surface~$\mathcal{T}$.
Moreover, one has $\mathcal{C}_8^1=\mathcal{T}\cap\mathcal{Q}_1$, which implies that $\mathcal{C}_8^1$ is connected.
On the~other hand, the~curves $\mathcal{C}_8^{2}$ and $\mathcal{C}_8^{3}$ are disjoint unions of $4$ conics.

Recall that $R$ is a generator of the group $G_{144,184}$ defined earlier. Let
$$
\mathcal{C}_8^{1,\prime}=R\big(\mathcal{C}_8^{1}\big), \mathcal{C}_8^{2,\prime}=R\big(\mathcal{C}_8^{2}\big), \mathcal{C}_8^{3,\prime}=R\big(\mathcal{C}_8^{3}\big),
$$
and let
$$
\mathcal{C}_8^{1,\prime\prime}=R^2\big(\mathcal{C}_8^{1}\big), \mathcal{C}_8^{2,\prime\prime}=R^2\big(\mathcal{C}_8^{2}\big), \mathcal{C}_8^{3,\prime\prime}=R^2\big(\mathcal{C}_8^{3}\big).
$$
Then $\mathcal{C}_8^{1,\prime}$, $\mathcal{C}_8^{2,\prime}$, $\mathcal{C}_8^{3,\prime}$ are contained in~$\mathcal{T}^\prime$,
and the~curves $\mathcal{C}_8^{1,\prime\prime}$, $\mathcal{C}_8^{2,\prime\prime}$, $\mathcal{C}_8^{3,\prime\prime}$ are contained in~$\mathcal{T}^{\prime\prime}$.
One has $\mathcal{C}_8^{1,\prime}=\mathcal{T}^\prime\cap\mathcal{Q}_1$ and $\mathcal{C}_8^{1,\prime\prime}=\mathcal{T}^{\prime\prime}\cap\mathcal{Q}_1$,
so that both curves $\mathcal{C}_8^{1,\prime}$ and $\mathcal{C}_8^{1,\prime\prime}$ are connected.
On the~other hand, the~curves $\mathcal{C}_8^{2,\prime}$, $\mathcal{C}_8^{3,\prime}$, $\mathcal{C}_8^{2,\prime\prime}$, $\mathcal{C}_8^{3,\prime\prime}$
are disjoint unions of $4$ conics.

\begin{lemma}
\label{lemma:48-conics}
Let $C$ be a $G$-irreducible curve in $\mathbb{P}^3$ that consists of at most $7$~irreducible conics.
Then $C$ is one of the~curves
$\mathcal{C}_8^{1}$, $\mathcal{C}_8^{2}$, $\mathcal{C}_8^{3}$,
$\mathcal{C}_8^{1,\prime}$, $\mathcal{C}_8^{2,\prime}$, $\mathcal{C}_8^{3,\prime}$,
$\mathcal{C}_8^{1,\prime\prime}$, $\mathcal{C}_8^{2,\prime\prime}$, $\mathcal{C}_8^{3,\prime\prime}$.
\end{lemma}

\begin{proof}
Let $\Gamma$ be the~stabilizer of an irreducible component of the~$G$-irreducible curve $C$,
and let $\Pi$ be the~hyperplane in $\mathbb{P}^3$ that contains this irreducible component.
Then $|\Gamma|>6$, and the~plane $\Pi$ is $\Gamma$-invariant.
This implies that $\mathbb{P}^3$ must contain a $\Gamma$-fixed point,
so that it follows from Remark~\ref{remark:48-Tim-subgroups} and Lemma~\ref{lemma:48-orbits} that
$\Gamma\cong\mathfrak{A}_4$, and the~plane $\Pi$ is an irreducible component of one of the~surfaces $\mathcal{T}$, $\mathcal{T}^\prime$, $\mathcal{T}^{\prime\prime}$.
Now, we can explicitly find all $\Gamma$-invariant conics in $\Pi$ to obtain the~required result.
\end{proof}

\begin{corollary}
\label{corollary:48-curves-in-tetrahedra}
Let $C$ be a $G$-irreducible curve contained in $\mathcal{T}\cup\mathcal{T}^\prime\cup\mathcal{T}^{\prime\prime}$ of degree $\leqslant 15$.
Then $C$ is one of the~curves $\mathcal{L}_6$, $\mathcal{L}_6^\prime$, $\mathcal{L}_6^{\prime\prime}$,
$\mathcal{C}_8^{1}$, $\mathcal{C}_8^{2}$, $\mathcal{C}_8^{3}$,
$\mathcal{C}_8^{1,\prime}$, $\mathcal{C}_8^{2,\prime}$, $\mathcal{C}_8^{3,\prime}$,
$\mathcal{C}_8^{1,\prime\prime}$, $\mathcal{C}_8^{2,\prime\prime}$, $\mathcal{C}_8^{3,\prime\prime}$.
\end{corollary}

\begin{proof}
Arguing as in the~proof of Lemma~\ref{lemma:48-conics}, we obtain the required assertion.
\end{proof}

Now, we are ready to prove the following result:

\begin{lemma}
\label{lemma:48-reducible-Q1-curves}
Let $C$ be a reducible $G$-irreducible curve in the~quadric $\mathcal{Q}_1$ of degree $\leqslant 15$.
Then either $C$ is one of the~$G$-irreducible curves
$\mathcal{L}_4$, $\mathcal{L}_4^\prime$, $\mathcal{L}_4^{\prime\prime}$, $\mathcal{L}_4^{\prime\prime\prime}$, $\mathcal{L}_6^{\prime\prime\prime}$, $\mathcal{L}_6^{\prime\prime\prime\prime}$, $\mathcal{C}_8^{1}$, $\mathcal{C}_8^{1,\prime}$, $\mathcal{C}_8^{1,\prime\prime}$,
or the~curve $C$ is a union of $12$ disjoint lines.
\end{lemma}

\begin{proof}
Let $r$ be the~number of irreducible components of the~curve $C$, let $C_1,\ldots,C_r$ be irreducible components of the~curve $C$,
let $d$ be the~degree of the~curve $C_1$, and let $\Gamma$ be the~stabilizer of the~curve $C_1$ in the~group $G$.
Then $\Gamma$ is a subgroup in $G$ of index $r\leqslant\frac{15}{d}$.
By Lemmas~\ref{label:48-lines} and~\ref{lemma:48-conics}, we may assume that $d\geqslant 3$, which gives $r\leqslant 5$,
so~that it follows from Remark~\ref{remark:48-Tim-subgroups} that we have the~following possibilities:
\begin{enumerate}
\item $r=3$, $\Gamma=\mathbb{H}$ and $d\in\{3,4,5\}$,
\item $r=4$, $\Gamma\cong\mathfrak{A}_4$ and $d=3$.
\end{enumerate}
In each case, the~group $\Gamma$ acts faithfully on the~curve $C_1$.

Let us consider the~curve $C_1$ as a divisor of degree  $(a,b)$ in the~quadric $\mathcal{Q}_1\cong\mathbb{P}^1\times\mathbb{P}^1$,
where $a$ and $b$ are some positive integers such that $a+b=d$.
Without loss of generality, we may assume that $a\leqslant b$.
If $r=3$, then $C_1$ is an~irreducible $\mathbb{H}$-invariant curve and
$$
(a,b)\in\big\{(1,2),(1,3),(1,4),(2,2),(2,3)\big\},
$$
which implies that the~genus of the~normalization of the~curve $C_1$ is at most $2$, which contradicts Lemma~\ref{lemma:48-curves}.
Hence, we see that $r\ne 3$.

Thus, we have $r=4$. Then $\Gamma\cong\mathfrak{A}_4$ and $(a,b)=(1,2)$, hence $C_1$ is a smooth twisted cubic curve.
Using GAP, one can check that $\mathcal{Q}_1$ is the~unique $\Gamma$-invariant quadric in~$\mathbb{P}^3$,
and $\mathbb{P}^3$ does not contain pencils of $\Gamma$-invariant quadrics.
Since all quadrics passing through the~curve $C_1$ form a net, we conclude that this net does not contain $\mathcal{Q}_1$,
otherwise we would have a pencil of quadrics surfaces passing through the~curve $C_1$.
This is a contradiction, since $C_1\subset\mathcal{Q}_1$ by assumption.
\end{proof}

Now, we are ready to prove the following result:

\begin{lemma}
\label{lemma:48-reducible-G-invariant-curves}
Let $C$ be a reducible $G$-irreducible curve in $\mathbb{P}^3$ of degree $d\leqslant 15$
that is not contained in $\mathcal{Q}_1\cup\mathcal{T}\cup\mathcal{T}^\prime\cup\mathcal{T}^{\prime\prime}$.
Then $d=12$, and either $C$ is a union of twelve lines, or the~curve $C$ is a union of four twisted cubic curves.
\end{lemma}

\begin{proof}
Since  $C$ is not contained in the~quadric $\mathcal{Q}_1$, we see that $\mathcal{Q}_1\cdot C$ is a $G$-invariant one-cycle in $\mathcal{Q}_1$ of degree $2d\leqslant 30$.
One the~other hand, we know that all $G$-orbits in the~quadric $\mathcal{Q}_1$ are of lengths $12$, $16$ and $24$.
Hence, one has 
$$
30\geqslant 2d=12a+16b
$$
for some non-negative integers $a$ and $b$.
Therefore, we conclude that $d\in\{6,8,12,14\}$.

By Lemmas~\ref{label:48-lines} and~\ref{lemma:48-conics},
we may assume that components of the~curve $C$ are neither lines nor conics.
Since $G$ does not contain subgroups of index~$2$, we see that $d=12$ and
\begin{itemize}
\item either $C$ is a union of four twisted cubic curves,
\item or $C$~is~a~union of three irreducible curves of degree $4$.
\end{itemize}
Moreover, in the~latter case, the subgroup $\mathbb{H}$ is the~stabilizer in $G$ of every irreducible component of the~curve $C$,
because $\mathbb{H}$ is the only subgroup in $G$ of index $3$ by Remark~\ref{remark:48-Tim-subgroups}.
One the~other hand, it follows from Lemma~\ref{lemma:48-curves}
that $\mathbb{H}$ cannot act faithfully on a rational curve,
and  $\mathbb{H}$ cannot act faithfully on a smooth elliptic curve.
Hence, we conclude that irreducible components of the~curve $C$ cannot be curves in $\mathbb{P}^3$ of degree $4$,
which implies that the~curve $C$ is a union of four twisted cubic curves as claimed.
\end{proof}

From the~proof of Lemma~\ref{label:48-lines}, we know that $\mathbb{P}^3$ contains infinitely many $G$-irreducible curves that are unions of twelve lines.
Similarly, one can show that $\mathbb{P}^3$ contains infinitely many $G$-irreducible curves that are unions of four twisted cubics.

\begin{example}
\label{example:48-four-twisted-cubics}
Let $L=\{x_0+ix_2=x_1+ix_3=0\}$, let $P_s=[i:s:si:1]$ for $s\in\mathbb{C}\cup\{\infty\}$, and let $\Gamma$ be the subgroup in $G$ generated by $ABA$ and $BMN$.
Then $L$ is an irreducible component of the curve $\mathcal{L}_6^{\prime\prime\prime}$, $P_s\in L$, $\Gamma\cong\mathfrak{A}_4$,
and $\mathrm{Orb}_{\Gamma}(P_s)$ consists of the~six points
\begin{multline*}
\big[-is:-i:s:1\big], \big[i:s:si:1\big], \big[1:i:is:s\big], \\
\big[-is:-1:i:s\big], \big[-s:si:-i:1\big], \big[-i:is:-1:s\big],
\end{multline*}
which are contained in six distinct irreducible components of the~$G$-irreducible curve $\mathcal{L}_6^{\prime\prime\prime}$.
Suppose, in addition, that $s\ne\frac{1+\sqrt{3}}{2}+\frac{1+\sqrt{3}}{2}i$ and $s\ne\frac{1-\sqrt{3}}{2}+\frac{1-\sqrt{3}}{2}i$.
Then $P_s\not\in\mathcal{L}^{\prime\prime\prime}_4\cup\mathcal{L}^{\prime\prime}_4$,
and no four points in the~$\Gamma$-orbit $\mathrm{Orb}_{\Gamma}(P_s)$ are coplanar.
Let $C_s$ be the~unique twisted cubic in $\mathbb{P}^3$ that contains $\mathrm{Orb}_{\Gamma}(P_s)$,
and let $\mathcal{C}_{12}^s$ be the~$G$-irreducible curve in $\mathbb{P}^3$~whose irreducible component is the~curve $C_s$.
Then $C_s=\{h_1=h_2=h_3=0\}$, where
\begin{multline*}
h_1=(s^2+(1+i)s-i)x_0^2-(2is^2+(2+2i)s-2)x_0x_1+(2is^2-(2+2i)s-2)x_3x_0-\\
-(s^2+(1+i)s-i)x_1^2+(-2is^2+(2+2i)s+2)x_2x_1+(s^2+(1+i)s-i)x_2^2+\\
-(2is^2+(2+2i)s+2)x_3x_2-(s^2+(1+i)s-i)x_3^2,
\end{multline*}
\begin{multline*}
h_2=(-s^2-(1+i)s+i)x_0^2+(2is^2-(2+2i)s-2)x_0x_1+(-2is^2-(2+2i)s+2)x_2x_0+\\
+(s^2+(1+i)s-i)x_1^2-(2is^2+(2+2i)s-2)x_3x_1+(s^2+(1+i)s-i)x_2^2+\\
-(2is^2-(2+2i)s-2)x_3x_2-(s^2+(1+i)s-i)x_3^2,
\end{multline*}
\begin{multline*}
h_3=(s^2+(1+i)s-i)x_0^2+(2is^2-(2+2i)s-2)x_0x_2+(2is^2+(2+2i)s-2)x_3x_0+\\
+(s^2+(1+i)s-i)x_1^2+(2is^2+(2+2i)s-2)x_1x_2-\\
-(2is^2-(2+2i)s-2)x_3x_1-(s^2+(1+i)s-i)x_2^2-(s^2+(1+i)s-i)x_3^2.
\end{multline*}
The curve $\mathcal{C}_{12}^s$ is a union of four twisted cubic curves.
For general choice of $s\in\mathbb{C}\cup\{\infty\}$, these twisted curves are disjoint,
but for some $s\in\mathbb{C}\cup\{\infty\}$ the~cubics are not disjoint.
To be precise, the~curve $\mathcal{C}_{12}^s$ is a disjoint union of four twisted cubic curves if and only if
$$
s\in\Big\{\infty,0,\pm 1,\pm i,\frac{-1\pm\sqrt{3}}{2}+\frac{-1\pm\sqrt{3}}{2}i,\frac{1\pm\sqrt{3}}{2}-\frac{1\pm\sqrt{3}}{2}i,\frac{-1\pm\sqrt{3}}{2}+\frac{1\pm\sqrt{3}}{2}i\Big\}.
$$
For~instance, one has $\mathrm{Sing}(\mathcal{C}_{12}^\infty)=\Sigma_{12}^\prime$ and $C_\infty$ is given by
$$
\left\{\aligned
&x_0^2-2ix_0x_1+2ix_3x_0-x_1^2-2ix_1x_2+x_2^2-2ix_2x_3-x_3^2=0,\\
&x_0^2-2ix_0x_1+2ix_0x_2-x_1^2+2ix_1x_3-x_2^2+2ix_2x_3+x_3^2=0,\\
&x_0^2+2ix_0x_2+2ix_3x_0+x_1^2+2ix_1x_2-2ix_1x_3-x_2^2-x_3^2=0.\\
\endaligned
\right.
$$
In this case, two irreducible component of the~curve $\mathcal{C}_{12}^\infty$ intersect by two points in $\Sigma_{12}^\prime$,
and every irreducible component of the~curve $\mathcal{C}_{12}^\infty$ contains four points in the~$G$-orbit~$\Sigma_{12}^\prime$.
Likewise, if $s=\frac{-1\pm\sqrt{3}}{2}+\frac{-1\pm\sqrt{3}}{2}i$,
then~all components of the~curve $\mathcal{C}_{12}^s$ contain $\Sigma_{4}$.
\end{example}

Let us present some irreducible $G$-invariant curves in $\mathcal{Q}_1$.

\begin{example}
\label{example:48-curve-degree-8}
Recall that $\mathcal{Q}_1$ is contained in the~net $\mathcal{M}_{4}$,
so that $\mathcal{M}_{4}\vert_{\mathcal{Q}_1}$ is a pencil, whose base locus is $\Sigma_{16}\cup\Sigma_{16}^\prime$ by Lemma~\ref{lemma:48-net}.
Note that all curves in $\mathcal{M}_{4}\vert_{\mathcal{Q}_1}$ are $G$-invariant.
Moreover, using Remark~\ref{remark:48-Tim-subgroups} and Lemmas~\ref{lemma:48-curves} and~\ref{lemma:48-conics}
one can show that every curve in the~pencil $\mathcal{M}_{4}\vert_{\mathcal{Q}_1}$ is reduced,
and all reducible curves in $\mathcal{M}_{4}\vert_{\mathcal{Q}_1}$ are
$$
\mathcal{T}\vert_{\mathcal{Q}_1}=\mathcal{C}_8^{1},
\mathcal{T}^\prime\vert_{\mathcal{Q}_1}=\mathcal{C}_8^{1,\prime},
\mathcal{T}^{\prime\prime}\vert_{\mathcal{Q}_1}=\mathcal{C}_8^{1,\prime\prime},
S_2\vert_{\mathcal{Q}_1}=\mathcal{L}_{4}^\prime+\mathcal{L}_{4}^{\prime\prime\prime}, S_3\vert_{\mathcal{Q}_1}=\mathcal{L}_{4}+\mathcal{L}_{4}^{\prime\prime}.
$$
Since the~arithmetic genus of irreducible curves in $\mathcal{M}_{4}\vert_{\mathcal{Q}_1}$ is $9$,
it follows from Lemma~\ref{lemma:48-orbits} that all remaining curves in $\mathcal{M}_{4}\vert_{\mathcal{Q}_1}$
are smooth irreducible $G$-invariant curves of genus $9$.
\end{example}

\begin{example}
\label{example:48-curve-degree-12-genus-21}
Observe that
\begin{align*}
(S_3+S_4)\vert_{\mathcal{Q}_1}&=\mathcal{L}_{4}+\mathcal{L}_{4}^\prime+2\mathcal{L}_{4}^{\prime\prime}\sim 2\mathcal{L}_{4}+2\mathcal{L}_{4}^{\prime\prime\prime}=2S_5\vert_{\mathcal{Q}_1},\\
(S_2+S_5)\vert_{\mathcal{Q}_1}&=\mathcal{L}_{4}+\mathcal{L}_{4}^\prime+2\mathcal{L}_{4}^{\prime\prime\prime}\sim 2\mathcal{L}_{4}^\prime+2\mathcal{L}_{4}^{\prime\prime}=2S_4\vert_{\mathcal{Q}_1},\\
(S_3+S_5)\vert_{\mathcal{Q}_1}&=2\mathcal{L}_{4}+\mathcal{L}_{4}^{\prime\prime}+\mathcal{L}_{4}^{\prime\prime\prime}\sim 2\mathcal{L}_{4}^\prime+2\mathcal{L}_{4}^{\prime\prime}=2S_4\vert_{\mathcal{Q}_1},\\
(S_2+S_4)\vert_{\mathcal{Q}_1}&=2\mathcal{L}_{4}^\prime+\mathcal{L}_{4}^{\prime\prime}+\mathcal{L}_{4}^{\prime\prime\prime}\sim 2\mathcal{L}_{4}+2\mathcal{L}_{4}^{\prime\prime\prime}=2S_5\vert_{\mathcal{Q}_1}.
\end{align*}
Using this, we can create $4$ pencils on the~quadric $\mathcal{Q}_1$ that consist of $G$-invariant curves.
These are the~pencils generated by $\mathcal{L}_{4}^\prime+2\mathcal{L}_{4}^{\prime\prime}$ and $\mathcal{L}_{4}+2\mathcal{L}_{4}^{\prime\prime\prime}$,
by~$\mathcal{L}_{4}+2\mathcal{L}_{4}^{\prime\prime\prime}$ and $\mathcal{L}_{4}^\prime+2\mathcal{L}_{4}^{\prime\prime}$,
by~$2\mathcal{L}_{4}+\mathcal{L}_{4}^{\prime\prime\prime}$ and $2\mathcal{L}_{4}^\prime+\mathcal{L}_{4}^{\prime\prime}$,
by~$2\mathcal{L}_{4}^\prime+\mathcal{L}_{4}^{\prime\prime}$ and $2\mathcal{L}_{4}+\mathcal{L}_{4}^{\prime\prime\prime}$, respectively.
One can show~that general curves in these four pencils are~smooth irreducible $G$-invariant curves of genus~$21$.
Moreover, one can also check that each pencil contains three irreducible singular curves whose singular loci are the~$G$-orbits $\Sigma_{12}^\prime$, $\Sigma_{12}^{\prime\prime}$, $\Sigma_{12}^{\prime\prime\prime}$, respectively.
These curves have ordinary nodes as singularities, so their normalizations have genus $9$.
\end{example}

Now, we are ready to describe irreducible $G$-irreducible curves in $\mathcal{Q}_1$ of small degree.

\begin{lemma}
\label{lemma:48-Q1-curves}
Let $C$ be an irreducible $G$-invariant curve in  $\mathcal{Q}_1\cong\mathbb{P}^1\times\mathbb{P}^1$ of degree $(a,b)$,
where $a$ and $b$ are some non-negative integers. Suppose, in addition, that $a+b\leqslant 15$.
Then one of the following three possibilities holds:
\begin{itemize}
\item $(a,b)=(4,4)$, and $C$ is a smooth curve of genus $9$,
\item $(a,b)=(4,8)$ or $(a,b)=(8,4)$, and $C$ is a smooth curve of genus $21$,
\item $(a,b)=(4,8)$ or $(a,b)=(8,4)$, and $C$ is a singular curve with $12$ ordinary nodes, and the normalization of the curve $C$ has genus $9$.
\end{itemize}
\end{lemma}

\begin{proof}
Without loss of generality, we may assume that $\mathcal{L}_4$ is a divisor in $\mathcal{Q}_1$ of degree~$(0,4)$,
so that $\mathcal{L}_4^{\prime\prime}$ is a divisor of degree $(4,0)$.
Observe also that the~quadric $\mathcal{Q}_1$ is $G_{96,227}$-invariant, and group $G_{96,227}$ maps $C$ to a curve of degree $(b,a)$.
Thus, we may assume that $a\leqslant b$.

By Lemma~\ref{lemma:48-curves}, the~curve $C$ is irrational, it is not elliptic and it is not hyperelliptic, so~that we have $a\geqslant 3$.
Moreover, if $a$ is odd, then $C\cdot\mathcal{L}_4=4a$ is not divisible by~$8$, which contradicts Lemma~\ref{lemma:48-orbits},
because all $G$-orbits in the~curve $\mathcal{L}_4$ have lengths $16$, $24$ or $48$.
Hence, we see that $a$ is even.
Similarly, we see that $b$ is also even, because $C\cdot\mathcal{L}_4^{\prime\prime}=4b$.
Therefore, we conclude that $(a,b)\in\{(4,4),(4,6),(4,8),(4,10),(6,6),(6,8)\}$.

Let $p_a(C)$ be the~arithmetic genus of the~curve~$C$. Then $p_a(C)=ab-a-b+1$, hence $$
\big(a,b,p_a(C)\big)\in\big\{(4,4,9),(4,6,15),(4,8,21),(4,10,27),(6,6,25),(6,8,35)\big\}.
$$
Let $\pi\colon\widetilde{C}\to C$ be the~normalization of the~curve $C$, let $g$ be the~genus of the~curve $\widetilde{C}$.
Then the~$G$-action lifts to $\widetilde{C}$, and it follows from Lemma~\ref{lemma:48-orbits} that
$$
0\leqslant g=p_a(C)-12\alpha-16\alpha
$$
for some integers $\alpha\geqslant 0$~and~$\alpha\geqslant 0$.
If $(a,b)=(4,4)$, then $g=p_a(C)$, hence $C$ is smooth.
Similarly, if $(a,b)=(4,6)$, then we have $g\in\{3,15\}$, which is impossible by Lemma~\ref{lemma:48-curves}.
Likewise, if $(a,b)=(4,10)$ or $(a,b)=(6,8)$, then
$$
g\in\{3,7,11,15,19,23,27,35\},
$$
so that $g\in\{23,27,35\}$ by Lemma~\ref{lemma:48-curves}.
Moreover, arguing as in the~proof of  Lemma~\ref{lemma:48-curves}, we see that $g\not\in\{23,27,35\}$.
Hence, we may assume that $(a,b)=(4,8)$ or $(a,b)=(6,6)$.

If $(a,b)=(6,6)$, then the~curve $C$ is cut out on $\mathcal{Q}_1$ by a $G$-invariant sextic surface~in~$\mathbb{P}^3$,
which gives $C=\mathcal{L}_6^{\prime\prime\prime}+\mathcal{L}_6^{\prime\prime\prime\prime}$ by Lemma~\ref{lemma:48-sextic-surfaces},
which is absurd, since $C$ is irreducible.

Therefore, we have $(a,b)=(4,8)$. If $C$ is smooth, then we are done. If $C$ is singular, then it follows from $g=21-12\alpha-16\alpha$ and Lemma~\ref{lemma:48-curves}
that $g=9$, which implies that the~curve $C$ has $12$ ordinary nodes as required.
\end{proof}

Now, we deal with irreducible $G$-invariant curves in $\mathbb{P}^3$ that are not contained in $\mathcal{Q}_1$.

\begin{lemma}
\label{lemma:48-irreducible-curves-simple}
Let $C$ be an irreducible $G$-invariant curve in $\mathbb{P}^3$ of degree $d\leqslant 15$ such that the~curve $C$ is not contained in $\mathcal{Q}_1$.
Then $C$ is smooth, $d=12$, its genus is $9$, $13$ or $17$,
the~curve $C$ is contained in a surface in $\mathcal{M}_{4}$ that has at most ordinary double points,
and the~curve $C$ does not contain $G$-orbits
$\Sigma_{4}$, $\Sigma_{4}^{\prime}$, $\Sigma_{4}^{\prime\prime}$, $\Sigma_{12}$, $\Sigma_{12}^\prime$,
$\Sigma_{12}^{\prime\prime}$, $\Sigma_{12}^{\prime\prime\prime}$.
\end{lemma}

\begin{proof}
If $C$ is smooth, then $C$ does not contain
$\Sigma_{4}$, $\Sigma_{4}^{\prime}$, $\Sigma_{4}^{\prime\prime}$, $\Sigma_{12}$, $\Sigma_{12}^\prime$,
$\Sigma_{12}^{\prime\prime}$, $\Sigma_{12}^{\prime\prime\prime}$,
because stabilizers in $G$ of smooth points in $C$ are cyclic groups \cite[Lemma~2.7]{FZ}.

Recall from Lemma~\ref{lemma:48-orbits} that $G$-orbits in the~quadric $\mathcal{Q}_1$ are of length $12$, $16$, $24$, $48$,
and the~$G$-orbits of length $12$ in $\mathcal{Q}_1$ are $\Sigma_{12}^\prime$, $\Sigma_{12}^{\prime\prime}$, $\Sigma_{12}^{\prime\prime\prime}$.
On the~other hand, if $C$ contains one of these $G$-orbits of length $12$, then $C$ must be singular at it.
Thus, we conclude that
$$
2d=\mathcal{Q}_1\cdot C=24a+16b
$$
for some non-negative integers $a$ and $b$. Hence, either $d=8$ or $d=12$.

Let $P$ and $Q$ be two general points in $C$, and let $S$ be a surface in the~net $\mathcal{M}_{4}$
that passes through $P$ and $Q$. Then $C\subset S$, since otherwise we would have
$$
48\geqslant 4d=S\cdot C\geqslant |\mathrm{Orb}_G(P)|+|\mathrm{Orb}_G(Q)|=96,
$$
because $G$-orbits of the~points $P$ and $Q$ are of length $48$.

Observe that $S$ is irreducible by Lemma~\ref{lemma:48-net}, since $C$ is not contained in $\mathcal{Q}_1$,
$\mathcal{T}$, $\mathcal{T}^\prime$, $\mathcal{T}^{\prime\prime}$.
Thus, it follows from Lemma~\ref{lemma:48-net} that $S$ has at most isolated ordinary double points.

Let $\pi\colon \widetilde{S}\to S$ be the~minimal resolution of singularities of the~$G$-invariant surface~$S$.
Then $\widetilde{S}$ is a smooth K3 surface, and the~action of the~group $G$ lifts to the~surface $\widetilde{S}$.
Let~$H$~be a general hyperplane section of the~surface $S$, let $\widetilde{H}=\pi^*(H)$,
let $\widetilde{C}$ be the~proper transform of the~curve $C$ on the~surface $\widetilde{S}$,
let~$p_a(\widetilde{C})$ be the~arithmetic genus of the~curve~$\widetilde{C}$,
and let $g$ be the~genus of the~normalization of the~curve $C$.
Then
$$
36\geqslant\frac{d^2}{4}=\frac{\big(\widetilde{H}\cdot\widetilde{C}\big)^2}{\widetilde{H}^2}\geqslant \widetilde{C}^2=2p_a(\widetilde{C})-2\geqslant 2g-2
$$
by the~Hodge index theorem, hence $g\leqslant p_a(\widetilde{C})\leqslant 19$. Then  $g\in\{9,13,17\}$ by Lemma~\ref{lemma:48-curves}.
But it follows from Corollary~\ref{corollary:48-net} that $G$-orbits in $\widetilde{S}$ are of length $16$, $24$ or $48$.
Then
$$
19\geqslant p_a(\widetilde{C})=g+16a+24b\geqslant 9
$$
for some non-negative integers $a$ and $b$.
This implies that $p_a(\widetilde{C})=g$, hence $\widetilde{C}$ is smooth.
Hence, we have $\mathrm{Sing}(C)\subset\mathrm{Sing}(S)$.

If $d=8$, then the~Hodge index theorem gives $g=9$ and $64=(\widetilde{H}\cdot\widetilde{C})^2=\widetilde{H}^2\widetilde{C}^2=4\widetilde{C}^2$,
so that $\widetilde{C}\sim_{\mathbb{Q}} 2\widetilde{H}$,
which implies that $\widetilde{C}\sim 2\widetilde{H}$, because the~group $\mathrm{Pic}(\widetilde{S})$ is torsion free.
Hence, if $d=8$, then $C$ is contained in the~smooth locus of the~surface $S$, and $C\sim 2H$.
On the~other hand, the~restriction map
$$
H^0\big(\mathbb{P}^3,\mathcal{O}_{\mathbb{P}^3}(2)\big)\rightarrow H^0\big(S,\mathcal{O}_{S}(2H)\big)
$$
is a surjective map of $\widehat{G}$-representations. Therefore, if $d=8$, then we have $C=S\cap \mathcal{Q}_1$,
which is impossible by our assumption. Hence, we see that $d\ne 8$.

To complete the~proof, we must show that $C$ is  smooth. Suppose that it is~not smooth.
Then the~surface $S$ is also singular, because $\mathrm{Sing}(C)\subset\mathrm{Sing}(S)$.
By Lemma~\ref{lemma:48-net}, we have the~following possibilities:
\begin{enumerate}[(i)]
\item either $\mathrm{Sing}(S)$ is a $G$-orbit of length $16$,
\item or $\mathrm{Sing}(S)$ is a $G$-orbit of length $12$,
\item or $\mathrm{Sing}(S)$ is a $G$-orbit of length $4$,
\item or $\mathrm{Sing}(S)$ is a union of a $G$-orbit of length $12$ and a $G$-orbit of length $4$,
\item or $\mathrm{Sing}(S)$ is a union of a $G$-orbit of length $12$ and a $G$-orbit of length $4$,
\item or $\mathrm{Sing}(S)$ is a union of two $G$-orbits of length $4$.
\end{enumerate}
Moreover, if $C$ contains a $G$-orbit of length $4$ or $12$, then $C$ is singular at this orbit,
because stabilizers in $G$ of smooth points in $C$ are cyclic.

Let $E_1,\ldots,E_k$ be $G$-irreducible $\pi$-exceptional curves.
Then $E_1,\ldots,E_k$ are disjoint unions of $(-2)$-curves, and $\pi(E_1),\ldots,\pi(E_k)$ are $G$-orbits in $\mathrm{Sing}(S)$. One has
$$
\widetilde{C}\sim_{\mathbb{Q}}\pi^*(C)-\sum_{i=1}^{k}m_iE_i
$$
for some non-negative rational numbers $m_1,\ldots,m_k$ such that $2m_1,\ldots,2m_k$ are integers.
Note that $m_i>0$ if and only if $C$ contains the~$G$-orbits $\pi(E_i)$.
Moreover, one has
\begin{center}
$m_i=\frac{1}{2}$ if and only if $C$ is smooth at the~points of the~$G$-orbits $\pi(E_i)$.
\end{center}
Therefore, if $\pi(E_i)\subset\mathrm{Sing}(C)$, then $m_i\geqslant 1$.
Furthermore, if all  $m_1,\ldots,m_k$ are integers, then the~curve $C$ is a Cartier divisor on the~surface $S$.

Without loss of generality, we may assume that $\pi(E_1)\subset\mathrm{Sing}(C)$. Then
$$
\widetilde{C}^2=C^2+\sum_{i=1}^{k}m_i^2E_i^2=C^2-2\sum_{i=1}^{k}m_i^2\big|\pi(E_i)\big|\leqslant C^2-2m_1^2\big|\pi(E_1)\big|\leqslant C^2-2\big|\pi(E_1)\big|.
$$
Applying Hodge index theorem to $S$, we get $C^2\leqslant 36$, hence $2g-2=\widetilde{C}^2\leqslant 36-2\big|\pi(E_1)\big|$.

Thus, if $\pi(E_1)$ is a $G$-orbit of length $12$ or $16$, then $2g-2=\widetilde{C}^2\leqslant 12$,
so that $g\leqslant 7$, which is impossible by Lemma~\ref{lemma:48-curves}.
Hence, we see that $\pi(E_1)$ is a $G$-orbits of length $4$.

Write $E_1=E_1^1+E_1^2+E_1^3+E_1^4$, where $E_1^1$, $E_1^2$, $E_1^3$ and $E_1^4$ are disjoint $(-2)$-curves.
Let~$\Gamma$ be the~stabilizer in $G$ of the~curve $E_{1}^1$.
Then $\Gamma\cong\mathfrak{A}_4$,
and the~group $\Gamma$ acts faithfully on the~curve $E_1^1$ by Corollary~\ref{corollary:48-net},
so that the~smallest $\Gamma$-orbit in $E_1^2\cong\mathbb{P}^1$ is of length $4$.
Hence, the~intersection $\widetilde{C}\cap E_1^1$ consists of at least $4$ points, which implies that
$$
4\leqslant\big|\widetilde{C}\cap E_{1}^1\big|\leqslant\widetilde{C}\cdot E_{1}^1=\Big(\pi^*(C)-\sum_{i=1}^{k}m_iE_i\Big)E_{1}^1=2m_1,
$$
so that $m_1\geqslant 2$. Then
$2g-2=\widetilde{C}^2\leqslant 36-2m_1^2|\pi(E_1)|=36-8m_1^2\leqslant 4$,
so that $g\leqslant 3$, which is impossible by Lemma~\ref{lemma:48-curves}.
\end{proof}

Unfortunately, we do not know whether $\mathbb{P}^3$ contains irreducible smooth $G$-invariant curves of degree $12$ and genus $9$ or $17$.
On the other hand, we know that $\mathbb{P}^3$ contains infinitely many irreducible smooth $G$-invariant curves of degree $12$ and genus $13$.

\begin{example}
\label{example:48-curve-degree-12}
By \cite[Theorem~3.22]{CheltsovShramov2017}, $\mathbb{P}^3$
contains four irreducible $G_{144,184}$-invariant smooth curves of degree $12$ and genus~$13$.
These four curves can be constructed as follows.
Observe that
$S_2\vert_{\mathcal{Q}_1}=\mathcal{L}_{4}^\prime+\mathcal{L}_{4}^{\prime\prime\prime}$,
$S_3\vert_{\mathcal{Q}_1}=\mathcal{L}_{4}+\mathcal{L}_{4}^{\prime\prime}$,
$S_4\vert_{\mathcal{Q}_1}=\mathcal{L}_{4}^\prime+\mathcal{L}_{4}^{\prime\prime}$ and
$S_5\vert_{\mathcal{Q}_1}=\mathcal{L}_{4}+\mathcal{L}_{4}^{\prime\prime\prime}$.
Hence, none of the~intersections $S_2\cap S_4$, $S_2\cap S_5$, $S_3\cap S_4$, $S_3\cap S_5$ is an irreducible curve.
Moreover, it follows from \cite[Lemma~3.19]{CheltsovShramov2017} that
\begin{align*}
S_2\cap S_4&=\mathcal{L}_{4}^\prime+\mathfrak{C}_{12}^\prime,\\
S_2\cap S_5&=\mathcal{L}_{4}^{\prime\prime\prime}+\mathfrak{C}_{12}^{\prime\prime\prime},\\
S_3\cap S_4&=\mathcal{L}_{4}^{\prime\prime}+\mathfrak{C}_{12}^{\prime\prime},\\
S_3\cap S_5&=\mathcal{L}_{4}+\mathfrak{C}_{12},
\end{align*}
where $\mathfrak{C}_{12}$, $\mathfrak{C}_{12}^\prime$, $\mathfrak{C}_{12}^{\prime\prime}$, $\mathfrak{C}_{12}^{\prime\prime\prime}$
are distinct smooth irreducible curves of degree $12$ and genus~$13$.
Now, we can use the same idea to construct infinitely many irreducible $G$-invariant smooth curves of degree $12$ and genus $13$.
For instance, if $\lambda$ is a~general complex number,
then
$$
\big\{\lambda f_1^2+f_3=f_5=0\big\}
$$
splits as a union of the~$G$-invariant reducible curve  $\mathcal{L}_{4}$ and a smooth $G$-invariant irreducible curve of degree $12$ and genus~$13$.
\end{example}

Irreducible $G$-invariant curves of degree $12$ from Example~\ref{example:48-curve-degree-12} are cut out by sextics.
We~think that this should be true for every $G$-invariant irreducible curve in $\mathbb{P}^3$ of degree~$12$
which is not contained in $\mathcal{Q}_1$. But we~are unable to show this $\frownie$. Instead, we prove

\begin{lemma}
\label{lemma:48-irreducible-curves-simple-deal}
Let $C$ be an irreducible $G$-invariant curve of degree $12$ in $\mathbb{P}^3$ that is not contained in $\mathcal{Q}_1$,
and let $\mathcal{D}$ be a linear subsystem in $|\mathcal{O}_{\mathbb{P}^3}(6)|$ that consists of surfaces passing through the~curve $C$.
Then $\mathcal{D}$ is not empty, $\mathcal{D}$~does not have fixed components,
the curve $C$ is the only curve that is contained in the base locus of the linear system $\mathcal{D}$.
Moreover, if $D$ and $D^\prime$ are general surfaces in $\mathcal{D}$, then $(D\cdot D^\prime)_C=1$.
\end{lemma}

\begin{proof}
It follows from Lemma~\ref{lemma:48-irreducible-curves-simple} that the curve $C$ is smooth, and its genus is $9$, $13$~or~$17$.
Moreover, it follows from  Lemma~\ref{lemma:48-irreducible-curves-simple} that
the curve $C$ is contained in an~irreducible quartic surface in $S\in\mathcal{M}_4$ that has at most ordinary double points.
Then $S+\mathcal{Q}\in\mathcal{D}$
for every quadric $\mathcal{Q}\in|\mathcal{O}_{\mathbb{P}^3}(2)|$.
Thus, the base locus of $\mathcal{D}$ is contained in $S$.

Let $\mathcal{I}_C$ be the ideal sheaf of the curve $C$.
The surfaces in $\mathcal{D}$ are cut out by the global sections in $H^0(\mathcal{O}_{\mathbb{P}^3}(6)\otimes\mathcal{I}_C)$.
On the other hand, we have the following exact sequence:
$$
0\longrightarrow H^0\big(\mathcal{O}_{\mathbb{P}^3}(6)\otimes\mathcal{I}_C\big)\longrightarrow H^0\big(\mathcal{O}_{\mathbb{P}^3}(6)\big)\longrightarrow H^0\big(\mathcal{O}_{\mathbb{P}^3}(6)\big\vert_{C}\big)
$$
Thus, using the~Riemann--Roch theorem and Serre duality, we see that
$$
h^0\big(\mathcal{O}_{\mathbb{P}^3}(6)\otimes\mathcal{I}_C\big)\geqslant h^0\big(\mathcal{O}_{\mathbb{P}^3}(6)\big)-h^0\big(\mathcal{O}_{\mathbb{P}^3}(6)\big\vert_{C}\big)=84-h^0\big(\mathcal{O}_{\mathbb{P}^3}(6)\big\vert_{C}\big)=11+g,
$$
where $g$ is the genus of the curve $C$.
Therefore, the dimension of $\mathcal{D}$ is at least $10+g$.
Then the dimension of the linear system $\mathcal{D}\vert_{S}$ is at least $g$,
because $h^0(\mathcal{O}_{\mathbb{P}^3}(2))=10$.

Let $\mathcal{M}_6$ be the~linear system introduced in Lemma~\ref{lemma:48-sextic-surfaces}.
By Lemma~\ref{lemma:48-sextic-surfaces}, this linear system is three-dimensional, every surface in $\mathcal{M}_6$ is $G$-invariant,
and $\mathcal{M}_6\big\vert_{\mathcal{Q}_1}=\mathcal{L}_6^{\prime\prime\prime}+\mathcal{L}_6^{\prime\prime\prime\prime}$,
so that the base locus of the linear system  $\mathcal{M}_6$ is a union of the curves  $\mathcal{L}_6^{\prime\prime\prime}$ and $\mathcal{L}_6^{\prime\prime\prime\prime}$.
We~claim that $\mathcal{M}_6$ contains a possibly reducible surface such that it passes through $C$, but $S$ is not its irreducible component.
Indeed, let $P$ and $Q$ be two sufficiently general points in the curve $C$,
and let $S_6$ and $S_6^\prime$ be two distinct surfaces in $\mathcal{M}_6$ that both pass through $P$ and $Q$.
If $C\not\subset S_6$, then
$$
72=S_6\cdot C\geqslant \big|\mathrm{Orb}_{G}(P)\big|+\big|\mathrm{Orb}_{G}(Q)\big|=96,
$$
which is absurd. Therefore, we conclude that $C\subset S_6$. Similarly, we see that $C\subset S_6\cap S_6^\prime$.
On the other hand, the quartic surface $S$ is not contained in $S_6\cap S_6^\prime$,
because otherwise we would have $S_6=S_6^\prime=S+\mathcal{Q}_1$, since $\mathcal{Q}_1$ is the only $G$-invariant quadric surface in~$\mathbb{P}^3$.
Hence, either $S\not\subset S_6$ or $S\not\subset S_6^\prime$. Without loss of generality, we may assume that $S$ is not an irreducible component of the surface $S_6$.

We see that $S_6\vert_{S}=C+Z$ for some $G$-invariant curve $Z$. Observe that $\mathrm{deg}(Z)=12$.
If $Z$ is not $G$-irreducible, then it follows from Corollary~\ref{corollary:48-curves-in-tetrahedra}
and Lemmas~\ref{lemma:48-reducible-G-invariant-curves} and \ref{lemma:48-irreducible-curves-simple}
that at least one irreducible component of the curve $Z$ is contained in the quadric $\mathcal{Q}_1$,
because the curves $\mathcal{L}_6$, $\mathcal{L}_6^\prime$, $\mathcal{L}_6^{\prime\prime}$ are not contained in $S$.
On the other hand, we know that
$$
S_6\vert_{\mathcal{Q}_1}=\mathcal{L}_6^{\prime\prime\prime}+\mathcal{L}_6^{\prime\prime\prime\prime},
$$
and neither $\mathcal{L}_6^{\prime\prime\prime}$ nor $\mathcal{L}_6^{\prime\prime\prime\prime}$ are contained in $S$.
Hence, we conclude that $Z$ is $G$-irreducible.
A priori, we may have $Z=C$.

Since $S_6\subset\mathcal{D}$, the base locus of the linear system $\mathcal{D}$ is contained in $S_6\cap S=C\cup Z$.
If~$Z$~is contained in the base locus of the linear system $\mathcal{D}$, then we have $\mathcal{D}\vert_{S}=C+Z$,
so~that $\mathcal{D}\vert_{S}$ is a zero-dimensional linear system.
On the other hand, we already proved earlier that the~dimension of $\mathcal{D}\vert_{S}$ is at least $g\geqslant 13$.
This shows that $C$ is the only curve contained in the base locus of the linear system $\mathcal{D}$.

Likewise, we see that $\mathcal{D}\vert_{S}\ne 2C$. Thus, for a general surface $D\in\mathcal{D}$, one has $(D\cdot S)_C=1$.
This implies the final assertion of the lemma,
since $S+\mathcal{Q}\in\mathcal{D}$ for every $\mathcal{Q}\in|\mathcal{O}_{\mathbb{P}^3}(2)|$.
\end{proof}

Let us conclude this section with one rather technical result.

\begin{proposition}
\label{proposition:48-curves}
Let $C$ be a $G$-irreducible curve in $\mathbb{P}^3$ that is different from $\mathcal{L}_6$, $\mathcal{L}_6^\prime$, $\mathcal{L}_6^{\prime\prime}$,
and let $\mathcal{D}$ be a linear subsystem in $|\mathcal{O}_{\mathbb{P}^3}(n)|$ that has no fixed components, where $n\in\mathbb{Z}_{>0}$.
Then $\mathrm{mult}_C(\mathcal{D})\leqslant\frac{n}{4}$.
Moreover, one has $\mathrm{mult}_{\mathcal{L}_6^\prime}(\mathcal{D})+\mathrm{mult}_{\mathcal{L}_6^{\prime\prime}}(\mathcal{D})\leqslant\frac{n}{2}$.
\end{proposition}

\begin{proof}
First, let us prove the last assertion. To do this, we let
\begin{align*}
L_1^\prime&=\{x_0+x_2=x_1-x_3=0\},\\
L_2^\prime&=\{x_0-x_2=x_1+x_3=0\},\\
L_1^{\prime\prime}&=\{x_0+x_3=x_1+x_2=0\},\\
L_2^{\prime\prime}&=\{x_0-x_3=x_1-x_2=0\}.
\end{align*}
Then the lines  $L_1^\prime$, $L_2^\prime$, $L_1^{\prime\prime}$, $L_2^{\prime\prime}$ are disjoint.
Moreover, the lines $L_1^\prime$ and $L_2^\prime$ are two irreducible components of the curve~$\mathcal{L}_{6}^\prime$,
but $L_1^{\prime\prime}$ and $L_2^{\prime\prime}$ are two irreducible components of the curve~$\mathcal{L}_{6}^{\prime\prime}$.
On the other hand, the lines  $L_1^\prime$, $L_2^\prime$, $L_1^{\prime\prime}$, $L_2^{\prime\prime}$ are contained
in $\mathcal{Q}_2=\{x_0^2+x_1^2-x_2^2-x_3^2=0\}$.
Thus, if $D$ is a general surface in $\mathcal{D}$, then
$$
D\big\vert_{\mathcal{Q}_2}=m_1^{\prime} L_1^\prime+m_2^\prime L_2^\prime+m_1^{\prime\prime}L_1^{\prime\prime}+m_2^{\prime\prime} L_2^{\prime\prime}+\Xi,
$$
where $m_1^{\prime}$, $m_2^{\prime}$,  $m_1^{\prime\prime}$, $m_2^{\prime\prime}$
are such that $m_1^{\prime}\geqslant\mathrm{mult}_{\mathcal{L}_6^\prime}(\mathcal{D})$,
\mbox{$m_2^{\prime}\geqslant\mathrm{mult}_{\mathcal{L}_6^\prime}(\mathcal{D})$},
$m_1^{\prime\prime}\geqslant\mathrm{mult}_{\mathcal{L}_6^{\prime\prime}}(\mathcal{D})$ and $m_2^{\prime\prime}\geqslant\mathrm{mult}_{\mathcal{L}_6^{\prime\prime}}(\mathcal{D})$,
and $\Xi$ is an effective divisor on $\mathcal{Q}_2$ whose support does not contain the lines $L_1^\prime$, $L_2^\prime$, $L_1^{\prime\prime}$, $L_2^{\prime\prime}$.
Now, let $\ell$ be a general line in $\mathcal{Q}_2$ that intersects $L_1^\prime$. Then
$$
n=\ell\cdot D\big\vert_{\mathcal{Q}_2}=m_1^{\prime}+m_2^\prime+m_1^{\prime\prime}+m_2^{\prime\prime}+\ell\cdot\Xi\geqslant 2\mathrm{mult}_{\mathcal{L}_6^\prime}(\mathcal{D})+2\mathrm{mult}_{\mathcal{L}_6^{\prime\prime}}(\mathcal{D}),
$$
so that $\mathrm{mult}_{\mathcal{L}_6^\prime}(\mathcal{D})+\mathrm{mult}_{\mathcal{L}_6^{\prime\prime}}(\mathcal{D})\leqslant\frac{n}{2}$ as claimed.

Now, let $D_1$ and $D_2$ be two general surfaces in the~system $\mathcal{D}$. Then $D_1\cdot D_2=\delta C+\Omega$,
where $\delta$ is a non-negative integer, and $\Omega$ is an effective one-cycle such that $C\not\subset\mathrm{Supp}(\Omega)$.
One has $\delta\geqslant \mathrm{mult}_C^2(\mathcal{D})$.
But the degree of the cycle $D_1\cdot D_2$ is $n^2$.
Then $\delta\mathrm{deg}(C)\leqslant n^2$, which gives the required inequality if $\mathrm{deg}(C)\geqslant 16$.
So, we may assume that $\mathrm{deg}(C)\leqslant 15$.

Now, we suppose that the~curve $C$ is contained in the $G$-invariant quadric surface~$\mathcal{Q}_1$.
Let $\ell_1$ and $\ell_2$ be general curves in the~surface $\mathcal{Q}_1$ of degrees $(0,1)$ and $(1,0)$, respectively.
Then $\ell_1$ and $\ell_2$ are not contained in the base locus of the linear system $\mathcal{D}$,
and it follows from Lemmas~\ref{lemma:48-reducible-Q1-curves} and \ref{lemma:48-Q1-curves} that $\ell_1\cdot C\geqslant 4$
or $\ell_1\cdot C\geqslant 4$. If $\ell_1\cdot C\geqslant 4$, we get
$$
n=D\cdot \ell_1\geqslant\mathrm{mult}_C\big(D\big)|C\cap\ell_1|=\mathrm{mult}_C\big(D\big)(C\cdot\ell_1)\geqslant 4\mathrm{mult}_C\big(D\big)
$$
for sufficiently general surface $D\in\mathcal{D}$, hence $\mathrm{mult}_C(\mathcal{D})=\mathrm{mult}_C(D)\leqslant\frac{n}{4}$ as required.
Similarly, we obtain the required inequality when $\ell_2\cdot C\geqslant 4$.

Thus, to complete the proof of the lemma, we may assume that $C\not\subset\mathcal{Q}_1$.

Now, we suppose that irreducible components of the curve $C$ are lines.
Then it follows from Lemma~\ref{label:48-lines} that $C$ is a union of $12$ disjoint lines.
Moreover, Lemma~\ref{label:48-lines} also implies that there is an~$\mathbb{H}$-invariant quadric
that contains at least four components of the~curve $C$.
Thus, arguing as in the case $C\subset\mathcal{Q}_1$, we~obtain the required inequality.

Now, suppose that irreducible components of the curve $C$ are conics.
By Lemma~\ref{lemma:48-conics}, the~curve $C$ is one of the curves
$\mathcal{C}_8^{2}$, $\mathcal{C}_8^{3}$, $\mathcal{C}_8^{2,\prime}$, $\mathcal{C}_8^{3,\prime}$, $\mathcal{C}_8^{2,\prime\prime}$, $\mathcal{C}_8^{3,\prime\prime}$,
because $\mathcal{C}_8^{1}\cup\mathcal{C}_8^{1,\prime}\cup\mathcal{C}_8^{1,\prime\prime}\subset\mathcal{Q}_1$.
Observe that $G_{144,184}$ transitively permutes the curves $\mathcal{C}_8^{2}$, $\mathcal{C}_8^{2,\prime}$, $\mathcal{C}_8^{2,\prime\prime}$,
and it transitively permutes the~curves $\mathcal{C}_8^{3}$, $\mathcal{C}_8^{3,\prime}$, $\mathcal{C}_8^{3,\prime\prime}$.
Therefore, we may assume that $C=\mathcal{C}_8^{2}$ or $C=\mathcal{C}_8^{3}$.
But the~group $G_{96,227}$ swaps the curves $\mathcal{C}_8^{2}$ and $\mathcal{C}_8^{3}$, hence we may assume that $C=\mathcal{C}_8^{3}$.
Recall that $\mathcal{C}_8^{3}$ is the~$G$-irreducible curve in $\mathbb{P}^3$ whose irreducible component is the conic
$$
\big\{x_0=2x_1^2-\big(1+\sqrt{3}i\big)x_2^2-\big(1-\sqrt{3}i\big)x_3^2=0\big\}\subset\mathbb{P}^3.
$$
Its remaining three irreducible components intersect the plane $\{x_0=0\}$ in the points
\begin{align*}
[0:\sqrt{3}-i:2:0]&, [0:-\sqrt{3}+i:2:0],\\
[0:0:\sqrt{3}+i:2]&, [0:0:-\sqrt{3}-i:2],\\
[0:0:\sqrt{3}-i:2]&, [0:0:-\sqrt{3}+i:2].
\end{align*}
None of these six points is contained in the conic $\{x_0=2x_1^2-(1+\sqrt{3}i)x_2^2-(1-\sqrt{3}i)x_3^2=0\}$.
Let $Z$ be a general conic in the plane $\{x_0=0\}\subset\mathbb{P}^3$ that contains the points
$$
[0:\sqrt{3}-i:2:0], [0:-\sqrt{3}+i:2:0], [0:0:\sqrt{3}+i:2], [0:0:-\sqrt{3}-i:2].
$$
Then $|Z\cap C|=8$, and $Z$ is not contained in the base locus of the linear system $\mathcal{D}$,~so~that
$$
2n=D\cdot Z\geqslant\mathrm{mult}_C\big(D\big)|C\cap Z|=8\mathrm{mult}_C\big(D\big)=8\mathrm{mult}_C\big(\mathcal{D}\big),
$$
where as above $D$ is a general surface in $\mathcal{D}$. This gives us the required inequality.

Therefore, to complete the proof of the proposition, we may assume that irreducible components of the curve $C$ are neither lines nor conics.
Hence, using Corollary~\ref{corollary:48-curves-in-tetrahedra}
and both Lemmas~\ref{lemma:48-reducible-G-invariant-curves} and \ref{lemma:48-irreducible-curves-simple},
we see that either $C$ is a union of four twisted cubic curves,
or $C$ is a smooth irreducible curve of degree $12$, and its genus is $9$, $13$ or $17$.

Now, we suppose that $C$ is a smooth irreducible curve of degree $12$ and genus $g\in\{9,13\}$.
Let $\varphi\colon X\to\mathbb{P}^3$ be the blow up of the smooth curve $C$, let $E_C$ be the $\varphi$-exceptional divisor,
let $\widehat{\mathcal{D}}$ be the proper transform on $X$ of the linear system $\mathcal{D}$,
let $\widehat{D}_1$ and $\widehat{D}_2$ be general surfaces in the~system $\widehat{\mathcal{D}}$.
Using Lemma~\ref{lemma:48-irreducible-curves-simple-deal}, we see that $|\varphi^*(\mathcal{O}_{\mathbb{P}^3}(6))-E_C|$ is not~empty,
this~linear system does not have fixed components,
and it also does not have base curves except possibly for the fibers of the natural projections $E_C\to C$.
Therefore, we conclude that the divisor $\varphi^*(\mathcal{O}_{\mathbb{P}^3}(6))-E_C$ is nef.
Thus, if $\mathrm{mult}_C\big(\mathcal{D}\big)>\frac{n}{4}$, then
$$
0\leqslant\Big(\varphi^*\big(\mathcal{O}_{\mathbb{P}^3}(6)\big)-E_C\Big)\cdot\widehat{D}_1\cdot\widehat{D}_2=(2g-26)\mathrm{mult}_C^2\big(\mathcal{D}\big)-24n\mathrm{mult}_C\big(\mathcal{D}\big)+6n^2<0,
$$
which is absurd. Thus, if $C$ is an irreducible smooth curve, then $g\ne 9$ or $g\ne 13$.

Hence, to complete the proof, we may assume that either $C$ is a smooth irreducible curve of degree $12$ and genus $17$,
or the curve $C$ is a union of four twisted cubic curves.
In~the former case, it follows from  Lemma~\ref{lemma:48-irreducible-curves-simple} that $C$ is contained in an~irreducible surface in the~net $\mathcal{M}_4$.
In fact, arguing as in the proof of Lemma~\ref{lemma:48-irreducible-curves-simple},
we conclude that the~curve $C$ is contained in an~irreducible surface $S\in\mathcal{M}_4$ in both cases,
and it follows from  Lemma~\ref{lemma:48-net} that $S$ has at most ordinary double points.

By the~Hodge index theorem, we have $C^2\leqslant 36$ on the surface $S$.
If $C$ is irreducible, then it follows from Lemmas~\ref{lemma:48-net} and \ref{lemma:48-irreducible-curves-simple}
that either $C\cap\mathrm{Sing}(S)=\varnothing$, or $S$ has $16$ ordinary double points, and $\mathrm{Sing}(S)\subset C$.
Thus, if $C$ is irreducible, the adjunction formula gives
$$
36\geqslant C^2=32+\frac{|C\cap\mathrm{Sing}(S)|}{2}=\left\{\aligned
&32\ \text{if $C\cap\mathrm{Sing}(S)=\varnothing$},\\
&40\ \text{if $C\cap\mathrm{Sing}(S)\ne\varnothing$},\\
\endaligned
\right.
$$
so that $C\cap\mathrm{Sing}(S)=\varnothing$ and $C^2=32$.

Arguing as in the proof of Lemma~\ref{lemma:48-irreducible-curves-simple-deal},
we see that there exists a $G$-invariant sextic surface $S_6\in|\mathcal{O}_{\mathbb{P}^3}(6)|$
such that $C\subset S_6$, but $S$ is not a~component of the sextic surface~$S_6$.
Then~$S_6\vert_{S}=C+Z$ for some $G$-invariant curve $Z$ of degree $12$.
Moreover, arguing as in the proof of Lemma~\ref{lemma:48-irreducible-curves-simple-deal},
we see that $Z$ is $G$-irreducible.
On $S$, we have $C\cdot Z=72-C^2$, since
$72=\big(C+Z\big)\cdot C=C^2+C\cdot Z$.
Similarly, we see that $C^2=Z^2$.

Let~$H$~be a hyperplane section of the~surface $S$,
and let $D$ be a general surface in~$\mathcal{D}$.
Then $nH\sim_{\mathbb{Q}} D\big\vert_{S}=mC+\epsilon Z+\Delta$
for some~effective divisor $\Delta$ on the surface $S$ whose support does not contain $C$~and~$Z$,
where $m$ and $\epsilon$ are some non-negative rational numbers. Then $m\geqslant\mathrm{mult}_C(\mathcal{D})$.
So, it is enough to show that $m\leqslant\frac{n}{4}$.
Suppose that $m>\frac{n}{4}$.

First, let us exclude the case when $C$ is irreducible.
In this case, the curve $C$~is~contained in the smooth locus of the surface $S$, and $C^2=32$ on the surface $S$,
so that it follows from the~Riemann--Roch theorem and Serre duality that
$$
h^0\big(\mathcal{O}_S(4H-C)\big)-h^1\big(\mathcal{O}_S(4H-C)\big)=2+\frac{(4H-C)^2}{2}=2,
$$
which implies that $|4H-C|$ is a pencil.
Since all curves in this pencil have degree~four,
the~pencil $|4H-C|$ has no fixed curves, since otherwise the~union of all its fixed curves would be a $G$-invariant curve in $S$ of degree less than $4$,
which contradicts Corollary~\ref{corollary:48-curves-in-tetrahedra}
and Lemmas~\ref{lemma:48-reducible-G-invariant-curves} and \ref{lemma:48-irreducible-curves-simple}.
In particular, we see that the divisor $4H-C$ is nef, hence $$
4n=nH\cdot (4H-C)=m(4H-C)\cdot C+\epsilon (4H-C)\cdot Z+(4H-C)\cdot \Delta\geqslant m(4H-C)\cdot C=16m,
$$
so that $m\leqslant\frac{n}{4}$, which is a contradiction.

Hence, we see that $C$ is a union of four twisted cubics.
Denote them by $C_1$, $C_2$, $C_3$, $C_4$.
On the surface $S$, we have $C_1^2=C_2^2=C_3^2=C_4^2=-2+|C_1\cap\mathrm{Sing}(S)|/2$,
because  $G$ acts transitively on the~set $\{C_1,C_2,C_3,C_4\}$.
This action gives a homomorphism \mbox{$\upsilon\colon G\to\mathfrak{S}_4$},
whose image $\mathrm{im}(\upsilon)$ is isomorphic to one of the following groups: $\mumu_4$, $\mumu_2^2$, $\mathrm{D}_{8}$, $\mathfrak{A}_4$, $\mathfrak{S}_4$.
Now,~using Remark~\ref{remark:48-Tim-subgroups} and Lemma~\ref{lemma:48-curves}, we conclude that $\mathrm{im}(\upsilon)\cong\mathfrak{A}_4$,
so that $\mathrm{ker}(\upsilon)\cong\mumu_2^2$.
Then $G$ acts two-transitively on $\{C_1,C_2,C_3,C_4\}$, hence $C_i\cdot C_j=C_1\cdot C_2$ for $i\ne j$.
Then
$$
C^2=12(C_1\cdot C_2)+4C_1^2.
$$
If $C\cap\mathrm{Sing}(S)=\varnothing$, then $C_1\cdot C_2$ is an even integer,
because $C_1\cap C_2$ is $\mathrm{ker}(\upsilon)$-invariant,
but the group $\mathrm{ker}(\upsilon)$ acts faithfully on the curve $C_1$, and its orbits have length $2$ or $4$.
Similarly, we see that $C_1\cdot C_2$ is an integer in the case when $C\cap\mathrm{Sing}(S)\ne\varnothing$, because singular points of the surfaces $S$ are at most ordinary double points.

Observe that $m+\epsilon\leqslant\frac{n}{3}$, because
$4n=nH^2=12(m+\epsilon)+H\cdot \Delta\geqslant 12(m+\epsilon)$. But
$$
12n=H\cdot Z=mC\cdot Z+\epsilon Z^2+Z\cdot\Delta\geqslant mC\cdot Z+\epsilon Z^2=mC\cdot Z+\epsilon C^2=72m+(\epsilon-m)C^2.
$$
Thus, if $C^2\leqslant 0$, then $12n\geqslant 72m>18n$, which is absurd. Hence, we have $C^2>0$.
Then
$$
12n\geqslant m(72-C^2)+\epsilon C^2\geqslant m(72-C^2)>\frac{(72-C^2)n}{4},
$$
which gives $C^2>24$. Thus, if $C\cap\mathrm{Sing}(S)=\varnothing$, then $24<C^2=12(C_1\cdot C_2)-8\leqslant 36$,
which is impossible, because $C_1\cdot C_2$ is an~even integer. Hence, we have $C\cap\mathrm{Sing}(S)\ne\varnothing$.

Observe that $\mathrm{Stab}_{G}(C_1)\cong\mathfrak{A}_4$, this group acts faithfully on the~curve $C_1$,
and its orbits in the~curve $C_1$ are of length $4$, $6$ and $12$.
Moreover, the twisted cubic curve $C_1$ contains exactly two $\mathrm{Stab}_{G}(C_1)$-orbits of length $4$, and it has a unique $\mathrm{Stab}_{G}(C_1)$-orbit of length~$6$.
But $|C_1\cap\mathrm{Sing}(S)|\leqslant 9$,
because the subset $\mathrm{Sing}(S)\subset\mathbb{P}^3$ is cut out by cubic hypersurfaces.
Thus, we conclude that one of the following three cases are possible:
\begin{itemize}
\item $C_1^2=0$ and $C_1\cap\mathrm{Sing}(S)$ is a $\mathrm{Stab}_{G}(C_1)$-orbit of length $4$;
\item $C_1^2=1$ and $C_1\cap\mathrm{Sing}(S)$ is the unique $\mathrm{Stab}_{G}(C_1)$-orbit of length $6$;
\item $C_1^2=2$ and $C_1\cap\mathrm{Sing}(S)$ is the union of two $\mathrm{Stab}_{G}(C_1)$-orbits of length $4$.
\end{itemize}
But we know that $24<C^2=12(C_1\cdot C_2)+4C_1^2\leqslant 36$, hence $6<3(C_1\cdot C_2)+C_1^2\leqslant 9$,
where $C_1\cdot C_2$ is an integer. Hence, we see that $C_1\cdot C_2=2$, and either $C_1^2=1$ or $C_1^2=2$.
Therefore, we conclude that either $C^2=28$ and $C_1^2=1$, or $C^2=32$ and $C_1^2=2$.

Let $\pi\colon \widetilde{S}\to S$ be the~minimal resolution of singularities,
let $E$ be the sum of exceptional curves of the morphism $\pi$,
let $\widetilde{C}$ be the~proper transform of the~curve $C$ on the~surface~$\widetilde{S}$,
let~$\widetilde{C}_1$, $\widetilde{C}_2$, $\widetilde{C}_3$, $\widetilde{C}_4$ be the~proper transforms on $\widetilde{S}$
of the~curves $C_1$, $C_2$, $C_3$, $C_4$, respectively.
Then $\widetilde{S}$ is a smooth K3 surface, and the~action of the~group $G$ lifts to the~surface $\widetilde{S}$.
Arguing as above, we get $\widetilde{C}^2=12(\widetilde{C}_1\cdot \widetilde{C}_2)-8$,
where $\widetilde{C}_1\cdot \widetilde{C}_2$ is an even non-negative~integer.

Suppose that the~set $\mathrm{Sing}(S)$ is formed by one $G$-orbit. Then $E$ is a $G$-irreducible~curve.
Let $P$ be a singular point of the~quartic surface $S$, and let $k$ be the number of irreducible components of the curve $C$ that pass through the~point $P$.
Then
$$
\widetilde{C}\sim_{\mathbb{Q}} \pi^*(C)-\frac{k}{2}E,
$$
because irreducible components of the curve $C$ are smooth.
Therefore, since all irreducible components of the curve $E$ are $(-2)$-curves, we get
$$
\widetilde{C}^2=C^2-\frac{k^2}{2}\big|\mathrm{Sing}(S)\big|,
$$
where $C^2=28$ or $C^2=32$. For instance, if $|\mathrm{Sing}(S)|=16$, then $C^2+8=8k^2+12(\widetilde{C}_1\cdot \widetilde{C}_2)$,
so that either $9=2k^2+3(\widetilde{C}_1\cdot \widetilde{C}_2)$ or $10=2k^2+3(\widetilde{C}_1\cdot \widetilde{C}_2)$,
which leads to a contradiction, since $\widetilde{C}_1\cdot \widetilde{C}_2$ is an even integer.
Similarly, if $|\mathrm{Sing}(S)|=12$, then
$$
C^2+8=6k^2+12\big(\widetilde{C}_1\cdot \widetilde{C}_2\big),
$$
which leads to a contradiction.
Thus, it follows from Lemma~\ref{lemma:48-net} that $|\mathrm{Sing}(S)|=4$.

Let $E_P$ be the $\pi$-exceptional curve that is mapped to the singular point $P\in\mathrm{Sing}(S)$.
Then $\mathrm{Stab}_P(G)\cong\mathfrak{A}_4$, and the group $\mathrm{Stab}_P(G)$ acts faithfully on the~exceptional curve~$E_P$.
Moreover, it is well known that the smallest $\mathrm{Stab}_P(G)$-orbit in $E_P\cong\mathbb{P}^1$ has length $4$.
Therefore, since the subset $E_P\cap\widetilde{C}$ is $\mathrm{Stab}_P(G)$-invariant, we conclude that $|E_P\cap\widetilde{C}|\geqslant 4$,
so that all irreducible components of the curve $\widetilde{C}$ pass through $P$. Then $k=4$ and
$$
12\big(\widetilde{C}_1\cdot \widetilde{C}_2\big)-8=\widetilde{C}^2=C^2-\frac{k^2}{2}|\mathrm{Sing}(S)|=C^2-32,
$$
which implies that $12(\widetilde{C}_1\cdot \widetilde{C}_2)+24=C^2$, which is impossible, since $C^2=28$ or $C^2=32$.
Hence, we conclude that $\mathrm{Sing}(S)$ is not a single $G$-orbit.

By Lemma~\ref{lemma:48-net}, $\mathrm{Sing}(S)$ is a union of a $G$-orbit of length $4$ and a $G$-orbit of length~$12$.
Therefore, we conclude that $E=E_{1}+E_{2}$, where $E_1$ and $E_{2}$ are two $G$-irreducible curves such that the~image $\pi(E_1)$ is a $G$-orbit of length $4$, and $\pi(E_{2})$ is a $G$-orbit of length $12$.
Take~two~points $P_1\in\pi(E_1)$ and $P_2\in\pi(E_{2})$.
Let $k_{1}$ and $k_{2}$ be the number of irreducible components of the curve $C$ that pass through the~points $P_1$ and $P_2$, respectively.
Then
$$
\widetilde{C}\sim_{\mathbb{Q}} \pi^*(C)-\frac{k_1}{2}E_1-\frac{k_2}{2}E_2,
$$
where $k_1>0$ or $k_2>0$. Then
$$
12\big(\widetilde{C}_1\cdot \widetilde{C}_2\big)-8=\widetilde{C}^2=C^2-2k_1^2-6k_2^2.
$$
If $k_1=0$, we obtain a contradiction exactly as in the case $|\mathrm{Sing}(S)|=12$, hence $k_1>0$.
Now, arguing as in the case $|\mathrm{Sing}(S)|=4$, we see that $k_1=4$, hence
$$
12\big(\widetilde{C}_1\cdot \widetilde{C}_2\big)+24+6k_2^2=C^2\in\{28,32\},
$$
which is impossible, since $C^2$ is not divisible by $6$.
This completes the proof.
\end{proof}

\section{Equivariant geometry of projective space: group of order 192}
\label{section:P3-192}

Let $G$ be the~subgroup in $\mathrm{PGL}_4(\mathbb{C})$ generated by
\begin{multline*}
M=\begin{pmatrix}
-1 & 0 & 0 & 0\\
0 & 1 & 0 & 0\\
0 & 0 & 1 & 0\\
0 & 0 & 0 & 1
\end{pmatrix},
N=\begin{pmatrix}
1 & 0 & 0 & 0\\
0 & -1 & 0 & 0\\
0 & 0 & 1 & 0\\
0 & 0 & 0 & 1
\end{pmatrix},
L=\begin{pmatrix}
1 & 0 & 0 & 0\\
0 & 1 & 0 & 0\\
0 & 0 & -1 & 0\\
0 & 0 & 0 & 1
\end{pmatrix},\\
A=\begin{pmatrix}
0 & 0 & 0 & i\\
1 & 0 & 0 & 0\\
0 & 1 & 0 & 0\\
0 & 0 & 1 & 0
\end{pmatrix},
B=\begin{pmatrix}
0 & 1 & 0 & 0\\
1 & 0 & 0 & 0\\
0 & 0 & i & 0\\
0 & 0 & 0 & 1
\end{pmatrix},
\end{multline*}
Then $G$ is the~subgroup $G_{192,185}\cong\mumu_2^3.\mathfrak{S}_4\cong\mumu_4^2\rtimes(\mumu_3\rtimes\mumu_4)$ introduced in Section~\ref{section:subgroups}.
Let
\begin{center}
$P_1=[1:0:0:0]$, $P_2=[0:1:0:0]$, $P_3=[0:0:1:0]$, $P_4=[0:0:0:1]$,
\end{center}
and let $\Sigma_4=P_1\cup P_2\cup P_3\cup P_4$.

\begin{lemma}
\label{lemma:192-orbits}
The subset $\Sigma_4$ the~unique $G$-orbit in $\mathbb{P}^3$ of length $\leqslant 15$.
\end{lemma}

\begin{proof}
Let $\Sigma$ be a $G$-orbit in $\mathbb{P}^3$ of length $\leqslant 15$.
Take a point $P\in\Sigma$. Let $G_P=\mathrm{Stab}_G(P)$.
Then $|G_P|\geqslant 16$.
Thus, using \cite{Tim}, we see that $G_P$ contains a~subgroup $\Gamma$ that is isomorphic to one of the~following groups:
$\mumu_4^2$, $\mumu_2^2\rtimes\mumu_4$ or $\mumu_4\rtimes\mumu_2^2$.

Suppose that $\Gamma\cong\mumu_4^2$. According to \cite{Tim},
the subgroup $\Gamma$ is normal, and $\Gamma$ is the~unique subgroup in $G$ that is isomorphic to $\mumu_4^2$.
Using this, we see that $\Gamma$ is generated by
$$
A^2BA^{-2}BL=\begin{pmatrix}
0 & -1 & 0 & 0 \\
i & 0 & 0 & 0\\
0 & 0 & 0 & i\\
 0 & 0 & 1 & 0
\end{pmatrix}
\ \text{and}\
A^3BA^{-2}BLA^{-1}=
\begin{pmatrix}
0 & 0 & 0 & i\\
0 & 0 & -1 & 0\\
0 & i & 0 & 0\\
1 & 0 & 0 & 0
\end{pmatrix}.
$$
Using this, one can show that $\mathbb{P}^3$ does not contain $\Gamma$-fixed points and $\Gamma$-invariant lines.
Thus, this case is impossible.

Now, we suppose that $\Gamma\cong\mumu_2^2\rtimes\mumu_4$.
According to \cite{Tim}, there are exactly two possibilities for the~subgroup $\Gamma$ up to conjugation,
which can be distinguished as follows:
\begin{enumerate}
\item either $\Gamma$ contains the~normal subgroup $\langle M,N,L\rangle\cong\mumu_2^3$,
\item or $\Gamma$ contains a non-normal subgroup isomorphic to $\mumu_2^3$.
\end{enumerate}
In the~first case, we may assume that $\Gamma$ is generated by $M$, $N$, $L$ and $B$,
which implies that the~only $\Gamma$-fixed points in $\mathbb{P}^3$ are the~points $P_3$ and $P_4$,
and the~only $\Gamma$-invariant lines are the~lines $\{x_0=x_1=0\}$ and $\{x_2=x_3=0\}$.
Similarly, in the~second case, we may assume that $\Gamma$ contains the~non-normal subgroup isomorphic to $\mumu_2^3$
that is generated by
$$
MN=\begin{pmatrix}
-1 & 0 & 0 & 0\\
0 & -1 & 0 & 0\\
0 & 0 & 1 & 0\\
0 & 0 & 0 & 1
\end{pmatrix}, ML=
\begin{pmatrix}
-1 & 0 & 0 & 0\\
0 & 1 & 0 & 0\\
0 & 0 & -1 & 0\\
0 & 0 & 0 & 1
\end{pmatrix},
A^2=
\begin{pmatrix}
0 & 0 & i & 0\\
0 & 0 & 0 & i\\
1 & 0 & 0 & 0\\
0 & 1 & 0 & 0
\end{pmatrix}.
$$
Observe that this subgroup does not fix any point in $\mathbb{P}^3$, and it leaves invariant exactly two lines: the~lines
$\{x_0=x_2=0\}$ and $\{x_1=x_3=0\}$.
In particular, the~group $\Gamma$ does not fix points in $\mathbb{P}^3$ either.
Hence, if $\Gamma\cong\mumu_2^2\rtimes\mumu_4$, then $P\in\Sigma_4$, hence $\Sigma=\Sigma_4$.

To complete the proof, we may assume that $\Gamma\cong\mumu_4\rtimes\mumu_2^2$.
Using \cite{Tim} again, we see that the~group $\Gamma$  contains a non-normal subgroup that is isomorphic to $\mumu^3$.
Hence, arguing as in the~previous~case, we conclude that $P\in\Sigma_4$, hence $\Sigma=\Sigma_4$ as required.
\end{proof}

As in Section~\ref{section:P3-48}, let $\ell_{ij}$ be the~line in $\mathbb{P}^3$ that contains $P_i$ and $P_j$, where $1\leqslant i<j\leqslant 4$.
Similarly, we let
$\mathcal{L}_6=\ell_{12}+\ell_{13}+\ell_{14}+\ell_{23}+\ell_{24}+\ell_{34}$ and $\mathcal{T}=F_1+F_2+F_3+F_4$,
where
$$
F_1=\{x_0=0\},\, F_2=\{x_1=0\},\, F_3=\{x_2=0\},\, F_4=\{x_3=0\}.
$$
The main result of this section is the~following

\begin{proposition}
\label{proposition:P3-192-curves}
Let $C$ be a $G$-irreducible (possibly reducible) curve in $\mathbb{P}^3$ of degree~$\leqslant 15$.
Then $C$ is one of the~following seven $G$-irreducible curves:
\begin{enumerate}[\normalfont(1)]
\item the~reducible curve $\mathcal{L}_6$,
\item the~reducible curve $\mathcal{C}_8\subset\mathcal{T}$ that is a disjoint union of $4$ conics
\begin{align*}
\big\{x_0=x_1^2-x_2^2-x_3^2=0\big\},\\
\big\{x_1=x_0^2+x_2^2-x_3^2=0\big\}, \\
\big\{x_2=x_0^2+x_1^2+x_3^2=0\big\}, \\
\big\{x_3=x_0^2-x_1^2-x_2^2=0\big\},
\end{align*}
\item the~reducible curve $\mathscr{C}_8$ that is a disjoint union of $2$ smooth quartic elliptic curves
\begin{align*}
\big\{x_0^2+\big(\zeta_6-1\big)x_2^2+\zeta_6x_3^2=x_1^2+\zeta_6x_2^2+\big(1-\zeta_6\big)x_3^2=0\big\},\\
\big\{x_0^2-\zeta_6x_2^2+\big(1-\zeta_6\big)x_3^2=x_1^2+\big(1-\zeta_6\big)x_2^2+\zeta_6x_3^2=0\big\},
\end{align*}
where $\zeta_6$ is a primitive sixth roon of unity,
\item the~reducible curve $\mathscr{C}_{12}$ that is a disjoint union of $3$ smooth quartic elliptic~curves
\begin{align*}
\big\{x_0^2+\sqrt{2}ix_1^2-x_2^2=x_1^2+\sqrt{2}ix_2^2-x_3^2=0\big\},\\
\big\{\sqrt{2}ix_0^2-x_1^2-x_3^2=x_0^2+\sqrt{2}ix_1^2-x_2^2=0\big\},\\
\big\{x_0^2+x_2^2+\sqrt{2}ix_3^2=\sqrt{2}ix_0^2-x_1^2-x_3^2\big\},
\end{align*}
\item the~reducible curve $\mathscr{C}_{12}^\prime$ that is a disjoint union of $3$ smooth quartic elliptic~curves
\begin{align*}
\big\{x_0^2-\sqrt{2}ix_1^2-x_2^2=x_1^2-\sqrt{2}ix_2^2-x_3^2=0\big\},\\
\big\{\sqrt{2}ix_0^2+x_1^2+x_3^2=x_0^2-\sqrt{2}ix_1^2-x_2^2=0\big\},\\
\big\{x_0^2+x_2^2-\sqrt{2}ix_3^2=\sqrt{2}ix_0^2+x_1^2+x_3^2\big\},
\end{align*}

\item the~irreducible smooth curve $\mathfrak{C}_{12}$ of degree $12$ and genus $17$ that is given by
\begin{align*}
(2+2\sqrt{2}i)(x_1^2x_2^2-x_0^2x_1^2-x_0^2x_3^2-x_2^2x_3^2)+3(x_0^4-x_1^4+x_2^4-x_3^4)=0,\\
(2+2\sqrt{2}i)(x_2^2x_3^2-x_0^2x_1^2-x_0^2x_2^2-x_1^2x_3^2)-3(x_0^4-x_1^4-x_2^4+x_3^4)=0,\\
(2+2\sqrt{2}i)(x_1^2x_3^2-x_0^2x_2^2+x_0^2x_3^2+x_1^2x_2^2)+3(x_0^4+x_1^4-x_2^4-x_3^4)=0,
\end{align*}

\item the~irreducible smooth curve $\mathfrak{C}_{12}^\prime$ of degree $12$ and genus $17$ that is given by
\begin{align*}
(2-2\sqrt{2}i)(x_1^2x_2^2-x_0^2x_1^2-x_0^2x_3^2-x_2^2x_3^2)+3(x_0^4-x_1^4+x_2^4-x_3^4)=0,\\
(2-2\sqrt{2}i)(x_2^2x_3^2-x_0^2x_1^2-x_0^2x_2^2-x_1^2x_3^2)-3(x_0^4-x_1^4-x_2^4+x_3^4)=0,\\
(2-2\sqrt{2}i)(x_1^2x_3^2-x_0^2x_2^2+x_0^2x_3^2+x_1^2x_2^2)+3(x_0^4+x_1^4-x_2^4-x_3^4)=0.
\end{align*}
\end{enumerate}
\end{proposition}

\begin{corollary}
\label{corollary:P3-192-curves}
Let $C$ be a $G$-irreducible curve in $\mathbb{P}^3$ such that $C$ is different from $\mathcal{L}_6$,
and let $\mathcal{D}$ be a linear subsystem in $|\mathcal{O}_{\mathbb{P}^3}(n)|$ that has no fixed components, where $n\in\mathbb{Z}_{>0}$.
Then $\mathrm{mult}_C(\mathcal{D})\leqslant\frac{n}{4}$.
\end{corollary}

\begin{proof}
Arguing as in the proof of Proposition~\ref{proposition:48-curves}, we may assume that $\mathrm{deg}(C)\leqslant 15$.
Thus, we conclude that $C$ is one of the $G$-irreducible curves described in Proposition~\ref{proposition:P3-192-curves},
which are different from $\mathcal{L}_6$.
Moreover, arguing as in the proof of Proposition~\ref{proposition:48-curves} again,
we obtain the required inequality if $C=\mathcal{C}_8$.
Thus, we may also assume that $C\ne\mathcal{C}_8$.
Then it follows from Proposition~\ref{proposition:P3-192-curves} that the curve $C$ is smooth,
but it maybe reducible.

Let $\varphi\colon X\to\mathbb{P}^3$ be the blow up of the smooth curve $C$, let $E_C$ be the $\varphi$-exceptional divisor,
let $\widehat{\mathcal{D}}$ be the proper transform on $X$ of the linear system $\mathcal{D}$,
let $\widehat{D}_1$ and $\widehat{D}_2$ be~two general surfaces in $\widehat{\mathcal{D}}$.
Then $\widehat{D}_1\cdot\widehat{D}_2$ is an effective one-cycle.
On the other hand, it follows from Proposition~\ref{proposition:P3-192-curves} that the linear system $|\varphi^*(\mathcal{O}_{\mathbb{P}^3}(k))-E_C|$ is base point free for
$$
k=\left\{\aligned
&4\ \text{if $C=\mathscr{C}_8$ or $C=\mathfrak{C}_{12}$ or $C=\mathfrak{C}_{12}^\prime$},\\
&6\ \text{if $C=\mathscr{C}_{12}$ or $C=\mathscr{C}_{12}^\prime$}.\\
\endaligned
\right.
$$
In particular, the divisor $\varphi^*(\mathcal{O}_{\mathbb{P}^3}(k))-E_C$ is nef.
Then
\begin{multline*}
0\leqslant\Big(\varphi^*\big(\mathcal{O}_{\mathbb{P}^3}(k)\big)-E_C\Big)\cdot\widehat{D}_1\cdot\widehat{D}_2=\Big(\varphi^*\big(\mathcal{O}_{\mathbb{P}^3}(k)\big)-E_C\Big)\cdot\Big(\varphi^*\big(\mathcal{O}_{\mathbb{P}^3}(n)\big)-\mathrm{mult}_C\big(\mathcal{D}\big)E_C\Big)^2=\\
=\big(-E^3-k\mathrm{deg}(C)\big)\mathrm{mult}_C^2\big(\mathcal{D}\big)-2n\mathrm{deg}(C)\mathrm{mult}_C\big(\mathcal{D}\big)+kn^2,
\end{multline*}
where
$$
E^3=\left\{\aligned
&-32\ \text{if $C=\mathscr{C}_8$},\\
&-48\ \text{if $C=\mathscr{C}_{12}$ or $C=\mathscr{C}_{12}^\prime$},\\
&-80\ \text{if $C=\mathfrak{C}_{12}$ or $C=\mathfrak{C}_{12}^\prime$}.\\
\endaligned
\right.
$$
This implies that $\mathrm{mult}_C\big(\mathcal{D}\big)\leqslant\frac{n}{4}$.
\end{proof}

Let us prove Proposition~\ref{proposition:P3-192-curves}.
Fix a $G$-irreducible curve $C\subset\mathbb{P}^3$.
Write $C=C_1+\cdots+C_r$, where each $C_i$ is an~irreducible curve in the~space $\mathbb{P}^3$, and $r$ is the~number of irreducible components of the~curve $C$.
Let $d$ be the~degree of the~curve $C_1$. Then $\mathrm{deg}(C)=rd$.
Suppose that $d\leqslant 15$.
Let us show that $C$ is one of the~curves listed in  Proposition~\ref{proposition:P3-192-curves}.

\begin{lemma}
\label{lemma:192-d-1}
If $d=1$, then $C=\mathcal{L}_6$.
\end{lemma}

\begin{proof}
The required assertion follows from the~proof of Lemma~\ref{lemma:192-orbits}.
\end{proof}

Hence, to complete the~proof of Proposition~\ref{proposition:P3-192-curves}, we may assume $d\geqslant 2$.

\begin{lemma}
\label{lemma:192-curves-in-T}
If $C\subset\mathcal{T}$, then either $C=\mathcal{L}_6$ or $C=\mathcal{C}_8$.
\end{lemma}

\begin{proof}
Left to the~reader.
\end{proof}

Therefore, we may assume that $C\not\subset\mathcal{T}$.
Then, using Lemma~\ref{lemma:192-orbits}, we conclude that no irreducible component of the~curve $C$ is contained in a plane.
In particular, we have~$d\geqslant 3$. Since $dr\leqslant 15$, we have the~following possibilities:
\begin{enumerate}
\item $r=1$ and $C=C_1$ is an~irreducible curve,

\item $r=2$ and each $C_i$ is an irreducible curve of degree $d\in\{3,4,5,6,7\}$,

\item $r=3$ and each $C_i$ is an irreducible curve of degree $d\in\{3,4,5\}$,

\item $r=4$ and each $C_i$ is a smooth rational cubic curve.
\end{enumerate}

\begin{lemma}
\label{lemma:192-r-4}
One has $r\ne 4$.
\end{lemma}

\begin{proof}
If $r=4$, the~stabilizer of the~curve $C_1$ is a group of order $48$.
According to \cite{Tim}, any~subgroup of the~group $G$ of order $48$ is isomorphic either to $\mathfrak{A}_4\rtimes\mumu_4$ or to $\mumu_4^2\rtimes\mumu_3$.
But~none of these groups can act faithfully on a rational curve, since
$\mathrm{PGL}_2(\mathbb{C})$ does not contain groups isomorphic to $\mathfrak{A}_4\rtimes\mumu_4$ or $\mumu_4^2\rtimes\mumu_3$.
Hence, we conclude that $r\ne 4$.
\end{proof}

Now, let us fix the~subgroup $\Gamma\subset G$ that is generated by
$$
A=\begin{pmatrix}
0 & 0 & 0 & i\\
1 & 0 & 0 & 0\\
0 & 1 & 0 & 0\\
0 & 0 & 1 & 0
\end{pmatrix},
A^2BA^2BL=\begin{pmatrix}
0 & -1 & 0 & 0 \\
i & 0 & 0 & 0\\
0 & 0 & 0 & i\\
 0 & 0 & 1 & 0
\end{pmatrix},
A^3BA^2BLA^3=
\begin{pmatrix}
0 & 0 & 0 & i\\
0 & 0 & -1 & 0\\
0 & i & 0 & 0\\
1 & 0 & 0 & 0
\end{pmatrix}.
$$
Using \cite{Tim}, we conclude that $\Gamma\cong\mumu_4^2\rtimes\mumu_4$,
and the~GAP ID of the~subgroup $\Gamma$ is [64,34].
Note that $\mathbb{P}^3$ contain neither $\Gamma$-fixed points nor $\Gamma$-invariant lines by Lemmas~\ref{lemma:192-orbits} and \ref{lemma:192-d-1}.
Moreover, according to \cite{Tim}, the~group $G$ contains $3$ subgroups that are isomorphic to $\Gamma$, and all of them are conjugated.

\begin{lemma}
\label{lemma:192-r-3}
Suppose that $r=3$. Then either $C=\mathscr{C}_{12}$ or $C=\mathscr{C}_{12}^\prime$.
\end{lemma}

\begin{proof}
The subgroup $\Gamma$ is a stabilizer of $C_1$, $C_2$ or $C_3$.
Without loss of generality, we may assume that $C_1$ is $\Gamma$-invariant.
The group $\Gamma$ acts faithfully on $C_1$.
This implies that $d\ne 3$, because $\Gamma$ cannot leave invariant smooth rational cubic curve,
since $\mathrm{PGL}_2(\mathbb{C})$ does not contain groups isomorphic to $\Gamma$.

Now, we claim that $d\ne 5$. Indeed, suppose that $d=5$.
Then the~curve $C_1$ is smooth.
Namely, if $C_1$ is singular, then it contains at least $4$ singular points,
so that, intersecting the~curve $C_1$ with a plane passing through $3$ of them, we conclude that $C_1$ is contained in this plane,
which contradicts our assumption.
Thus, we see that $C_1$ is smooth, so~that its genus does not exceed $2$ by \cite[Theorem 6.4]{Hartshorne}.
But the~order of the~automorphism group of a smooth curve of genus $2$ does not exceed $48$,
and, as we already mentioned, the~group $\Gamma$ cannot faithfully act on a rational curve.
Thus, we see that $C_1$ is a smooth elliptic curve.
Then the~$\Gamma$-action on $C_1$ gives an embedding
$$
\Gamma\hookrightarrow\mathrm{Aut}\big(C_1,\mathcal{O}_{\mathbb{P}^3}(1)\vert_{C_1}\big).
$$
This is impossible, since the~order of the~group $\mathrm{Aut}(C_1,\mathcal{O}_{\mathbb{P}^3}(1)\vert_{C_1})$ is not divisible by $64$,
because $\mathrm{Aut}(C_1,\mathcal{O}_{\mathbb{P}^3}(1)\vert_{C_1})$ is an extension of the~group $\mumu_5^2$ by one of the~following cyclic groups: $\mumu_2$, $\mumu_6$ or $\mumu_4$.
Hence, we see that $d\ne 5$.

Thus, we see that $d=4$. As above, we see that  $C_1$, $C_2$, $C_3$ are smooth elliptic curves,
which implies that each of them is a complete intersection of two quadric surfaces in~$\mathbb{P}^3$.
Hence, there exists a $\Gamma$-invariant pencil of quadric surfaces in $\mathbb{P}^3$ whose base locus is $C_1$.
On the~other hand, it is not hard to find all $\Gamma$-invariant pencils of quadric surfaces in $\mathbb{P}^3$.
Namely, let $\widehat{\Gamma}$ be the the~subgroup in $\mathrm{GL}_4(\mathbb{C})$ that is generated by the~matrices
$$
\begin{pmatrix}
0 & 0 & 0 & i\\
1 & 0 & 0 & 0\\
0 & 1 & 0 & 0\\
0 & 0 & 1 & 0
\end{pmatrix},
\begin{pmatrix}
0 & -1 & 0 & 0 \\
i & 0 & 0 & 0\\
0 & 0 & 0 & i\\
 0 & 0 & 1 & 0
\end{pmatrix},
\begin{pmatrix}
0 & 0 & 0 & i\\
0 & 0 & -1 & 0\\
0 & i & 0 & 0\\
1 & 0 & 0 & 0
\end{pmatrix}.
$$
Then $\Gamma$ is the~image of the~group $\widehat{\Gamma}$ via the~natural projection $\mathrm{GL}_4(\mathbb{C})\to\mathrm{PGL}_4(\mathbb{C})$,
and the~GAP ID of the~group $\widehat{\Gamma}$ is [256,420].
Now, going through all irreducible $4$-dimensional representations of the~group $\widehat{\Gamma}$ in GAP \cite{GAP},
and checking their symmetric squares, we see that $\mathbb{P}^3$ contains three $\Gamma$-invariant pencils of quadrics.
These pencils are
\begin{enumerate}[(i)]
\item $\lambda(x_0^2+\sqrt{2}ix_1^2-x_2^2)=\mu(x_1^2+\sqrt{2}ix_2^2-x_3^2)$,
\item $\lambda(x_0^2-\sqrt{2}ix_1^2-x_2^2)=\mu(x_1^2-\sqrt{2}ix_2^2-x_3^2)$,
\item $\lambda x_0x_2=\mu x_1x_3$,
\end{enumerate}
where $[\lambda:\mu]\in\mathbb{P}^1$. In case (iii), the~base locus of the pencil is the union $\ell_{12}\cup\ell_{14}\cup\ell_{23}\cup\ell_{24}$.
In case (i), the~base locus  of the pencil is $\mathscr{C}_{12}$.
Finally, in case (ii), the~base locus is $\mathscr{C}_{12}^\prime$.
Hence, we conclude that either $C=\mathscr{C}_{12}$ or $\mathscr{C}_{12}^\prime$.
\end{proof}

To complete the~proof of Proposition~\ref{proposition:P3-192-curves}, we may assume that $r\ne 3$.
Then $r\in\{1,2\}$.
Observe that the~group $G$ contains unique subgroup of index two ---
the normal subgroup isomorphic to $\mumu_2^3.\mathfrak{A}_4\cong\mumu_4^2\rtimes\mumu_6$.
This subgroup does not contain $\Gamma$.
Therefore, if $r=2$, then $\Gamma$ swaps the~curves $C_1$ and $C_2$.
Thus, we see that $C$ is $\Gamma$-irreducible.

Note that $\Gamma$ leaves invariant $\mathcal{T}$ and the~Fermat quartic  $\{x_0^4+x_1^4+x_2^4+x_3^4=0\}\subset\mathbb{P}^3$.
These are not all $\Gamma$-invariant quartic surfaces.
Namely, the~group $\Gamma$ leaves invariant every surface in the~pencil $\mathcal{P}$ given by
$$
\lambda(x_1^2x_2^2-x_0^2x_1^2-x_0^2x_3^2-x_2^2x_3^2)+\mu(x_0^4-x_1^4+x_2^4-x_3^4)=0,
$$
where $[\lambda:\mu]\in\mathbb{P}^1$. One can show that these are all $\Gamma$-invariant surfaces in $\mathbb{P}^3$.

Let $P$ be a general point in $C$, let $\Sigma_P$ be its $\Gamma$-orbit,
and let $S$ be a surface in $\mathcal{P}$ that passes through $P$.
Then $|\Sigma_P|=64$, which implies that $C\subset S$. Indeed, if $C\not\subset S$, then
$$
60\geqslant 4rd=S\cdot C\geqslant |\Sigma_P|=64,
$$
which is absurd, hence $C\subset S$.
Let $a$ and $b$ be complex numbers such that $S$ is given by
$$
a(x_1^2x_2^2-x_0^2x_1^2-x_0^2x_3^2-x_2^2x_3^2)+b(x_0^4-x_1^4+x_2^4-x_3^4)=0.
$$
Note that the~surface $S$ is not $G$-invariant, because the~only $G$-invariant quartic surfaces
are the~surfaces $\mathcal{T}$ and $\{x_0^4+x_1^4+x_2^4+x_3^4=0\}$.
But $C$ is $G$-invariant by assumption.
Thus, using the~$G$-action, we see that $C$ is contained in the~subset in $\mathbb{P}^3$ given by
\begin{equation}
\label{equation:P1-P3}
\left\{\aligned
&a(x_1^2x_2^2-x_0^2x_1^2-x_0^2x_3^2-x_2^2x_3^2)+b(x_0^4-x_1^4+x_2^4-x_3^4)=0,\\
&a(x_2^2x_3^2-x_0^2x_1^2-x_0^2x_2^2-x_1^2x_3^2)-b(x_0^4-x_1^4-x_2^4+x_3^4)=0,\\
&a(x_1^2x_3^2-x_0^2x_2^2+x_0^2x_3^2+x_1^2x_2^2)+b(x_0^4+x_1^4-x_2^4-x_3^4)=0.\\
\endaligned
\right.
\end{equation}

\begin{lemma}
\label{lemma:192-final}
Either $3a^2-4ab+4b^2=0$ or $a+2b=0$.
\end{lemma}

\begin{proof}
Note that  the~subset \eqref{equation:P1-P3} in $\mathbb{P}^3$ is zero-dimensional for a general choice of $a$ and~$b$.
To find all possible values of $a$ and $b$ such that \eqref{equation:P1-P3} is not zero-dimensional,
one can consider the~subscheme in $\mathbb{P}^1\times\mathbb{P}^3$ defined over $\mathbb{Q}$ that is given by \eqref{equation:P1-P3},
where $a$ and $b$ are considered as coordinates on $\mathbb{P}^1$.
Using Magma, we see that this subscheme is reduced and one-dimensional,
and we also find all its irreducible (over $\mathbb{Q}$) components.

Going through these irreducible components and checking which one is mapped to a~zero-dimensional subscheme of $\mathbb{P}^1$
via the~natural projection $\mathbb{P}^1\times\mathbb{P}^3\to\mathbb{P}^1$, we see that
the subset \eqref{equation:P1-P3} contains a curve if and only if either $3a^2-4ab+4b^2=0$ or $a+2b=0$.
\end{proof}

If $3a^2-4ab+4b^2=0$, we may assume that $b=3$ and $a^2-4a+12$.
In this case, the~subscheme in $\mathbb{P}^3$ given by \eqref{equation:P1-P3} is a smooth irreducible curve of degree $12$ and genus~$17$.
This can be checked using Magma. Now, taking two roots of the~quadratic $a^2-4a+12$,
we get the~curves $\mathfrak{C}_{12}$~and~$\mathfrak{C}_{12}^\prime$.
One can check that these curves are disjoint.

Finally, if $a+2b=0$, the~subscheme in $\mathbb{P}^3$ given by \eqref{equation:P1-P3}
splits as a disjoint union of the~$\Gamma$-irreducible curves $\mathscr{C}_8$ and $\mathscr{C}_8^\prime$.
This completes the~proof of the~Proposition~\ref{proposition:P3-192-curves}.

\section{Equivariant geometry of projective space: large groups}
\label{section:P3-large}

Let us use assumptions and notations of Section~\ref{section:subgroups}.
Recall from Section~\ref{section:subgroups} that
\begin{center}
$P_1=[1:0:0:0]$, $P_2=[0:1:0:0]$, $P_3=[0:0:1:0]$, $P_4=[0:0:0:1]$,
\end{center}
and $G$ is a finite subgroup in $\mathrm{PGL}_{4}(\mathbb{C})$ such that the~following conditions are satisfied:
\begin{enumerate}
\item the~group $G$ does not have fixed points in  $\mathbb{P}^3$,
\item the~group $G$ does not leave a union of two skew lines in $\mathbb{P}^3$ invariant,
\item the~group $G$ leaves invariant the subset $\{P_1,P_2,P_3,P_4\}$.
\end{enumerate}
Recall from Section~\ref{section:subgroups} that $\upsilon\colon G\to\mathfrak{S}_4$ is the homomorphism
induced by the~$G$-action on the~set $\{P_1,P_2,P_3,P_4\}$, and $T$ is the~kernel of this homomorphism.
Then $T$ is not trivial, and either the~homomorphism $\upsilon$ is surjective, or its image is $\mathfrak{A}_4$.
Suppose, in addition, that the group $G$ is not conjugate to any of the~following eight subgroups:
\begin{center}
$G_{48,50}$, $G_{48,3}$, $G_{96,70}$, $G_{96,72}$, $G_{96,227}$, $G_{96,227}^\prime$, $G_{192,955}$, $G_{192,185}$.
\end{center}
Moreover, if $G$ is conjugate to any subgroup among $G_{324,160}$, $G_{324,160}^\prime$, $G_{648,704}$ or $G_{648,704}^\prime$,
we will always assume that $G$ is this subgroup.

For every $1\leqslant i<j\leqslant 4$, let $\ell_{ij}$ be the~line in $\mathbb{P}^3$ that passes through $P_i$ and $P_j$.
Let
$$
F_1=\{x_0=0\},\, F_2=\{x_1=0\},\, F_3=\{x_2=0\},\, F_4=\{x_3=0\}.
$$
Let $\Sigma_4=\{P_1,P_2,P_3,P_4\}$, let $\mathcal{L}_6=\ell_{12}+\ell_{13}+\ell_{14}+\ell_{23}+\ell_{24}+\ell_{34}$,
let $\mathcal{T}=F_1+F_2+F_3+F_4$.

\begin{lemma}
\label{lemma:orbits}
Let $\Sigma$ be a $G$-orbit in $\mathbb{P}^3$. Then
$$
|\Sigma|\geqslant \left\{\aligned
&|T|\ \text{if $\Sigma\not\subset\mathcal{T}$},\\
&4n^2\ \text{if $\Sigma\subset\mathcal{T}\setminus\mathcal{L}_6$},\\
&6n\ \text{if $\Sigma\subset\mathcal{L}_6\setminus\Sigma_4$}.\\
\endaligned
\right.
$$
\end{lemma}

\begin{proof}
The required assertion follows from the~explicit description of the~subgroup $T$,
which has been given in the~proofs of Lemmas~\ref{lemma:S4} and \ref{lemma:A4}.
\end{proof}

Let $C$ be a $G$-irreducible curve in $\mathbb{P}^3$ of degree $d\leqslant 15$.
Our goal is to classify all possibilities for the curve $C$.
Firstly, we show that $C\subset\mathcal{T}$.

\begin{lemma}
\label{lemma:P3-curve-vertices}
Suppose that $C\not\subset\mathcal{T}$. Then $\Sigma_4\not\subset C$.
\end{lemma}

\begin{proof}
We suppose that $C$ contains $\Sigma_4$. Let $\sigma\colon X\to \mathbb{P}^3$ be the~blow up of the~$G$-orbit $\Sigma_4$,
let $G_i$ be the~$\sigma$-exceptional surface that is mapped to the~point $P_i$,
let $\widetilde{F}_i$ be the proper transform on $X$ of the plane $F_i$,
let $\widetilde{C}$ be the proper transform on $X$ of the curve $C$,
and let $\widetilde{\ell}_{ij}$ be the proper transform on $X$ of the line $\ell_{ij}$.
Then the $G$-action lifts to $X$, the~curve $\widetilde{C}$ is $G$-invariant, and
$$
\widetilde{F}_4\cdot\widetilde{C}=\Big(\sigma^*\big(F_4\big)-G_1-G_2-G_3\Big)\cdot\widetilde{C}=d-3\widetilde{C}\cdot G_1\leqslant 15-3\widetilde{C}\cdot G_1,
$$
so that $1\leqslant |\widetilde{C}\cap G_1|\leqslant\widetilde{C}\cdot G_1\leqslant 5$.

The surface $G_1$ is $\mathrm{Stab}_G(P_4)$-invariant,
and the induces $\mathrm{Stab}_G(P_4)$-action on it is faithful.
Moreover, the surface $G_1\cong\mathbb{P}^2$ does not contain $\mathrm{Stab}_G(P_4)$-orbits of length $1$, $2$, $4$, $5$,
and the only $\mathrm{Stab}_G(P_4)$-orbit of length $3$ is formed by the points $G_1\cap\widetilde{\ell}_{12}$, $G_1\cap\widetilde{\ell}_{13}$ and $G_1\cap\widetilde{\ell}_{14}$.
Thus, we conclude that  $|\widetilde{C}\cap G_1|=\widetilde{C}\cdot G_1=3$,
and $\widetilde{C}$ intersects the surface $G_1$ transversally in the points $G_1\cap\widetilde{\ell}_{12}$, $G_1\cap\widetilde{\ell}_{13}$ and $G_1\cap\widetilde{\ell}_{14}$.
Similarly, we see that the curve $\widetilde{C}$ intersects the surface $G_2$ transversally
in the points $G_2\cap\widetilde{\ell}_{12}$, $G_2\cap\widetilde{\ell}_{23}$ and $G_2\cap\widetilde{\ell}_{24}$,
and $\widetilde{C}$ intersects the surface $G_3$ transversally in the points
$G_3\cap\widetilde{\ell}_{13}$, $G_1\cap\widetilde{\ell}_{23}$ and $G_1\cap\widetilde{\ell}_{34}$.

Note that $\widetilde{F}_4$ is a smooth del Pezzo surface of degree $6$,
and its $(-1)$-curves are $\widetilde{\ell}_{12}$, $\widetilde{\ell}_{13}$, $\widetilde{\ell}_{23}$,
$G_1\cap\widetilde{F}_4$, $G_2\cap\widetilde{F}_4$, $G_3\cap\widetilde{F}_4$.
Note also that the curves  $\widetilde{\ell}_{12}$, $\widetilde{\ell}_{13}$, $\widetilde{\ell}_{23}$ are pairwise disjoint,
and each of them contains at least two points of the intersection $\widetilde{F}_4\cap\widetilde{C}$.
This gives
$$
6\leqslant|\widetilde{F}_4\cap\widetilde{C}|\leqslant\widetilde{F}_4\cdot\widetilde{C}=\Big(\sigma^*\big(F_4\big)-G_1-G_2-G_3\Big)\cdot\widetilde{C}=d-3\widetilde{C}\cdot G_1=d-9\leqslant 6.
$$
Thus, we conclude that $d=15$, $|\widetilde{F}_4\cap\widetilde{C}|=\widetilde{F}_4\cdot\widetilde{C}$=6,
and $\widetilde{C}$ intersects $\widetilde{F}_4$
transversally at the~points $G_1\cap\widetilde{\ell}_{12}$, $G_1\cap\widetilde{\ell}_{13}$,
$G_2\cap\widetilde{\ell}_{12}$, $G_2\cap\widetilde{\ell}_{23}$, $G_3\cap\widetilde{\ell}_{13}$, $G_1\cap\widetilde{\ell}_{23}$.
In particular, we see that the~curve $\widetilde{C}$ is smooth at these six intersection points.
Note also that $C\cap\mathcal{T}=\Sigma_4$.

Let $P=G_1\cap\widetilde{\ell}_{12}$.
Then $T\subset\mathrm{Stab}_G(P)$, and $T$ is not cyclic by Lemmas~\ref{lemma:S4} and \ref{lemma:A4}.
In~particular, we conclude that $\mathrm{Stab}_G(P)$ is not cyclic.
This implies that $C$ is reducible.
Indeed, if $C$ were irreducible, then $\mathrm{Stab}_G(P)$ would act faithfully on $C$,
so it would act faithfully on $\widetilde{C}$,
which would imply that $\mathrm{Stab}_G(P)$ is cyclic \cite[Lemma~2.7]{FZ},
because the curve $\widetilde{C}$ is smooth at the point $P$. Contradiction.

Let $C=C_1+\cdots+C_{r}$, where $r$ is the number of irreducible components of the curve~$C$, and each $C_i$ is an irreducible curve.
Since  $d=15$,  one of the following cases holds:
\begin{itemize}
\item $r=15$ and each $C_i$ is a line;
\item $r=5$ and each $C_i$ is a cubic curve;
\item $r=3$ and each $C_i$ is a quintic curve.
\end{itemize}
Let $k$ be the number of irreducible components of the curve $C$ that passes through $P_1$,
and let $l$ be the numbers of points in $\Sigma_{4}$ that are contained in $C_1$. Then
$$
4k=rl,
$$
so that $r=k=3$ and $l=4$, i.e. $C$ is a union of three irreducible quintic curves $C_1$,~$C_2$,~$C_3$,
and each of these quintic curves contains $\Sigma_{4}$.
In particular, these curves are not planar.
Moreover, since $\widetilde{C}\cdot G_1=3$, we conclude that $C_1$, $C_2$, $C_3$ are smooth at $P_1$,
so that these curves are smooth at the points of the $G$-orbit $\Sigma_{4}$.

The group $\mathrm{Stab}_{G}(C_1)$ acts faithfully on $C_1$, so $T\not\subset\mathrm{Stab}_{G}(C_1)$ by \cite[Lemma~2.7]{FZ},
because the group $T$ fixes the~point $P_1$, but the group $T$ is not cyclic.
Therefore, since $\mathrm{Stab}_{G}(C_1)$ is a subgroup in $G$ of index $3$, we conclude that
$$
\upsilon\big(\mathrm{Stab}_{G}(C_1)\big)=\mathrm{im}(\upsilon),
$$
where $\upsilon\colon G\to\mathfrak{S}_4$ is the group homomorphism induced by the $G$-action on the~set $\Sigma_4$.
Thus, we see that the~group $\mathrm{Stab}_{G}(C_1)$ acts transitively on the points of the~$G$-orbit $\Sigma_4$,
and the stabilizer in $\mathrm{Stab}_{G}(C_1)$ of the plane $F_4$ acts transitively on the set $\{P_1,P_2,P_3\}$.
On the other hand, we know that $C\cap\mathcal{T}=\Sigma_4$, hence $C\cap F_4=P_1\cup P_2\cup P_3$. Then
$$
5=F_4\cdot C_1=\big(F_4\cdot C_1)_{P_1}+\big(F_4\cdot C_1)_{P_2}+\big(F_4\cdot C_1)_{P_3}=3\big(F_4\cdot C_1)_{P_1},
$$
which is absurd. The obtained contradiction completes the proof of the lemma.
\end{proof}

\begin{lemma}
\label{lemma:P3-curves-edges}
Suppose that $C\not\subset\mathcal{T}$. Then $C\cap\mathcal{L}_6=\varnothing$.
\end{lemma}

\begin{proof}
Suppose  $C\cap\mathcal{L}_6\ne\varnothing$. Let $k=|C\cap\ell_{12}|$.
By Lemma~\ref{lemma:P3-curve-vertices}, $\Sigma_4\not\subset C$, hence $k\geqslant n\geqslant 3$.
Then $C\cap F_4\subset\ell_{12}\cup\ell_{13}\cup\ell_{23}$.
Indeed, if $F_4\cap C$ contains a point $P\not\in\ell_{12}\cup\ell_{13}\cup\ell_{23}$, then
\begin{multline*}
15\geqslant d=F_4\cdot C\geqslant |F_4\cap C|\geqslant|C\cap\ell_{12}|+|C\cap\ell_{13}|+|C\cap\ell_{23}|+|\mathrm{Orb}_{\Gamma}(P)|=\\
=3k+|\mathrm{Orb}_{\Gamma}(P)|\geqslant 3k+n^2\geqslant 3n+n^2\geqslant 18,
\end{multline*}
where $\Gamma=\mathrm{Stab}_{G}(F_{4})$. Similarly, we see that $3\leqslant n\leqslant k\leqslant 5$, and the~curve $C$ is smooth at the points of the intersection $C\cap\ell_{12}$.
Therefore, we conclude that $C\cap\mathcal{T}=C\cap\mathcal{L}_6$, and the~curve $C$ is smooth at the points of the intersection $C\cap\mathcal{L}_{6}$.

Let $P$ be a point in $C\cap\ell_{12}$. Since $C$ is smooth at $P$, there exists a unique irreducible component of the curve $C$
that contains $P$. Denote this curve by $Z$. Let $K=\mathrm{Stab}_G(Z)$. Then $Z$ is smooth~at~$P$, and $\mathrm{Stab}_G(P)=\mathrm{Stab}_K(P)$, hence $Z$ is $\mathrm{Stab}_G(P)$-invariant.
However, if $n\in\{3,5\}$, then it follows from the proofs of Lemmas~\ref{lemma:S4} and~\ref{lemma:A4} that
$$
T\cap\mathrm{Stab}_G(P)\cong\mumu_n^2,
$$
and $T\cap\mathrm{Stab}_G(P)$ acts faithfully on $Z$, because $Z$ is not contained in $\mathcal{T}$ by assumption.
This is impossible by \cite[Lemma~2.7]{FZ}, since $Z$ is smooth at $P$.
Hence, we have $n=4$.

Arguing as above and using the proofs of Lemmas~\ref{lemma:S4} and~\ref{lemma:A4}, we see that $T\cong\mumu_4^2$ and
$$
T\cap\mathrm{Stab}_G(P)=\langle(i,i,-1)\rangle\cong\mumu_4.
$$
On the other hand, if $|\mathrm{Stab}_G(P)|=4$, then $|\mathrm{Orb}_G(P)|\geqslant 48$,  hence $5\geqslant k=|C\cap\ell_{12}|\geqslant 8$.
Therefore, if $\mathrm{im}(\upsilon)=\mathfrak{A}_4$, then $\upsilon(\mathrm{Stab}_G(P))=\langle (12)(34)\rangle\subset\mathfrak{A}_4$.
Similarly, if $\mathrm{im}(\upsilon)\cong\mathfrak{S}_4$, then $|\mathrm{Stab}_G(P)|\geqslant 16$,
which immediately implies that $\upsilon(\mathrm{Stab}_G(P))=\langle (12),(34)\rangle\subset\mathfrak{S}_4$.
Thus, there exists $\theta\in\mathrm{Stab}_G(P)$ such that $\upsilon(\theta)=(12)(34)$ and $\theta^2\in\langle(i,i,-1)\rangle$. Then
$$
\theta=\begin{pmatrix}
0 & b_2 & 0 & 0\\
b_1 & 0 & 0 & 0\\
0 & 0 & 0 & 1\\
0 & 0 & b_3 & 0
\end{pmatrix}
$$
for some non-zero numbers $b_1$, $b_2$, $b_3$ such that $b_1b_2=\pm b_3$.
Hence, conjugating $G$ by an~appropriate element of the torus $\mathbb{T}$, we may assume that $b_1=1$, $b_2=1$ and $b_3=\pm 1$.
In both cases, the subgroup $\langle(i,i,-1),\theta\rangle\subset\mathrm{Stab}_G(P)$ is not cyclic. In fact, one has
$$
\langle(i,i,-1),\theta\rangle\cong
\left\{\aligned
&\mathrm{D}_8\ \text{if $b_3=1$},\\
&\mathrm{Q}_8\ \text{if $b_3=-1$}.\\
\endaligned
\right.
$$
On the other hand, the subgroup $\langle(i,i,-1)\rangle\cong\mumu_4$ acts faithfully on $Z$, because $Z\not\subset\mathcal{T}$.
This implies that the whole group $\langle(i,i,-1),\theta\rangle$ also acts faithfully on $Z$, because
neither the dihedral group $\mathrm{D}_8$ nor the quaternion group $\mathrm{Q}_8$ have quotients isomorphic to $\mumu_4$.
Therefore, as above, we obtain a contradiction with \cite[Lemma~2.7]{FZ}.
\end{proof}

\begin{lemma}
\label{lemma:P3-curves-faces}
If $G$ is not conjugate to $G_{324,160}^\prime$, then $C\subset\mathcal{T}$.
If $G=G_{324,160}^\prime$ and $C\not\subset\mathcal{T}$, then $C$ is one of the following smooth irreducible curves of degree nine and genus ten:
\begin{align}
\big\{(1+\zeta_3)x_1^3+\zeta_3x_2^3+x_3^3&=x_0^3+\zeta_3 x_1^3-(1+\zeta_3)x_2^3=0\big\},\label{equation:48-curve-degree-9-1}\\
\big\{\zeta_3x_1^3+(1+\zeta_3)x_2^3-x_3^3&=x_0^3-(1+\zeta_3)x_1^3+\zeta_3x_2^3=0\big\}.\label{equation:48-curve-degree-9-2}
\end{align}
\end{lemma}

\begin{proof}
Suppose that  $C\not\subset\mathcal{T}$.
Then $C\cap\mathcal{L}_6=\varnothing$ by Lemma~\ref{lemma:P3-curves-edges}, hence Lemma~\ref{lemma:orbits} gives
$$
60\geqslant 4d=\mathcal{T}\cdot C\geqslant|\mathcal{T}\cap C|\geqslant 4n^2,
$$
which gives $n\leqslant 3$. Then $G$ is one of the subgroups $G_{324,160}$, $G_{324,160}^\prime$, $G_{648,704}$~or~$G_{648,704}^\prime$.
Recall from Section~\ref{section:subgroups} that $G_{324,160}\subset G_{648,704}$ and $G_{324,160}^\prime\subset G_{648,704}^\prime$.
Hence, to proceed, we may assume that $G=G_{324,160}$ or $G=G_{324,160}^\prime$, since $G_{648,704}^\prime$ swaps \eqref{equation:48-curve-degree-9-1} and \eqref{equation:48-curve-degree-9-2}.

Let $\Gamma=\mathrm{Stab}_G(F_4)$.
Then $\Gamma\cong\mumu_3^2\rtimes\mumu_3$ and $\Gamma$ is generated by
$$
\big(\zeta_3,1,1\big),\big(1,\zeta_3,1\big),
\begin{pmatrix}
0 & 0 & 1 & 0\\
1 & 0 & 0 & 0\\
0 & 1 & 0 & 0\\
0 & 0 & 0 & 1
\end{pmatrix}.
$$
Since  $15\geqslant d=F_4\cdot C\geqslant|F_4\cap C|$ and $F_4\cap C$ is a $\Gamma$-invariant subset in $F_4\setminus(\ell_{12}\cup\ell_{13}\cup\ell_{23})$,
we conclude that $F_4\cap C$ is the $\Gamma$-orbit of one of the following points:
$$
[1:1:1:0], [1:\zeta_3:\zeta_3^2:0],[1:\zeta_3^2:\zeta_3:0].
$$
Moreover, in both cases, we have  $d=F_4\cdot C=|F_4\cap C|=9$, which implies, in particular,
that the curve $C$ is smooth in every  intersection point $C\cap\mathcal{T}$.

Suppose that $G=G_{324,160}$. Let $S$ be the Fermat cubic $\{x_0^3+x_1^3+x_2^3+x_3^3=0\}\subset\mathbb{P}^3$.
Then $S$ is $G$-invariant, and $S$ does not contain $[1:1:1:0]$, $[1:\zeta_3:\zeta_3^2:0]$, $[1:\zeta_3^2:\zeta_3:0]$.
Thus, we conclude that $C\not\subset S$,
and the intersection $S\cap C$ is a $G$-invariant finite subset,
which is disjoint from the surface $\mathcal{T}$.
Moreover, since $|S\cap C|\leqslant S\cdot C\leqslant 27$,
it follows from~Lemma~\ref{lemma:orbits} that $S\cap C$ must be a $G$-orbit of length $27$,
which is not contained in $\mathcal{T}$.
This implies that $S\cap C=\mathrm{Orb}_G([1:1:1:1])$. But $[1:1:1:1]\not\in S$.

Thus, we see that $G=G_{324,160}^\prime$.
If $C$ is one of the curves \eqref{equation:48-curve-degree-9-1} or \eqref{equation:48-curve-degree-9-2}, we are done.
Hence, we assume that $C$ is not one of them. Let us seek for a contradiction.

We claim that $C$ is irreducible. Suppose it is not.
Then $C$ is a union of three cubics, or $C$ is a union of nine lines.
In the former case, the cubic curves must be non-planar, because $\mathbb{P}^3$ does not have $G$-orbits of length $3$.
Moreover, the group $G$ contains unique subgroup of index three up to conjugation \cite{Tim},
and this subgroup is isomorphic to $\mumu_3^3\rtimes\mumu_2^2$.
Since~$\mumu_3^3\rtimes\mumu_2^2$ cannot faithfully act on $\mathbb{P}^1$, we see that $C$ is not a union of three cubics.
Similarly, if the curve $C$ is a union of nine lines, then it follows from \cite{Tim} that
their stabilizers are isomorphic to $\mathfrak{S}_3\times\mathfrak{S}_3$.
The group $G$ contains nine such subgroups~\cite{Tim}, but all of them are conjugate.
One of these nine subgroups is generated by
$$
\begin{pmatrix}
\zeta_3 & 0 & 0 & 0\\
0 & 1 & 0 & 0\\
0 & 0 & \zeta_3 & 0\\
0 & 0 & 0 & 1
\end{pmatrix},
\begin{pmatrix}
\zeta_3 & 0 & 0 & 0\\
0 & \zeta_3 & 0 & 0\\
0 & 0 & 1 & 0\\
0 & 0 & 0 & 1
\end{pmatrix},
\begin{pmatrix}
0 & -1 & 0 & 0\\
1 & 0 & 0 & 0\\
0 & 0 & 0 & -1\\
0 & 0 & 1 & 0
\end{pmatrix},
\begin{pmatrix}
0 & 0 & 0 & -1\\
0 & 0 & -1 & 0\\
0 & 1 & 0 & -1\\
1 & 0 & 1 & 0
\end{pmatrix}
$$
Now, one can verify that this particular subgroup does not leave any line in $\mathbb{P}^3$ invariant.
The obtained contradiction shows that the curve $C$ is irreducible.

We claim that $C$ is smooth. Suppose $C$ is not smooth. Let $P$ be its singular point,
and let $S$ be a surface in the pencil of cubic surfaces that pass through \eqref{equation:48-curve-degree-9-1} such $P\in S$.
Then the~surface $S$ is given by the~equation
$$
\lambda\big((1+\zeta_3)x_1^3+\zeta_3x_2^3+x_3^3\big)=\mu\big(x_0^3+\zeta_3 x_1^3-(1+\zeta_3)x_2^3\big)
$$
for some $[\lambda:\mu]\in\mathbb{P}^1$. This surface is not $G$-invariant, but $\mathrm{Stab}_G(S)$ contains $T\cong\mumu_3^3$.
One the other hand, we have $|\mathrm{Orb}_P(G)|\geqslant 27$, because $P\not\in\mathcal{T}$.
Thus, if $C\not\subset S$, then
$$
27=S\cdot C\geqslant\sum_{O\in\mathrm{Orb}_P(G)}\big(S\cdot C\big)_O\geqslant\sum_{O\in\mathrm{Orb}_P(G)}\mathrm{mult}_O(S)\mathrm{mult}_O(C)\geqslant 2|\mathrm{Orb}_P(G)|\geqslant 54,
$$
which is absurd. Hence, we see that $C\subset S$.
Thus, since the~surface $S$ is not $G$-invariant,
the curve $C$ is contained in another cubic surface in the pencil of cubic surfaces that pass through the~curve \eqref{equation:48-curve-degree-9-1},
which implies that $C$ is contained in the base locus of this pencil.
But the base locus of this pencil is the irreducible curve \eqref{equation:48-curve-degree-9-1}, hence $C$ is the curve \eqref{equation:48-curve-degree-9-1},
which contradicts our assumption.
Therefore, we conclude that  $C$ is smooth.

Let $g$ be the genus of the curve $C$.
Now, using Castelnuovo bound, we see that $g\leqslant 12$.
Moreover, arguing exactly as in the proof of Lemma~\ref{lemma:48-curves}, we can easily prove that $g=10$.
Namely, recall that the~stabilizer in the~group $G$ of a point in $C$ is cyclic \cite[Lemma~2.7]{FZ},
which implies that the~$G$-orbits in the curve $C$ can be only of lengths $36$, $54$, $108$, $162$,
because  cyclic subgroups in $G$ are isomorphic to $\mumu_9$, $\mumu_6$, $\mumu_3$, $\mumu_2$ (see, for example, \cite{Tim}).
As in the proof of Lemma~\ref{lemma:48-curves}, let $\widehat{C}= C/G$, let $\hat{g}$ be the~genus of the~quotient curve~$\widehat{C}$,
and let $a_{36}$, $a_{54}$, $a_{108}$, $a_{162}$ be the~number of $G$-orbits in $C$ of length $36$, $54$, $108$, $162$.
Then
$$
22\geqslant 2g-2=48\big(2\hat{g}-2\big)+288a_{36}+270a_{54}+216a_{108}+162a_{162}
$$
by the Hurwitz's formula. This gives $g=10$ or $g=1$. But $g\ne 1$, since $G$ cannot act faithfully on a~smooth elliptic curve,
because our group $G\cong\mumu_3^3\rtimes\mathfrak{A}_4$ does not have abelian subgroups of index at most $6$ --- the largest abelian subgroup in $G$ is the subgroup $T\cong\mumu_3$.
Therefore, we conclude that $g=10$.

Let $\mathcal{M}_3$ be the linear system consisting of all cubic surfaces in $\mathbb{P}^3$ that pass through~$C$.
Then $\mathcal{M}_3$ is $G$-invariant. But a priori $\mathcal{M}_3$ may be empty. We claim that $\mathcal{M}_3$ is~not~empty, it is a pencil, and $C$ is its base locus.
Indeed, let $\mathcal{I}_C$ be the ideal sheaf of the curve $C\subset\mathbb{P}^3$.
Then we have the following exact sequence:
$$
0\longrightarrow H^0\big(\mathcal{O}_{\mathbb{P}^3}(3)\otimes\mathcal{I}_C\big)\longrightarrow H^0\big(\mathcal{O}_{\mathbb{P}^3}(3)\big)\longrightarrow H^0\big(\mathcal{O}_{\mathbb{P}^3}(3)\big\vert_{C}\big).
$$
On the other hand, it follows from the~Riemann--Roch theorem and Serre duality that
$$
h^0\big(\mathcal{O}_{\mathbb{P}^3}(3)\big\vert_{C}\big)=3d-g+1+h^0\big(K_C-\mathcal{O}_{\mathbb{P}^3}(3)\big\vert_{C}\big)=3d-g+1=18.
$$
Thus, since $h^0(\mathcal{O}_{\mathbb{P}^3}(3))=20$, we conclude that $\mathcal{M}_3$ is not empty, and it is at least a pencil.
Moreover, the linear system $\mathcal{M}_3$ does not have fixed components, because $\mathbb{P}^3$ does not contain
$G$-invariant planes and $G$-invariant quadrics.
Therefore, since $C$ is contained in the base locus of the~linear system $\mathcal{M}_4$ and $d=9$,
we conclude that $\mathcal{M}_3$ is a pencil, and the curve $C$ is its base locus.

On the other hand, the only $G$-invariant pencils in $|\mathcal{O}_{\mathbb{P}^3}(3)|$ are
the pencils of cubic surfaces that pass through \eqref{equation:48-curve-degree-9-1} or \eqref{equation:48-curve-degree-9-2}.
This can be shown explicitly or by using GAP.
This shows that $C$ is one of the curves  \eqref{equation:48-curve-degree-9-1} or \eqref{equation:48-curve-degree-9-2},
which contradicts our assumption.
\end{proof}

Now, we are ready to state the main result of this section:

\begin{proposition}
\label{proposition:P3-tetrahedra-curves}
Suppose that $C\ne\mathcal{L}_6$.
Then $G$ is conjugate to one of the following four subgroups: $G_{324,160}$, $G_{324,160}^\prime$, $G_{648,704}$, $G_{648,704}^\prime$.
Moreover, if $G=G_{648,704}$ or $G=G_{648,704}^\prime$, then $C$ is the~reducible curve of degree $9$ whose irreducible component is the~cubic
$$
\big\{x_0^3+x_1^3+x_2^3=x_4=0\big\}.
$$
Similarly, if $G=G_{324,160}$ or  $G=G_{324,160}^\prime$, then either $C$ is a curve of degree $9$ whose irreducible component is
$$
\big\{x_0^3+\zeta_3^r x_1^3+\zeta_3^{3-r}x_2^3=x_4=0\big\}
$$
for $r\in\{0,1,2\}$, or $G=G_{324,160}^\prime$ and $C$ is one of the curves \eqref{equation:48-curve-degree-9-1} and \eqref{equation:48-curve-degree-9-2}.
\end{proposition}

\begin{proof}
Using Lemma~\ref{lemma:P3-curves-faces}, we may assume that $\mathcal{T}$ contains $C$.
Let $Z$ be the union of all components of the curve $C$ that are contained in the plane~$F_4$, and let $\Gamma=\mathrm{Stab}_G(F_4)$.
Then $Z$ is $\Gamma$-invariant, but the degree of the curve $Z$ is at most~$3$, because $d\leqslant 15$.

Observe that the group $\Gamma$ acts transitively on the subset $\{P_1,P_2,P_3\}$, and this action induces a homomorphism $\Gamma\to\mathfrak{S}_3$,
whose image is either $\mumu_3$ or $\mathfrak{S}_3$.
Moreover, it follows from the~description of the~subgroup $T$ given in the~proofs of Lemmas~\ref{lemma:S4} and \ref{lemma:A4}
that the kernel of this homomorphism contains the subgroup $\langle(\zeta_n,1,1),(1,\zeta_n,1)\rangle\cong\mumu_n^2$,
which implies that either $Z$ is a smooth conic or $Z$ is a smooth cubic.

Our assumption on the group $G$ implies that $n\geqslant 3$. Thus, the curve $Z$ is not a conic, since the group $\mumu_n^2$
cannot act faithfully on $\mathbb{P}^1$ for $n\geqslant 3$.
Hence, we see that $Z$ is a~cubic.
Then $n=3$, since $\mathrm{Aut}(\mathbb{P}^2,Z)$ does not contain subgroups isomorphic to  $\mumu_n^2$ for $n\geqslant 4$.

Now, it follows from our assumption on $G$ and the results proved in the end of Section~\ref{section:subgroups}
that the~group $G$ is  conjugate to one of the subgroups $G_{324,160}$, $G_{324,160}^\prime$, $G_{648,704}$, $G_{648,704}^\prime$.
The remaining assertions are elementary computations.
\end{proof}

\begin{corollary}
\label{corollary:P3-tetrahedra-curves}
Let $C$ be a $G$-irreducible curve in $\mathbb{P}^3$ such that $C$ is different from $\mathcal{L}_6$,
and let $\mathcal{D}$ be a linear subsystem in $|\mathcal{O}_{\mathbb{P}^3}(n)|$ that has no fixed components, where $n\in\mathbb{Z}_{>0}$.
If the~subgroup $G$ is not conjugate to $G_{324,160}^\prime$, then
$$
\mathrm{mult}_C\big(\mathcal{D}\big)\leqslant\frac{n}{4}.
$$
If $G=G_{324,160}^\prime$ and $C$ is not one of the curves \eqref{equation:48-curve-degree-9-1} or \eqref{equation:48-curve-degree-9-2},
then $\mathrm{mult}_C(\mathcal{D})\leqslant\frac{n}{4}$.
\end{corollary}

\begin{proof}
Arguing as in the proof of Proposition~\ref{proposition:48-curves}, we obtain the required assertion.
\end{proof}

Let us conclude this section by proving the following technical result:

\begin{lemma}
\label{lemma:cubic-surface}
Let $S$ be a cubic surface in $\mathbb{P}^3$ that contains one of the curves \eqref{equation:48-curve-degree-9-1}~or~\eqref{equation:48-curve-degree-9-2},
let $\Gamma$ be the stabilizer of the surface $S$ in the~group $G_{324,160}^\prime$,
and let $D$ be a $\Gamma$-invariant effective $\mathbb{Q}$-divisor on the~surface $S$ such that $D\equiv-K_S$.
Then $(S,D)$ has log canonical singularities away from from singular points (if any) of the surface $S$.
\end{lemma}

\begin{proof}
Suppose that $S$ is smooth.
In this case, the required assertion means  $\alpha_\Gamma(S)\geqslant 1$, which is well-known.
Indeed, the group $\Gamma$ contains the subgroup $T\cong\mumu_2^3$, which implies that $S$ does not have $\Gamma$-fixed points.
On the other hand, if $(S,D)$ is not log canonical, then
there exists a point $P\in S$ such that the log pair $(S,D)$ is log canonical away from~$P$.
This~follows from \cite[Lemma 3.7]{CheltsovGAFA}
or from \cite{CheltsovParkWonJEMS} and the~Koll\'ar--Shokurov connectedness.
Thus, the point $P$ must be fixed by $\Gamma$, which is a contradiction.

Thus, we may assume that $S$ is singular.
Then there are exactly eight possibilities for the surface $S$, and all of them are similar.
So, without loss of generality, we may assume that the~surface $S$ is given by the~equation
$$
(1+\zeta_3)x_1^3+\zeta_3x_2^3+x_3^3=0.
$$
This is the cone with~vertex at $[0:0:0:1]$ over the~curve $\{(1+\zeta_3)x_1^3+\zeta_3x_2^3+x_3^3=x_4=0\}$.
Observe that $\Gamma\cong\mumu_3^3\rtimes\mumu_3$, since $\Gamma$ is the subgroup in $G_{324,160}^\prime$ that is generated by
$$
\begin{pmatrix}
\zeta_3 & 0 & 0 & 0\\
0 & 1 & 0 & 0\\
0 & 0 & 1 & 0\\
0 & 0 & 0 & 1
\end{pmatrix},
\begin{pmatrix}
1 & 0 & 0 & 0\\
0 & \zeta_3 & 0 & 0\\
0 & 0 & 1 & 0\\
0 & 0 & 0 & 1
\end{pmatrix},
\begin{pmatrix}
1 & 0 & 0 & 0\\
0 & 1 & 0 & 0\\
0 & 0 & \zeta_3 & 0\\
0 & 0 & 0 & 1
\end{pmatrix},
\begin{pmatrix}
0 & 0 & 1 & 0\\
1 & 0 & 0 & 0\\
0 & 1 & 0 & 0\\
0 & 0 & 0 & 1
\end{pmatrix}.
$$
Suppose that $(S,D)$ is not log canonical at some point $O\in S$ such that $S\ne[0:0:0:1]$.
Let us seek for a contradiction.

Let $L$ be the ruling of the cone $S$ that passes through $O$,
and let $\mathcal{L}$ be $\Gamma$-irreducible curve in $S$ whose irreducible components is the line $L$.
Then $\mathrm{deg}(\mathcal{L})\geqslant 9$, because $\Gamma$-orbits in the cubic curve $\{(1+\zeta_3)x_1^3+\zeta_3x_2^3+x_3^3=x_4=0\}$ have length at least $9$.
Let
$$
D^\prime=(1+\mu)D-\frac{3\mu}{\mathrm{deg}(\mathcal{L})}\mathcal{L},
$$
where $\mu$ is the largest positive rational number $\mu$ such that $\mathrm{Supp}(D^\prime)$ does not contain $\mathcal{L}$.
It follows from the proof of \cite[Lemma~2.2]{CheltsovParkWonJEMS} that such positive rational number $\mu$ exists.
Moreover, since $\mathrm{deg}(\mathcal{L})\geqslant 9$, the singularities of the log pair
$$
\Big(S,\frac{3}{\mathrm{deg}(\mathcal{L})}\mathcal{L}\Big)
$$
are log canonical at $O$. Therefore, the log pair $(S,D^\prime)$ is not log canonical at $O$, because
$$
D=\frac{\mu}{1+\mu}\Big(\frac{3}{\mathrm{deg}(\mathcal{L})}\mathcal{L}\Big)+\frac{1}{1+\mu}D^\prime.
$$
Observe that $D^\prime\equiv D\equiv -K_S$ by construction, hence
$1=D^\prime\cdot L\geqslant(D^\prime\cdot L)_O\geqslant\mathrm{mult}_{O}(D^\prime)$,
so the pair $(S,D^\prime)$ is log canonical at $O$ by \cite[Theorem~4.5]{KoMo98} or \cite[Exercise~6.18]{CoKoSm03}.
\end{proof}

\section{Rational Fano--Enriques threefold of degree 24}
\label{section:Fano-Enriques}

Let us use assumptions and notations of Section~\ref{section:subgroups}.
Recall from Section~\ref{section:subgroups} that
\begin{center}
$P_1=[1:0:0:0]$, $P_2=[0:1:0:0]$, $P_3=[0:0:1:0]$, $P_4=[0:0:0:1]$,
\end{center}
and $G$ is a finite subgroup in $\mathrm{PGL}_{4}(\mathbb{C})$ such that the~following conditions are satisfied:
\begin{enumerate}
\item the~group $G$ does not have fixed points in  $\mathbb{P}^3$,
\item the~group $G$ does not leave a union of two skew lines in $\mathbb{P}^3$ invariant,
\item the~group $G$ leaves invariant the subset $\{P_1,P_2,P_3,P_4\}$.
\end{enumerate}
Suppose that $G$ conjugate neither to $G_{48,3}$ nor to $G_{96,72}$.
Moreover, if $G$ is conjugate to one of the subgroups
$G_{48,50}$, $G_{96,70}$, $G_{96,227}$, $G_{96,227}^\prime$, $G_{192,955}$, $G_{192,185}$, $G_{324,160}^\prime$, $G_{648,704}^\prime$,
then we will always assume that $G$ is this subgroup. Recall that $G_{48,50}\triangleleft G_{96,70}\triangleleft G_{192,955}$,
$G_{48,50}\triangleleft G_{96,227}\triangleleft G_{192,955}$,
$G_{48,50}\triangleleft G_{96,227}^\prime\triangleleft G_{192,955}$ and $G_{324,160}^\prime \triangleleft G_{648,704}^\prime$.

For every $1\leqslant i<j\leqslant 4$, let $\ell_{ij}$ be the~line in $\mathbb{P}^3$ that passes through $P_i$ and $P_j$,
let
$$
F_1=\{x_0=0\},\, F_2=\{x_1=0\},\, F_3=\{x_2=0\},\, F_4=\{x_3=0\},
$$
let $\Sigma_4=\{P_1,P_2,P_3,P_4\}$, let $\mathcal{L}_6=\ell_{12}+\ell_{13}+\ell_{14}+\ell_{23}+\ell_{24}+\ell_{34}$, let $\mathcal{T}=F_1+F_2+F_3+F_4$.
By Corollary~\ref{corollary:P3-P3-explicit}, there exists a~$G$-birational involution $\iota\colon\mathbb{P}^3\dasharrow\mathbb{P}^3$
that is given by
$$
[x_0:x_1:x_2:x_3]\mapsto [\lambda_1 x_1x_2x_3: \lambda_2 x_0x_2x_3: \lambda_3 x_0x_1x_3: x_0x_1x_2]
$$
for some non-zero complex numbers $\lambda_1$, $\lambda_2$, $\lambda_3$.
This involution is well-defined away from the~curve~$\mathcal{L}_6$, and it contracts $F_1$, $F_2$, $F_3$, $F_4$ to the point $P_1$,~$P_2$,~$P_3$,~$P_4$,~respectively.
Observe also that the~involution $\iota$ fits the following $G$-commutative diagram:
\begin{equation}
\label{equation:Cremona-involution-1}
\xymatrix{
&\widetilde{V}_4\ar@{->}[ld]_{\pi}\ar@{->}[rd]^{\phi}\ar@{-->}[rrrr]^{\nu}&&&&\widetilde{V}_4\ar@{->}[rd]^{\pi}\ar@{->}[ld]_{\phi}\\%
\mathbb{P}^3\ar@/_2pc/@{-->}[rrrrrr]_{\iota}&&V_4\ar@{->}[rr]^{\sigma}&&V_4&&\mathbb{P}^3}
\end{equation}
where $V_4$ is an~intersection of two quadrics in $\mathbb{P}^5$
that has six ordinary double points, the~map $\pi$~is the blow up of the orbit $\Sigma_4$,
the map $\phi$ is the contraction of the proper transforms of the lines
$\ell_{12}$, $\ell_{13}$, $\ell_{14}$, $\ell_{23}$, $\ell_{24}$, $\ell_{34}$
to the singular points of the threefold~$V_4$, the~map $\sigma$ is a $G$-biregular involution,
and  $\nu$ is a $G$-birational non-biregular involution, which is a~composition of six Atiyah flops.

Moreover, it follows from \cite{CheltsovShramov2017,CDK} that $\iota$ also
its the following $G$-commutative diagram:
\begin{equation}
\label{equation:Cremona-involution-2}
\xymatrix{
&\widetilde{X}_{24}\ar@{->}[ld]_{\varpi}\ar@{->}[rd]^{\varphi}&&&&\widetilde{X}_{24}\ar@{->}[rd]^{\varpi}\ar@{->}[ld]_{\varphi}\\%
\mathbb{P}^3\ar@/_2pc/@{-->}[rrrrrr]_{\iota}&&X_{24}\ar@{->}[rr]^{\varsigma}&&X_{24}&&\mathbb{P}^3}
\end{equation}
where $X_{24}$ is the toric Fano--Enriques threefold described in Example~\ref{example:Fano-Enriques},
which has eight quotient singular points of type $\frac{1}{2}(1,1,1)$,
the~morphism $\varpi$~is a birational $G$-extremal contraction that contracts six irreducible surfaces to
the lines $\ell_{12}$, $\ell_{13}$, $\ell_{14}$, $\ell_{23}$, $\ell_{24}$, $\ell_{34}$,
the~morphism $\varphi$ is the contraction of the proper transforms of the~planes
$F_1$, $F_2$, $F_3$, $F_4$~to four singular points of the threefold~$X_{24}$, and $\varsigma$ is a biregular involution.

As we already mentioned in Example~\ref{example:Fano-Enriques},
the threefold $X_{24}$ can also be obtained as the~quotient $\mathbb{P}^1\times\mathbb{P}^1\times\mathbb{P}^1/\tau$,
where $\tau$ is the involution in $\mathrm{Aut}(\mathbb{P}^1\times\mathbb{P}^1\times\mathbb{P}^1)$ given by
$$
\big([u_1:v_1],[u_2:v_2],[u_3:v_3]\big)\mapsto\big([u_1:-v_1],[u_2:-v_2],[u_3:-v_3]\big).
$$
Then $\mathrm{Sing}(X_{24})$ consists of $8$ singular points of type $\frac{1}{2}(1,1,1)$ --- the~images of the~points
\begin{align*}
\big([0:1],[0:1],[0:1]\big),\big([0:1],[0:1],[1:0]\big),\big([0:1],[1:0],[0:1]\big),\big([0:1],[1:0],[1:0]\big),\\
\big([1:0],[0:1],[0:1]\big),\big([1:0],[0:1],[1:0]\big),\big([1:0],[1:0],[0:1]\big),\big([1:0],[1:0],[1:0]\big).
\end{align*}
To match this description of the~threefold $X_{24}$ with the~description given by~\eqref{equation:Cremona-involution-2},
we~set
$$
V_2=\big\{w^2=x_0x_1x_2x_3\big\}\subset\mathbb{P}(1,1,1,1,2).
$$
Let $\xi\colon V_2\to\mathbb{P}^3$ be the projection that is given by $[x_0:x_1:x_2:x_3:w]\mapsto[x_0:x_1:x_2:x_3]$,
where $x_0$, $x_1$, $x_2$ and $x_3$ are coordinates of weight $1$, and $w$ is a~coordinate of weight $2$.
Then it~follows from \cite{CDK} that there is
a birational map $\zeta\colon V_2\to \mathbb{P}^1\times\mathbb{P}^1\times\mathbb{P}^1$ such that
the following diagram commutes:
\begin{equation}
\label{equation:CDK-big-diagram}
\xymatrix{
V_2\ar@{->}[dd]_{\xi}\ar@/^2pc/@{-->}[rrrrrrd]^{\zeta}&&&&&&\\%
&&\widetilde{X}_{24}\ar@{->}[lld]_{\varpi}\ar@{->}[rrd]^{\varphi}&&&&\mathbb{P}^1\times\mathbb{P}^1\times\mathbb{P}^1\ar@{->}[lld]_{\omega}\ar@{->}[dd]^{\mathrm{pr}_i}\\%
\mathbb{P}^3\ar@{-->}[rrrr]_{\psi}&&&&X_{24}\ar@{->}[drr]_{\eta_i}&&\\
&&&&&&\mathbb{P}^1}
\end{equation}
where $\omega$ is the~quotient map by $\tau$, $\varpi$ and $\varphi$ are the birational morphisms defined in~\eqref{equation:Cremona-involution-2},
the~map $\psi$ is given~by the linear system of all sextic surfaces singular~along the~curve $\mathcal{L}_6$,
the map $\mathrm{pr}_i$ is the~projection to the~$i$-th factor,
and $\eta_i$ is the~morphism induced by~$\mathrm{pr}_i$.

It follows from \cite{CDK} that the~maps $\zeta$,  $\omega$, $\psi$ in the~diagram \eqref{equation:CDK-big-diagram} can be described in coordinates as follows:
the birational map $\zeta$ is given~by
$$
\big[x_{0}:x_{1}:x_{2}:x_{3}:w\big]\mapsto\Big(\big[x_{0}x_{1}:w\big],\big[x_{0}x_{2}:w\big],\big[x_{1}x_{2}:w\big]\Big),
$$
the quotient map $\omega$ is induced by the map $\mathbb{P}^1\times\mathbb{P}^1\times\mathbb{P}^1\to\mathbb{P}^{13}$ given by
\begin{multline*}\hspace*{-0.4cm}
\big([u_{1}:v_{1}],[u_{2}:v_{2}],[u_{3}:v_{3}]\big)\mapsto\Big[u_{1}^{2}u_{2}^{2}u_{3}^{2}:u_{1}^{2}u_{2}^{2}v_{3}^{2}:u_{1}^{2}u_{2}v_{2}u_{3}v_{3}:u_{1}^{2}v_{2}^{2}u_{3}^{2}:u_{1}^{2}v_{2}^{2}v_{3}^{2}:u_{1}v_{1}u_{2}^{2}u_{3}v_{3}:\\
:u_{1}v_{1}u_{2}v_{2}u_{3}^{2}:u_{1}v_{1}u_{2}v_{2}v_{3}^{2}:u_{1}v_{1}v_{2}^{2}u_{3}v_{3}:v_{1}^{2}u_{2}^{2}u_{3}^{2}:v_{1}^{2}u_{2}^{2}v_{3}^{2}:v_{1}^{2}u_{2}v_{2}u_{3}v_{3}:v_{1}^{2}v_{2}^{2}u_{3}^{2}:v_{1}^{2}v_{2}^{2}v_{3}^{2}\Big],
\end{multline*}
and the birational map $\psi$ is induced by the map $\mathbb{P}^3\dasharrow \mathbb{P}^{13}$ given by
\begin{multline*}
\big[x_{0}:x_{1}:x_{2}:x_{3}\big]\mapsto\Big[x_{0}^{2}x_{1}^{2}x_{2}^{2}:x_{0}^{3}x_{1}x_{2}x_{3}:x_{0}^{2}x_{1}^{2}x_{2}x_{3}:x_{0}x_{1}^{3}x_{2}x_{3}:x_{0}^{2}x_{1}^{2}x_{3}^{2}:x_{0}^{2}x_{1}x_{2}^{2}x_{3}:\\
:x_{0}x_{1}^{2}x_{2}^{2}x_{3}:x_{0}^{2}x_{1}x_{2}x_{3}^{2}:x_{0}x_{1}^{2}x_{2}x_{3}^{2}:x_{0}x_{1}x_{2}^{3}x_{3}:x_{0}^{2}x_{2}^{2}x_{3}^{2}:x_{0}x_{1}x_{2}^{2}x_{3}^{2}:x_{1}^{2}x_{2}^{2}x_{3}^{2}:x_{0}x_{1}x_{2}x_{3}^{3}\Big].
\end{multline*}
Using this, we see that the biregular involution $\varsigma$ in \eqref{equation:Cremona-involution-2}
is induced~by the biregular involution of $\mathbb{P}^1\times\mathbb{P}^1\times\mathbb{P}^1$ that is given by
$$
\big(\big[u_{1}:v_{1}\big],\big[u_{2}:v_{2}\big],\big[u_{3}:v_{3}\big]\big)\mapsto\big(\big[\lambda_1\lambda_2v_{1}:\lambda_3u_{1}\big],\big[\lambda_1\lambda_3v_{2}:\lambda_2u_{2}\big],\big[\lambda_2\lambda_3v_{3}:\lambda_1u_{3}\big]\big).
$$
Similarly, we see that
\begin{itemize}
\item the map $\eta_1\circ\psi\colon\mathbb{P}^3\dasharrow\mathbb{P}^1$ is given by $[x_{0}:x_{1}:x_{2}:x_{3}]\mapsto[x_{0}x_{1}:x_{2}x_{3}]$,
\item the~map $\eta_2\circ\psi\colon\mathbb{P}^3\dasharrow\mathbb{P}^1$ is given by  $[x_{0}:x_{1}:x_{2}:x_{3}]\mapsto[x_{0}x_{2}:x_{1}x_{3}]$,
\item the~map $\eta_3\circ\psi\colon\mathbb{P}^3\dasharrow\mathbb{P}^1$ is given by $[x_{0}:x_{1}:x_{2}:x_{3}]\mapsto[x_{1}x_{2}:x_{0}x_{3}]$.
\end{itemize}

Using  $\psi$, we can define the $G$-action on $X_{24}$ such that $\psi$ is $G$-equivariant. Note that
\begin{align*}
\psi(F_1)&=[0:0:0:0:0:0:0:0:0:0:0:0:1:0]=\omega\big([0:1],[0:1],[1:0]\big),\\
\psi(F_2)&=[0:0:0:0:0:0:0:0:0:0:1:0:0:0]=\omega\big([0:1],[1:0],[0:1]\big),\\
\psi(F_3)&=[0:0:0:0:1:0:0:0:0:0:0:0:0:0]=\omega\big([1:0],[0:1],[0:1]\big),\\
\psi(F_4)&=[1:0:0:0:0:0:0:0:0:0:0:0:0:0]=\omega\big([1:0],[1:0],[1:0]\big).
\end{align*}
Thus, the locus $\mathrm{Sing}(X_{24})$ splits into two $G$-orbits: the~orbit
$\{\psi(F_1),\psi(F_2),\psi(F_3),\psi(F_4)\}$, and the orbit that consists of the  points
$$\hspace*{-0.3cm}
\omega\big([0:1],[0:1],[0:1]\big),\omega\big([0:1],[1:0],[1:0]\big),\omega\big([1:0],[0:1],[1:0]\big),\omega\big([1:0],[1:0],[0:1]\big).
$$
The involution $\varsigma$ swaps these two $G$-orbits.

Let $E_{11}$, $E_{12}$, $E_{21}$, $E_{22}$, $E_{31}$, $E_{32}$
be the images in $X_{24}$ of the surfaces in $\mathbb{P}^1\times\mathbb{P}^1\times\mathbb{P}^1$
that are given by the equations $u_1=0$, $v_1=0$, $u_2=0$, $v_2=0$, $u_3=0$, $v_3=0$, respectively.
Then $E_{11}$, $E_{12}$, $E_{21}$, $E_{22}$, $E_{31}$ and $E_{32}$ are singular toric del Pezzo surfaces of degree $4$,
and $\psi$ induces an isomorphism
$$
\mathbb{P}^3\setminus\big(F_1\cup F_2\cup F_3 \cup F_3\big)\cong X_{24}\setminus\big(E_{11}\cup E_{12}\cup E_{21}\cup E_{22}\cup E_{31}\cup E_{32}\big).
$$
Let $\widetilde{E}_{11}$, $\widetilde{E}_{12}$, $\widetilde{E}_{21}$, $\widetilde{E}_{22}$, $\widetilde{E}_{31}$, $\widetilde{E}_{32}$
be the~proper transforms on $\widetilde{X}_{24}$ of the surfaces $E_{11}$, $E_{12}$, $E_{21}$, $E_{22}$, $E_{31}$, $E_{32}$, respectively.
Then $\widetilde{E}_{11}$, $\widetilde{E}_{12}$, $\widetilde{E}_{21}$, $\widetilde{E}_{22}$, $\widetilde{E}_{31}$, $\widetilde{E}_{32}$
are smooth del Pezzo surfaces of degree $6$. Moreover, we have
$$
\varpi\big(\widetilde{E}_{11}\big)=\ell_{34}, \varpi\big(\widetilde{E}_{12}\big)=\ell_{12}, \varpi\big(\widetilde{E}_{21}\big)=\ell_{24},
\varpi\big(\widetilde{E}_{22}\big)=\ell_{13}, \varpi\big(\widetilde{E}_{31}\big)=\ell_{14}, \varpi\big(\widetilde{E}_{32}\big)=\ell_{23}.
$$
Let $\mathcal{E}=E_{11}+E_{12}+E_{21}+E_{22}+E_{31}+E_{32}$, and let $\mathcal{Z}_{12}=\mathrm{Sing}(\mathcal{E})$.
Then $\mathcal{E}$ is a~$G$-irreducible surface, and $\mathcal{Z}_{12}$ is a~$G$-irreducible curve in $X_{24}$
that consists of $12$ distinct lines in $\mathbb{P}^{13}$,
which are all lines contained in $\mathrm{Supp}(\mathcal{E})$.
Note that $\mathrm{Sing}(\mathcal{Z}_{12})=\mathrm{Sing}(X_{24})$.

If the subgroup $G$ is conjugate to none of the groups $G_{48,50}$ and $G_{96,227}$, then it follows from Lemmas~\ref{lemma:48-orbits}, \ref{lemma:192-orbits}, \ref{lemma:orbits}
that the~$G$-orbit $\Sigma_4$ is the unique $G$-orbit in $\mathbb{P}^3$ of length~four.
On the other hand, if $G=G_{48,50}$ or $G=G_{96,227}$, then
it follows from Lemma~\ref{lemma:48-orbits} that the~space $\mathbb{P}^3$ contains exactly three $G$-orbits of length~four:
$\Sigma_4$, $\Sigma_4^\prime$ and $\Sigma_4^{\prime\prime}$, where
$$
\Sigma_4^\prime=\Big\{[1:1:1:-1], [1:1:-1:1], [1:-1:1:1], [-1:1:1:1]\Big\}\not\subset\mathcal{T}
$$
and also
$$
\Sigma_4^{\prime\prime}=\{[1:1:1:1], [1:1:-1:-1], [1:-1:-1:1], [-1:-1:1:1]\}\not\subset\mathcal{T}.
$$
Therefore, if $G=G_{48,50}$ or $G=G_{96,227}$, then  $\psi(\Sigma_4^\prime)$ and $\psi(\Sigma_4^{\prime\prime})$ are $G$-orbits of length $4$.

Similarly, if $G$ is conjugate to none of the groups $G_{48,50}$ and $G_{96,227}$, it easily follows from~Lemmas~\ref{lemma:48-orbits}, \ref{lemma:192-orbits} and \ref{lemma:orbits},
that $\mathbb{P}^3\setminus\mathcal{T}$ does not contain $G$-orbits of length $<16$.
On the other hand, if $G=G_{48,50}$ or $G=G_{96,227}$, then
it follows from Lemma~\ref{lemma:48-orbits} that the~$G$-orbits of length $<16$ contained in $\mathbb{P}^3\setminus\mathcal{T}$
can be described as follows:
\begin{center}
$\Sigma_4^\prime$, $\Sigma_4^{\prime\prime}$, $\Sigma_{12}^{\prime\prime}=\mathrm{Orb}_G([i:i:1:1])$, $\Sigma_{12}^{\prime\prime\prime}=\mathrm{Orb}_G([-i:i:1:1])$.
\end{center}
Keeping in mind that $\psi$ gives an isomorphism $\mathbb{P}^3\setminus\mathcal{T}\cong X_{24}\setminus\mathcal{E}$, we get

\begin{corollary}
\label{corollary:X24-orbits}
Let $\Sigma$ be a $G$-orbit in $X_{24}$ such that $|\Sigma|\leqslant 15$, and $\Sigma$ is not contained~in~$\mathcal{E}$.
Then $G=G_{48,50}$ or $G=G_{96,227}$, and $\Sigma$ is one of the orbits $\psi(\Sigma_4^\prime)$, $\psi(\Sigma_4^{\prime\prime})$,
$\psi(\Sigma_{12}^{\prime\prime})$, $\psi(\Sigma_{12}^{\prime\prime\prime})$.
\end{corollary}

Let $H$ be a general hyperplane section of the threefold $X_{24}\subset\mathbb{P}^{13}$. Then
$$
\varphi^*(H)\sim\varpi^*\big(\mathcal{O}_{\mathbb{P}^3}(6)\big)-2\big(\widetilde{E}_{11}+\widetilde{E}_{12}+\widetilde{E}_{21}+\widetilde{E}_{22}+\widetilde{E}_{31}+\widetilde{E}_{32}\big).
$$
Let $\widetilde{F}_1$, $\widetilde{F}_2$, $\widetilde{F}_3$, $\widetilde{F}_4$ be the proper transform on $\widetilde{X}_{24}$ of the planes $F_1$, $F_2$, $F_3$, $F_4$, respectively.
Then $\varpi^*(\mathcal{O}_{\mathbb{P}^3}(2))\sim\varphi^*(H)-\widetilde{F}_1-\widetilde{F}_2-\widetilde{F}_3-\widetilde{F}_4$, because
$$
\widetilde{F}_1+\widetilde{F}_2+\widetilde{F}_3+\widetilde{F}_4\sim\varpi^*\big(\mathcal{O}_{\mathbb{P}^3}(4)\big)-2\big(\widetilde{E}_{11}+\widetilde{E}_{12}+\widetilde{E}_{21}+\widetilde{E}_{22}+\widetilde{E}_{31}+\widetilde{E}_{32}\big).
$$
Thus, we conclude that there exists $G$-commutative diagram
$$
\xymatrix{
&X_{24}\ar@{-->}[rd]&\\%
\mathbb{P}^3\ar@{-->}[ur]^{\psi}\ar@{^{(}->}[rr]&&\mathbb{P}^9},
$$
where $\mathbb{P}^3\hookrightarrow\mathbb{P}^9$ is the second Veronese embedding,
and $X_{24}\dasharrow\mathbb{P}^9$ is the rational map which is given by the linear projection $\mathbb{P}^{13}\dasharrow\mathbb{P}^9$ from
the~three-dimensional linear subspace in $\mathbb{P}^{13}$ that contains the points $\varphi(\widetilde{F}_1)$, $\varphi(\widetilde{F}_2)$, $\varphi(\widetilde{F}_3)$, $\varphi(\widetilde{F}_4)$.
As above, we can translate these maps into equations as follows: the projection $X_{24}\dasharrow\mathbb{P}^9$ is given by
\begin{multline*}
[z_0:z_1:z_2:z_3:z_4:z_5:z_6:z_7:z_8:z_9:z_{10}:z_{11}:z_{12}:z_{13}]\mapsto\\
\mapsto[z_1:z_2:z_3:z_5:z_6:z_7:z_8:z_9:z_{11}:z_{13}],
\end{multline*}
and the second Veronese embedding $\mathbb{P}^3\hookrightarrow\mathbb{P}^9$ is given by
$$
[x_0:x_1:x_2:x_3]\mapsto\big[x_0^2:x_0x_1:x_1^2:x_0x_2:x_1x_2:x_0x_3:x_1x_3:x_2^2:x_2x_3:x_3^2\big].
$$

As we already mentioned, the six surfaces $E_{11}$, $E_{12}$, $E_{21}$, $E_{22}$, $E_{31}$, $E_{32}$ are singular toric del Pezzo surfaces of degree $4$,
and each of them has four isolated ordinary double points. The singular locus of each of these surfaces consists of $4$ points in $\mathrm{Sing}(X_{24})$,
and exactly two of them are contained in $\{\varphi(\widetilde{F}_1),\varphi(\widetilde{F}_2),\varphi(\widetilde{F}_3),\varphi(\widetilde{F}_4)\}$.
For instance, one has
$$
\mathrm{Sing}\big(E_{11}\big)=\Big\{\varphi\big(\widetilde{F}_1\big),\varphi\big(\widetilde{F}_2\big),\omega\big([0:1],[0:1],[0:1]\big),\omega\big([0:1],[1:0],[1:0]\big)\Big\},
$$
and the map $X_{24}\dasharrow\mathbb{P}^9$ induces the rational map $E_{11}\dasharrow \mathbb{P}^9$ that whose image is a conic, which is the Veronese image of the line $\ell_{34}$.

\begin{lemma}
\label{lemma:alpha-quartic}
Let $S$ be one of the toric del Pezzo surfaces $E_{11}$, $E_{12}$, $E_{21}$, $E_{22}$, $E_{31}$,~$E_{32}$,
and let $\Gamma$ be the image of the natural homomorphism $\mathrm{Stab}_G(S)\to\mathrm{Aut}(S)$.
Then $\alpha_{\Gamma}(S)=1$.
\end{lemma}

\begin{proof}
We may assume that $S=E_{11}$. Note that $\mathrm{Stab}_G(S)$ does not
always act faithfully on the surface $S$, hence we may have $\Gamma\not\cong\mathrm{Stab}_G(S)$.
For instance, if $G=G_{48,50}$, then
$$
\mathrm{Stab}_G(S)=\mathrm{Stab}_G(\ell_{34})=\big\langle M,N,B\big\rangle\cong\mumu_2^3,
$$
where $M$, $N$ and $B$ are involutions in $G_{48,50}$ described in Section~\ref{section:P3-48}.
However, using \eqref{equation:CDK-big-diagram}, one can check that the involution $N$ acts trivially on $S$,
and $\Gamma\cong\mumu_2^2$.

Let us describe geometry of the surface $S$. To do this, we let
\begin{align*}
L_1=\omega\big(\{u_1=u_3=0\}\big)&, L_1^\prime=\omega\big(\{u_1=v_3=0\}\big),\\
L_2=\omega\big(\{u_1=u_2=0\}\big)&, L_2^\prime=\omega\big(\{u_1=v_2=0\}\big).
\end{align*}
Then $L_1$, $L_1^\prime$, $L_2$, $L_{2}^\prime$ are smooth rational curves in $S$ such that
$2L_1\sim 2L_1^\prime$ and $2L_2\sim 2L_2^\prime$.
Note that $L_1\cap L_1^\prime=\varnothing$, $L_2\cap L_2^\prime=\varnothing$ and
\begin{align*}
L_1\cap L_{2}^\prime=\varphi\big(\widetilde{F}_2\big)&, L_1\cap L_{2}=\omega\big([0:1],[0:1],[0:1]\big),\\
L_1^\prime\cap L_{2}=\varphi\big(\widetilde{F}_1\big)&, L_1^\prime\cap L_{2}^\prime=\omega\big([0:1],[1:0],[1:0]\big).
\end{align*}
The~intersections of these curves on the surface $S$ are contained in following table:
\begin{center}
\renewcommand\arraystretch{1.3}
	\begin{tabular}{|c||c|c|c|c|}
		\hline
		 & $L_1$ & $L_1^\prime$ & $L_2$ & $L_{2}^\prime$\\
		\hline
		\hline
$L_1$ & $0$& $0$ & $\frac{1}{2}$ & $\frac{1}{2}$\\
		\hline
$L_1^\prime$  & $0$& $0$ & $\frac{1}{2}$ & $\frac{1}{2}$\\
		\hline
$L_2$ & $\frac{1}{2}$ & $\frac{1}{2}$ & $0$& $0$ \\
\hline
$L_{2}^\prime$ & $\frac{1}{2}$ & $\frac{1}{2}$ & $0$& $0$ \\
		\hline
	\end{tabular}
\end{center}

Note that $H\vert_{E_{11}}\sim 2L_1+2L_2$ and $-K_S\sim L_1+L_1^\prime+L_1+L_1^\prime$.
In particular, since the~divisor $L_1+L_1^\prime+L_1+L_1^\prime$ is $\Gamma$-invariant, we see that $\alpha_{\Gamma}(S)\leqslant 1$.

Observe that $L_1$ is the unique curve in $|L_1|$,
the curve $L_1^\prime$ is the unique curve in $|L_1^\prime|$,
the curve $L_2$ is the unique curve in~$|L_2|$,
and $L_2^\prime$ is the unique curve in $|L_2^\prime|$.

Note that $L_1+L_2\sim L_1^\prime+L_2^\prime$, and the linear system $|L_1+L_2|$ is a $\Gamma$-invariant pencil,
whose base locus consists of the~points $\varphi(\widetilde{F}_1)$ and $\varphi(\widetilde{F}_2)$.
This pencil gives a $\Gamma$-rational map $S\dasharrow\mathbb{P}^1$,
which is the map $S\dasharrow\ell_{34}$ induced by the birational map $\psi^{-1}\colon X_{24}\dasharrow\mathbb{P}^3$.
Then $|L_1+L_2|$ does not have $\Gamma$-invariant curves,
since $\ell_{34}$ has no $\mathrm{Stab}_G(\ell_{34})$-fixed points,
because $\mathbb{P}^3$ does not have $G$-orbits of length $6$ by Lemmas~\ref{lemma:48-orbits}, \ref{lemma:192-orbits},~\ref{lemma:orbits}.

Similarly, we see that $|L_1+L_2^\prime|$ is a $\Gamma$-invariant pencil generated by $L_1+L_2^\prime$ and $L_1^\prime+L_2$,
and its base locus consists of the points $\omega([0:1],[0:1],[0:1])$ and $\omega([0:1],[1:0],[1:0])$.
This pencil gives a rational map $S\dasharrow\ell_{12}$,
which is induced by the birational map $\psi^{-1}\circ\sigma$,
where $\sigma$ is the~involution from \eqref{equation:Cremona-involution-2}.
As above, we conclude that $|L_1+L_2^\prime|$ also does not contain $\Gamma$-invariant curves.

Since neither $|L_1+L_2|$ nor $|L_1+L_2^\prime|$ contains $G$-invariant curves,
we also conclude that none of the curves $L_1$, $L_1^\prime$, $L_2$, $L_2^\prime$ is $\Gamma$-invariant,
which can be checked directly.

We claim that $S$ does not have $\Gamma$-fixed points.
Indeed, the stabilizer $\mathrm{Stab}_G(\ell_{34})$ swaps the planes $F_1$ and $F_2$,
so that the group $\Gamma$ swaps the singular points $\varphi(\widetilde{F}_1)$ and $\varphi(\widetilde{F}_2)$.
Thus, if $S$ contained a $\Gamma$-fixed point $P$, then $|L_1+L_2|$  would contain a unique curve that passes through this point,
so that this curve would be $\Gamma$-invariant. But we already proved that the~pencil $|L_1+L_2|$
has no $\Gamma$-invariant curves.
So, the surface $S$ has no $\Gamma$-fixed~points.

Now, we ready to prove that $\alpha_{\Gamma}(S)=1$.
We suppose that $\alpha_{\Gamma}(S)<1$.
Then $S$ contains a~$\Gamma$-invariant effective $\mathbb{Q}$-divisor $D$ such that $D\sim_{\mathbb{Q}}-K_S$, but
the~pair $(S,\lambda D)$ is not log canonical for some rational number \mbox{$\lambda<1$}.
Note that the~locus $\mathrm{Nklt}(S,\lambda D)$ is $\Gamma$-invariant.
Therefore, if this locus is zero-dimensional,
then using Koll\'ar--Shokurov connectedness theorem \cite[Corollary~5.49]{KoMo98},
we conclude that $\mathrm{Nklt}(S,\lambda D)$ consists of a~single~point,
which is impossible, because $S$ does not have $\Gamma$-fixed points.

Since the~locus $\mathrm{Nklt}(S,\lambda D)$ is not zero-dimensional, it contains a $\Gamma$-irreducible curve~$C$.
Then $D=\mu C+\Delta$,
where $\mu\in\mathbb{Q}_{>0}$ such that $\mu\geqslant\frac{1}{\lambda}>1$, and $\Delta$ is an effective divisor.
Using~\cite[Lemma~2.9]{CheltsovProkhorov}, we see that $C\sim a_1L_1+a_2L_1^\prime+a_3L_2+a_4L_2^\prime$
for some non-negative integers $a_1$, $a_2$, $a_3$, $a_4$.
Then $-K_S\sim_{\mathbb{Q}}\mu(a_1L_1+a_2L_1^\prime+a_3L_2+a_4L_2^\prime)+\Delta$, hence $$
1=\mu\big(a_1L_1+a_2L_1^\prime+a_3L_2+a_4L_2^\prime\big)\cdot L_1+\Delta\cdot L_1=\mu\frac{a_3+a_4}{2}+\Delta\cdot L_1\geqslant \mu\frac{a_3+a_4}{2}>\frac{a_3+a_4}{2},
$$
so that $a_3+a_4<2$. Hence, we have $(a_3,a_4)\in\{(0,0),(1,0),(0,1)\}$.
Similarly, intersecting the divisor $D$ with $L_2$, we see that $(a_1,a_2)\in\{(0,0),(1,0),(0,1)\}$.

If $(a_3,a_4)=(0,0)$, then $(a_1,a_2)\ne(0,0)$, hence $(a_1,a_2)=(1,0)$ or $(a_1,a_2)=(0,1)$, which is impossible,
since $L_1$ is the unique curve in $|L_1|$,
and $L_1^\prime$ is the unique curve in~$|L_1^\prime|$,
but none of these two curves is $\Gamma$-invariant.
Therefore, we conclude that $(a_3,a_4)\ne (0,0)$.
Similarly, we see that $(a_1,a_2)\ne (0,0)$.
Hence, we see that $C\in|L_1+L_2|$ or $C\in|L_1+L_2^\prime|$,
which is impossible, because  neither $|L_1+L_2|$ nor $|L_1+L_2^\prime|$ contains $G$-invariant curves.
\end{proof}

Let us conclude this section by proving the following result.

\begin{lemma}
\label{lemma:X24-curves}
Let $C$ be a $G$-irreducible curve in $X_{24}$ such that $C\not\subset\mathcal{E}$ and $\mathrm{deg}(C)<24$.
Then $G$ is one of the groups $G_{48,50}$, $G_{96,70}$, $G_{96,227}$, and the following assertions hold:
\begin{itemize}
\item if $G=G_{48,50}$, then $C$ is one of the curves $\psi(\mathcal{L}_6^\prime)$, $\psi(\mathcal{L}_6^{\prime\prime})$, $\psi(\mathcal{L}_6^{\prime\prime\prime})$, $\psi(\mathcal{L}_6^{\prime\prime\prime\prime})$,
\item if $G=G_{96,70}$, then $C$ is one of the curves $\psi(\mathcal{L}_6^{\prime\prime\prime})$, $\psi(\mathcal{L}_6^{\prime\prime\prime\prime})$,
\item if $G=G_{96,227}$, then $C$ is one of the curves $\psi(\mathcal{L}_6^\prime)$, $\psi(\mathcal{L}_6^{\prime\prime})$,
\end{itemize}
where $\mathcal{L}_6^\prime$, $\mathcal{L}_6^{\prime\prime}$, $\mathcal{L}_6^{\prime\prime\prime}$, $\mathcal{L}_6^{\prime\prime\prime\prime}$
are $G_{48,50}$-irreducible curves in $\mathbb{P}^3$ introduced in Section~\ref{section:P3-48}.
\end{lemma}

\begin{proof}
Let $\widetilde{C}$ be the proper transform of the curve $C$ on the threefold $\widetilde{X}$.
Since $C\not\subset\mathcal{E}$, we conclude that $\varpi(\widetilde{C})$ is a $G$-irreducible curve in $\mathbb{P}^3$
which is not contained in $\mathcal{T}$.
Then
$$
2\mathrm{deg}\big(\varpi(\widetilde{C})\big)=\varpi^*(\mathcal{O}_{\mathbb{P}^3}(2))\cdot\widetilde{C}=H\cdot C-\big(\widetilde{F}_1+\widetilde{F}_2+\widetilde{F}_3+\widetilde{F}_4\big)\cdot\widetilde{C}\leqslant\mathrm{deg}(C)<24.
$$
So, the degree of the curve $\varpi(\widetilde{C})$ is at most $11$.
Now, using Lemma~\ref{proposition:P3-tetrahedra-curves} and our assumption, we see that $G$ is one of the groups
$G_{48,50}$, $G_{96,70}$, $G_{96,227}$, $G_{96,227}^\prime$, $G_{192,955}$, $G_{192,185}$, $G_{324,160}^\prime$.

If $G=G_{192,185}$, then it follows from Proposition~\ref{proposition:P3-192-curves}
that the curve $\varpi(\widetilde{C})$ is a disjoint union of two smooth quartic elliptic curves that are both disjoint from the~curve $\mathcal{L}_6$.
If~$G=G_{324,160}^\prime$, it follows from Lemma~\ref{proposition:P3-tetrahedra-curves}
that $\varpi(\widetilde{C})$ is one of the curves \eqref{equation:48-curve-degree-9-1} and \eqref{equation:48-curve-degree-9-2},
which are also disjoint from  $\mathcal{L}_6$.
Therefore, if $G=G_{192,185}$ or $G=G_{324,160}^\prime$, then
\begin{multline*}
24>\mathrm{deg}(C)=\varphi^*(H)\cdot\widetilde{C}=\\
=\Big(\varpi^*\big(\mathcal{O}_{\mathbb{P}^3}(6)\big)-2\big(\widetilde{E}_{11}+\widetilde{E}_{12}+\widetilde{E}_{21}+\widetilde{E}_{22}+\widetilde{E}_{31}+\widetilde{E}_{32}\big)\Big)\cdot\widetilde{C}=\\
=\varpi^*\big(\mathcal{O}_{\mathbb{P}^3}(6)\big)\cdot\widetilde{C}=\mathcal{O}_{\mathbb{P}^3}(6)\cdot \varpi(\widetilde{C})=6\mathrm{deg}\big(\varpi(\widetilde{C})\big)\geqslant 48.
\end{multline*}
Thus, we see that $G$ is one of the groups $G_{48,50}$, $G_{96,70}$, $G_{96,227}$, $G_{96,227}^\prime$, $G_{192,955}$, $G_{192,185}$.

Note that all groups
$G_{96,70}$, $G_{96,227}$, $G_{96,227}^\prime$, $G_{192,955}$, $G_{192,185}$ contains the group~$G_{48,50}$.
Moreover, each finite group $G_{96,70}$, $G_{96,227}^\prime$, $G_{192,955}$, $G_{192,185}$ swaps the~curves $\mathcal{L}_6^\prime$~and~$\mathcal{L}_6^{\prime\prime}$,
and each finite group among $G_{96,227}$, $G_{96,227}^\prime$, $G_{192,955}$, $G_{192,185}$ swaps the~curves $\mathcal{L}_6^{\prime\prime\prime}$~and~$\mathcal{L}_6^{\prime\prime\prime\prime}$. 
Therefore, to complete the proof of the lemma, we may assume that~$G=G_{48,50}$.

Now, using results of Section~\ref{section:P3-48},
we conclude that either $\varpi(\widetilde{C})$
is a smooth irreducible curve of degree $8$ and genus $9$ contained in the~quadric $\mathcal{Q}_1$,
or $\varpi(\widetilde{C})$ is one of the reducible curves
$\mathcal{L}_4$,
$\mathcal{L}_4^\prime$,
$\mathcal{L}_4^{\prime\prime}$, $\mathcal{L}_4^{\prime\prime\prime}$,
$\mathcal{L}_6$, $\mathcal{L}_6^{\prime}$, $\mathcal{L}_6^{\prime\prime}$, $\mathcal{L}_6^{\prime\prime\prime}$,
$\mathcal{L}_6^{\prime\prime\prime\prime}$,
which have been introduced in Section~\ref{section:P3-48}.
In~the former case, the~curve $\varpi(\widetilde{C})$ does not intersect the~curve $\mathcal{L}_6$,
because $\mathcal{L}_6\cap\mathcal{Q}_1=\Sigma_{12}^\prime$,
but smooth $G$-invariant irreducible curves contain no $G$-orbits of length $12$ by Lemma~\ref{lemma:48-curves}.
Similarly, all curves $\mathcal{L}_4$, $\mathcal{L}_4^\prime$, $\mathcal{L}_4^{\prime\prime}$, $\mathcal{L}_4^{\prime\prime\prime}$
are disjoint from $\mathcal{L}_{6}$.
Hence, if the~curve $\varpi(\widetilde{C})$ is not one of the curves  $\mathcal{L}_6^{\prime}$, $\mathcal{L}_6^{\prime\prime}$, $\mathcal{L}_6^{\prime\prime\prime}$,
$\mathcal{L}_6^{\prime\prime\prime\prime}$, then $\varpi(\widetilde{C})\cap\mathcal{L}_6=\varnothing$, hence $$
24>\mathrm{deg}(C)=\varphi^*(H)\cdot\widetilde{C}=\Big(\varpi^*\big(\mathcal{O}_{\mathbb{P}^3}(6)\big)-2\big(\widetilde{E}_{11}+\widetilde{E}_{12}+\widetilde{E}_{21}+\widetilde{E}_{22}+\widetilde{E}_{31}+\widetilde{E}_{32}\big)\Big)\cdot\widetilde{C}\geqslant 24,
$$
which is absurd. So, we conclude that $\varpi(\widetilde{C})$ is one of the curves  $\mathcal{L}_6^{\prime}$, $\mathcal{L}_6^{\prime\prime}$, $\mathcal{L}_6^{\prime\prime\prime}$,
$\mathcal{L}_6^{\prime\prime\prime\prime}$.
\end{proof}

\section{The proof of Main Theorem}
\label{section:proof}

Let us use assumptions and notations of Sections~\ref{section:P3-large} and \ref{section:Fano-Enriques}.
In particular, we have
\begin{center}
$P_1=[1:0:0:0]$, $P_2=[0:1:0:0]$, $P_3=[0:0:1:0]$, $P_4=[0:0:0:1]$,
\end{center}
and $G$ is a finite subgroup in $\mathrm{PGL}_{4}(\mathbb{C})$ such that
\begin{enumerate}
\item $G$ does not have fixed points in  $\mathbb{P}^3$,
\item $G$ does not leave a union of two skew lines in $\mathbb{P}^3$ invariant,
\item $G$ leaves invariant the subset $\{P_1,P_2,P_3,P_4\}$,
\item $G$ is conjugate neither to $G_{48,3}$ nor to $G_{96,72}$.
\end{enumerate}
If the~subgroup $G$ is conjugate to a subgroup among
$G_{48,50}$, $G_{96,70}$, $G_{96,227}$, $G_{96,227}^\prime$, $G_{192,955}$, $G_{192,185}$, $G_{324,160}^\prime$,
then we will always assume that $G$ is this subgroup.

If  $G$ is not conjugate to $G_{48,50}$ and $G_{96,227}$,
then $\Sigma_4$ is the unique $G$-orbit of length~four.
On the other hand, if $G=G_{48,50}$ or $G=G_{96,227}$, then the projective space
$\mathbb{P}^3$ contains two additional orbits of length~four:
$\Sigma_4^\prime$ and $\Sigma_4^{\prime\prime}$, which are described in Sections~\ref{section:P3-48} and \ref{section:Fano-Enriques}.
Note that $\Sigma_4$, $\Sigma_{4}^\prime$, $\Sigma_4^{\prime\prime}$ are transitively permuted by
the~following element of order three:
\begin{equation}
\label{equation:48-R}
R=\frac{1}{2}\begin{pmatrix}
1 & 1 & 1 & 1\\
1 & 1 & -1 & -1\\
1 &-1 & 1 & -1\\
-1 & 1 & 1 & -1
\end{pmatrix}\in G_{576,8654},
\end{equation}
where $G_{576,8654}$ is the subgroup in  $\mathrm{PLG}_4(\mathbb{C})$ generated by
$$
\begin{pmatrix}
-1&0&0&0\\
0&1&0&0\\
0&0&1&0\\
0&0&0&1
\end{pmatrix},
\begin{pmatrix}
0&0&0&1\\
1&0&0&0\\
0&1&0&0\\
0&0&1&0
\end{pmatrix},
\begin{pmatrix}
0&1&0&0\\
1&0&0&0\\
0&0&1&0\\
0&0&0&1
\end{pmatrix},R.
$$
By Lemma~\ref{lemma:48-50-normalizer},  the subgroup  $G_{576,8654}$ is the normalizer of the~groups $G_{48,50}$ and $G_{96,227}$.

\begin{remark}
\label{remark:three-involutions}
In Section~\ref{section:Fano-Enriques}, we have constructed a non-biregular involution $\iota\in\mathrm{Bir}^G(\mathbb{P}^3)$.
Moreover, if $G=G_{48,50}$ or $G=G_{96,227}$, we can choose $\iota$ such that it is given by
$$
[x_0:x_1:x_2:x_3]\mapsto [x_1x_2x_3:x_0x_2x_3:x_0x_1x_3:x_0x_1x_2].
$$
In these two cases, the group $\mathrm{Bir}^G(\mathbb{P}^3)$ also contains two birational involutions $\iota^\prime$ and $\iota^{\prime\prime}$,
which can be defined as follows: $\iota^\prime=R\circ\iota\circ R^2$ and $\iota^{\prime\prime}=R^2\circ\iota\circ R$.
Note that the~biratinal involution $\iota^\prime$ maps $[x_0:x_1:x_2:x_3]$ to the~point
\begin{multline*}
\big[x_0^3-(x_1^2+x_2+x_3^2)x_0-2x_1x_2x_3:x_1^3-(x_0^2+x_2^2+x_3^2)x_1-2x_0x_3x_2:\\
:x_2^3-(x_0^2+x_1^2+x_3^2)x_2-2x_0x_3x_1:x_3^3-(x_0^2+x_1^2+x_2^2)x_3-2x_1x_2x_0\big].
\end{multline*}
Similarly, the birational involution $\iota^{\prime}$ maps $[x_0:x_1:x_2:x_3]$ to the~point
\begin{multline*}
\big[x_0^3-(x_1^2+x_2^2+x_3^2)x_0+2x_1x_2x_3:x_1^3-(x_0^2+x_2^2+x_3^2)x_1+2x_0x_3x_2:\\
:x_2^3-(x_0^2+x_1^2+x_3^2)x_2+2x_0x_3x_1:x_3^3-(x_0^2+x_1^2+x_2^2)x_3+2x_1x_2x_0\big].
\end{multline*}
If $G=G_{48,50}$ or $G=G_{96,227}$, then
$\langle \iota,\iota^\prime,\iota^{\prime\prime}\rangle\triangleleft\langle \iota,G_{576,8654}\rangle$, where $\langle \iota,G_{576,8654}\rangle\subset\mathrm{Bir}^{G}(\mathbb{P}^3)$.
\end{remark}

Let $\Gamma$ be the subgroup in $\mathrm{Bir}^G(\mathbb{P}^3)$ generated by the involution $\iota$ described in Section~\ref{section:Fano-Enriques}
and the normalizer of the group $G$ in $\mathrm{PGL}_4(\mathbb{C})$.
We will see later that $\Gamma=\mathrm{Bir}^G(\mathbb{P}^3)$.
Let~$\varphi\circ\varpi^{-1}\colon\mathbb{P}^3\dasharrow X_{24}$ be the $G$-birational map from the~commutative diagram~\eqref{equation:Cremona-involution-2},
where $X_{24}$ is the toric Fano--Enriques threefold from Example~\ref{example:Fano-Enriques}.

\begin{theorem}
\label{theorem:NFI}
Suppose that for every non-empty $G$-invariant linear system $\mathcal{M}$ on the~projective space $\mathbb{P}^3$
that does not have fixed components,
there exists $\rho\in\Gamma$ such that one of the log pairs $(\mathbb{P}^3,\lambda_{\rho}\rho(\mathcal{M}))$
or $(X_{24},\lambda_{\varrho}\varrho(\mathcal{M}))$ has at most canonical singularities,
where $\varrho=\varphi\circ\varpi^{-1}\circ\rho$, and $\lambda_{\rho}$ and $\lambda_{\varrho}$ are positive rational numbers defined by
$$
\left\{\aligned
&\lambda_{\rho}\rho(\mathcal{M})\sim_{\mathbb{Q}}-K_{\mathbb{P}^3},\\
&\lambda_{\varrho}\varrho(\mathcal{M})\sim_{\mathbb{Q}}-K_{X_{24}}.\\
\endaligned
\right.
$$
Then $\mathbb{P}^3$ and $X_{24}$ are the only $G$-Mori fibred spaces that are $G$-birational to the~space $\mathbb{P}^3$.
Moreover, one also has $\mathrm{Bir}^G(\mathbb{P}^3)=\Gamma$.
\end{theorem}

\begin{proof}
The proof is essentially the same as the proof of \cite[Theorem~3.3.1]{CheltsovShramov}.
\end{proof}

To apply Theorem~\ref{theorem:NFI}, we need two technical results about $\mathbb{P}^3$ and $X_{24}$.
As in Section~\ref{section:P3-48}, let $\mathcal{L}_6$, $\mathcal{L}_6^\prime$, $\mathcal{L}_6^{\prime\prime}$
be the curves in $\mathbb{P}^3$ that consist of six lines in $\mathbb{P}^3$ that contain two points in $\Sigma_4$, $\Sigma_4^\prime$, $\Sigma_4^{\prime\prime}$, respectively.
Two technical results we need are Propositions~\ref{proposition:technical-1} and \ref{proposition:technical-2}.

\begin{proposition}
\label{proposition:technical-1}
Let $\mathcal{M}$ be a non-empty $G$-invariant linear system $\mathcal{M}$ on $\mathbb{P}^3$ that does not have fixed components,
let $\lambda$ be a positive rational number such that $\lambda\mathcal{M}\sim_{\mathbb{Q}}-K_{\mathbb{P}^3}$.
Suppose that $(\mathbb{P}^3,\lambda\mathcal{M})$ is not canonical.
If $G$ is not conjugate to $G_{48,50}$, $G_{96,227}$, $G_{324,160}^\prime$, then
$\mathrm{mult}_{\mathcal{L}_6}(\lambda\mathcal{M})>1$ or $\mathrm{mult}_{\Sigma_4}(\lambda\mathcal{M})>2$.
Similarly, if $G=G_{48,50}$ or $G=G_{96,227}$, then
\begin{equation}
\label{equation::technical-1-curves}
\mathrm{max}\Big(\mathrm{mult}_{\mathcal{L}_6}\big(\lambda\mathcal{M}\big),\mathrm{mult}_{\mathcal{L}_6^{\prime}}\big(\lambda\mathcal{M}\big),\mathrm{mult}_{\mathcal{L}_6^{\prime\prime}}\big(\lambda\mathcal{M}\big)\Big)>1
\end{equation}
or
\begin{equation}
\label{equation::technical-1-points}
\mathrm{max}\Big(\mathrm{mult}_{\Sigma_4}\big(\lambda\mathcal{M}\big),\mathrm{mult}_{\Sigma_4^\prime}\big(\lambda\mathcal{M}\big),\mathrm{mult}_{\Sigma_4^{\prime\prime}}\big(\lambda\mathcal{M}\big)\Big)>2.
\end{equation}
Finally, if $G=G_{324,160}^\prime$, then $\mathrm{mult}_{\mathcal{L}_6}(\lambda\mathcal{M})>1$ or
$\mathrm{mult}_{\Sigma_4}(\lambda\mathcal{M})>2$ or $\mathrm{mult}_{\mathfrak{C}}(\lambda\mathcal{M})>1$,
where $\mathfrak{C}$ is one of the $G$-invariant irreducible curves \eqref{equation:48-curve-degree-9-1} or \eqref{equation:48-curve-degree-9-2}.
\end{proposition}

\begin{proof}
Let $P$ be a point in the~$G$-orbit~$\Sigma_4$.
Then the~group $\mathrm{Stab}_G(P)$ faithfully and linearly acts on the Zariski tangent space $T_P(\mathbb{P}^3)$,
and this action is an irreducible representation.
Therefore, if $P$ is a center of non-canonical singularities of the log pair $(\mathbb{P}^3,\lambda\mathcal{M})$,
then
$$
\mathrm{mult}_{\Sigma_{4}}\big(\lambda\mathcal{M}\big)>2
$$
by \cite[Lemma~2.4]{ACPS}.
Thus, we may assume that no point in $\Sigma_{4}$ is a center of non-canonical singularities of the pair $(\mathbb{P}^3,\lambda\mathcal{M})$.
Likewise, if $G=G_{48,50}$ or $G=G_{96,227}$, then we may assume that
no point in $\Sigma_4^\prime\cup\Sigma_4^{\prime\prime}$
is a center of non-canonical singularities of our log pair.

If $G$ is conjugate to none of the groups $G_{48,50}$, $G_{96,70}$, $G_{96,227}$, $G_{96,227}^\prime$, $G_{192,955}$, $G_{324,160}^\prime$,
it follows from Corollaries~\ref{corollary:P3-192-curves} and \ref{corollary:P3-tetrahedra-curves}
that
$$
\mathrm{mult}_C\big(\lambda\mathcal{M}\big)\leqslant 1
$$
for every $G$-irreducible curve $C\subset\mathbb{P}^3$ such that $C\ne \mathcal{L}_6$.
If $G=G_{48,50}$ or $G=G_{96,227}$, then it~follows from Proposition~\ref{proposition:48-curves}
that we have $\mathrm{mult}_C(\lambda\mathcal{M})\leqslant 1$ for every $G$-irreducible curve $C$
which is different from the $G$-irreducible curves $\mathcal{L}_6$, $\mathcal{L}_6^\prime$, $\mathcal{L}_6^{\prime\prime}$.
If $G=G_{324,160}^\prime$, then it follows from Corollary~\ref{corollary:P3-tetrahedra-curves}
that $\mathrm{mult}_C(\lambda\mathcal{M})\leqslant 1$ for every $G$-irreducible curve $C\subset\mathbb{P}^3$
such that $C$ is not one of the curves $\mathcal{L}_6$, \eqref{equation:48-curve-degree-9-1} or \eqref{equation:48-curve-degree-9-2}.

Observe that $G_{96,70}$, $G_{96,227}^\prime$, $G_{192,955}$ swap the $G_{48,50}$-irreducible
curves $\mathcal{L}_6^\prime$ and $\mathcal{L}_6^{\prime\prime}$.
Therefore, if $G$ is one of these three groups, then it follows from Proposition~\ref{proposition:48-curves}
that we also have $\mathrm{mult}_C(\lambda\mathcal{M})\leqslant 1$ for every $G$-irreducible curve $C\subset\mathbb{P}^3$
that is different from  $\mathcal{L}_6$.

Thus, to complete the proof, we may assume that $\mathrm{mult}_C(\lambda\mathcal{M})\leqslant 1$ for every $C\subset\mathbb{P}^3$.
Then $(\mathbb{P}^3,\lambda\mathcal{M})$ is canonical outside of finitely many points by \cite[Theorem~4.5]{KoMo98}.

Let $P$ be a point in $\mathbb{P}^3$ such that $(\mathbb{P}^3,\lambda\mathcal{M})$ is not canonical at $P$.
Then every point in the~orbit $\mathrm{Orb}_G(P)$ must be a center of non-canonical singularities of the log pair $(\mathbb{P}^3,\lambda\mathcal{M})$.
Recall that $P\not\in\Sigma_4$.
Similarly, if $G=G_{48,50}$ or $G=G_{96,227}$, then $P\not\in\Sigma_4^\prime\cup\Sigma_4^{\prime\prime}$.

Now, we claim that $|\mathrm{Orb}_G(P)|\geqslant 12$.
Indeed, if $G=G_{48,50}$~or~$G=G_{96,227}$, this follows from~Lemma~\ref{lemma:48-orbits}.
Similarly, if we have $G=G_{192,185}$, then $|\mathrm{Orb}_G(P)|\geqslant 12$ by Lemma~\ref{lemma:192-orbits}.
If $G$ is not conjugate to any group among $G_{48,50}$, $G_{96,70}$, $G_{96,227}$, $G_{96,227}^\prime$, $G_{192,955}$, $G_{192,185}$,
then $|\mathrm{Orb}_G(P)|\geqslant 12$ by Lemma~\ref{lemma:orbits}.
If $G$ is one of the subgroups $G_{96,227}$, $G_{96,227}^\prime$, $G_{192,955}$, and $|\mathrm{Orb}_G(P)|\leqslant 12$,
then it follows from Lemma~\ref{lemma:48-orbits} that 
$$
\mathrm{Orb}_G(P)=\Sigma_4^\prime\cup\Sigma_4^{\prime\prime}.
$$

Let $\upsilon\colon V\to\mathbb{P}^3$ be the blow up of the points $\Sigma_4^\prime\cup\Sigma_4^{\prime\prime}$,
let $F$ be the sum of all $\upsilon$-exceptional surfaces, and
let $\widetilde{M}$ be the proper transform on $V$ of a sufficiently general surface in $\mathcal{M}$.
Note that the linear system $|\upsilon^*(\mathcal{O}_{\mathbb{P}^3}(2))-F|$ is two-dimensional and has no base points.
Let $S_1$ and $S_2$ be general surfaces in this linear system.
If $\mathrm{Orb}_G(P)=\Sigma_4^\prime\cup\Sigma_4^{\prime\prime}$, then
$$
0\leqslant \lambda\widetilde{M}\cdot S_1\cdot S_2=16-8\mathrm{mult}_P\big(\lambda\mathcal{M}\big),
$$
which is impossible, since we already proved that $\mathrm{mult}_P(\lambda\mathcal{M})>2$,
because the linear system $\mathcal{M}$ is $G_{48,50}$-invariant.
Therefore, we see that $|\mathrm{Orb}_G(P)|\geqslant 12$.

Now, we claim that $P$ is not contained in a $G$-invariant curve in $\mathbb{P}^3$ of degree at most~$8$.
To prove this claim, we may assume that $G=G_{48,50}$ or $G=G_{192,185}$
or $G$ is conjugate to none of the finite subgroups $G_{48,50}$, $G_{96,70}$, $G_{96,227}$, $G_{96,227}^\prime$, $G_{192,955}$, $G_{192,185}$,
because the~subgroups $G_{96,70}$, $G_{96,227}$, $G_{96,227}^\prime$, $G_{192,955}$ contain the~subgroup $G_{48,50}$.

Let $C$ be some $G$-irreducible curve in $\mathbb{P}^3$ of degree $d\leqslant 8$.
If  $G=G_{48,50}$, then it follows from
Corollary~\ref{corollary:48-curves-in-tetrahedra} and Lemmas~\ref{lemma:48-reducible-Q1-curves},
\ref{lemma:48-reducible-G-invariant-curves}, \ref{lemma:48-Q1-curves} and \ref{lemma:48-irreducible-curves-simple} that
either $C$ is a smooth irreducible $G$-invariant curve described in Example~\ref{example:48-curve-degree-8},
or $C$ is one of the curves
$$
\mathcal{L}_4, \mathcal{L}_4^\prime, \mathcal{L}_4^{\prime\prime}, \mathcal{L}_4^{\prime\prime\prime},
\mathcal{L}_6, \mathcal{L}_6^\prime, \mathcal{L}_6^{\prime\prime}, \mathcal{L}_6^{\prime\prime\prime},
\mathcal{L}_6^{\prime\prime\prime\prime}, \mathcal{C}_8^{1}, \mathcal{C}_8^{2}, \mathcal{C}_8^{3},
\mathcal{C}_8^{1,\prime}, \mathcal{C}_8^{2,\prime}, \mathcal{C}_8^{3,\prime},
\mathcal{C}_8^{1,\prime\prime}, \mathcal{C}_8^{2,\prime\prime}, \mathcal{C}_8^{3,\prime\prime}
$$
described in Section~\ref{section:P3-48}.
Similarly, if $G=G_{192,185}$, then $C$ is one of the curves $\mathcal{L}_6$,  $\mathcal{C}_8$,~$\mathscr{C}_8$,
which are described in Proposition~\ref{proposition:P3-192-curves}.
Finally, if the~group $G$ is not conjugate to a~group among $G_{48,50}$, $G_{96,70}$, $G_{96,227}$, $G_{96,227}^\prime$, $G_{192,955}$, $G_{192,185}$,
then $C=\mathcal{L}_6$~by~Proposition~\ref{proposition:P3-tetrahedra-curves}.
Among all these curves, only the curves $\mathcal{C}_8^{1}$, $\mathcal{C}_8^{1,\prime}$, $\mathcal{C}_8^{1,\prime\prime}$ are singular.

Let $\mathcal{D}$ be the linear system on $\mathbb{P}^3$ consisting of surfaces of degree $k$  that contain $C$,~where
$$
k=\left\{\aligned
&3\ \text{if $C$ is one of the curves $\mathcal{L}_6$, $\mathcal{L}_6^\prime$, $\mathcal{L}_6^{\prime\prime}$},\\
&4\ \text{if $C$ is one of the curves $\mathcal{L}_4$, $\mathcal{L}_4^\prime$, $\mathcal{L}_4^{\prime\prime}$, $\mathcal{L}_4^{\prime\prime\prime}$},\\
&4\ \text{if $C$ is one of the curves $\mathcal{C}_8^{1}$, $\mathcal{C}_8^{2}$, $\mathcal{C}_8^{3}$, $\mathcal{C}_8^{1,\prime}$, $\mathcal{C}_8^{2,\prime}$, $\mathcal{C}_8^{3,\prime}$, $\mathcal{C}_8^{1,\prime\prime}$, $\mathcal{C}_8^{2,\prime\prime}$, $\mathcal{C}_8^{3,\prime\prime}$},\\
&4\ \text{if $C$ is a smooth irreducible curve described in Example~\ref{example:48-curve-degree-8}},\\
&4\ \text{if $C$ is the curve $\mathscr{C}_8$ described in Proposition~\ref{proposition:P3-192-curves}},\\
&6\ \text{if $C$ is one of the curves $\mathcal{L}_6^{\prime\prime\prime}$ or $\mathcal{L}_6^{\prime\prime\prime\prime}$ described in Section~\ref{section:P3-48}},\\
&8\ \text{if $C$ is the curve $\mathcal{C}_8$ described in Proposition~\ref{proposition:P3-192-curves}}.
\endaligned
\right.
$$
Then the linear system $\mathcal{D}$ is not empty. Moreover, it does not have fixed components.
Furthermore, if $C$ is not one of the curves $\mathcal{C}_8^{2}$, $\mathcal{C}_8^{3}$, $\mathcal{C}_8^{2,\prime}$, $\mathcal{C}_8^{3,\prime}$, $\mathcal{C}_8^{2,\prime\prime}$, $\mathcal{C}_8^{3,\prime\prime}$,
then $\mathcal{D}$ does not have base points away from the curve $C$.
If $C$ is one of the curves $\mathcal{C}_8^{2}$, $\mathcal{C}_8^{3}$, $\mathcal{C}_8^{2,\prime}$, $\mathcal{C}_8^{3,\prime}$, $\mathcal{C}_8^{2,\prime\prime}$,~$\mathcal{C}_8^{3,\prime\prime}$,
we can describe $\mathcal{D}$ explicitly. For instance, if $C=\mathcal{C}_8^{2}$ or $C=\mathcal{C}_8^{3}$,
then $\mathcal{D}$ is the pencil
$$
\lambda x_0x_1x_2x_3+\mu(x_0^2x_1^2+x_0^2x_2^2+x_0^2x_3^2+x_1^2x_2^2+x_1^2x_3^2+x_2^2x_3^2)-\mu(x_0^4+x_1^4+x_2^4+x_3^4)=0,
$$
where $[\lambda:\mu]\in\mathbb{P}^1$.
Note that the base locus of this pencil consists of the curves $\mathcal{C}_8^{2}$~and~$\mathcal{C}_8^{3}$.
Similarly, if $C=\mathcal{C}_8^{2,\prime}$ or $C=\mathcal{C}_8^{3,\prime}$,
then the linear system $\mathcal{D}$ is a pencil of quartic surfaces, and its base locus is the union $\mathcal{C}_8^{2,\prime}\cup\mathcal{C}_8^{3,\prime}$.
Finally, if $C=\mathcal{C}_8^{2,\prime\prime}$ or $C=\mathcal{C}_8^{3,\prime\prime}$,
then the linear system $\mathcal{D}$ is a pencil of quartic surfaces whose base locus is the~union $\mathcal{C}_8^{2,\prime\prime}\cup \mathcal{C}_8^{3,\prime\prime}$.

Now, we suppose that $P\in C$, hence $\mathrm{Orb}_G(P)\subset C$ as well.
Let $M_1$ and $M_2$ be two general surfaces in $\mathcal{M}$. Write
$$
\lambda^2 M_1\cdot M_2=mC+\Delta,
$$
where $m$ is a non-negative rational number, and $\Delta$ is an effective one-cycle whose support does not contain  $C$.
Then $m\leqslant 4$, since $\lambda^2 M_1\cdot M_2$ is a one-cycle of degree $16$, and $d\geqslant 4$.
On the other hand, it follows from \cite{Pukhlikov} or \cite[Corollary 3.4]{Co00} that $\lambda^2(M_1\cdot M_2)_P>4$.
Therefore, if the~curve $C$ is smooth at $P$, then
$$
\mathrm{mult}_P\big(\Delta\big)>4-m.
$$
Let $S$ be a general surface in $\mathcal{D}$.
If $C$ is not one of the curves $\mathcal{C}_8^{2}$, $\mathcal{C}_8^{3}$, $\mathcal{C}_8^{2,\prime}$, $\mathcal{C}_8^{3,\prime}$, $\mathcal{C}_8^{2,\prime\prime}$, $\mathcal{C}_8^{3,\prime\prime}$,
then the base locus of the linear system $\mathcal{D}$ does not contain curves different from $C$, which implies that
$S$ does not contains curves in the support of the one-cycle $\Delta$, hence \begin{equation}
\label{equation:the-proof-inequality}
16k-kdm=S\cdot\Delta\geqslant|\mathrm{Orb}_G(P)|\mathrm{mult}_P\big(\Delta\big)>|\mathrm{Orb}_G(P)|(4-m)\geqslant 12(4-m)
\end{equation}
provided that the curve $C$ is smooth at~$P$.
This immediately gives us a contradiction in the~case when $C$ is one of the curves $\mathcal{L}_6$, $\mathcal{L}_6^\prime$, $\mathcal{L}_6^{\prime\prime}$.
Thus, we conclude that $P\not\in\mathcal{L}_6\cup\mathcal{L}_6^\prime\cup\mathcal{L}_6^{\prime\prime}$.
In particular, we obtain our local claim in the case when $G$ is not conjugate to any group among $G_{48,50}$, $G_{96,70}$, $G_{96,227}$, $G_{96,227}^\prime$, $G_{192,955}$, $G_{192,185}$.
Thus, to proceed, we may assume that either $G=G_{48,50}$ or $G=G_{192,185}$.

If $G=G_{192,185}$ and $C$ is the curve $\mathscr{C}_8$ described in Proposition~\ref{proposition:P3-192-curves},
then it follows from the~inequality \eqref{equation:the-proof-inequality} and Lemma~\ref{lemma:192-orbits} that
$$
64-32m>|\mathrm{Orb}_G(P)|(4-m)\geqslant 16(4-m)=64-16m,
$$
which is a contradiction.
If $G=G_{192,185}$ and $C$ is the curve $\mathcal{C}_8$ described in Proposition~\ref{proposition:P3-192-curves},
then \eqref{equation:the-proof-inequality} implies that $128-64m>|\mathrm{Orb}_G(P)|(4-m)$,
so that we have $|\mathrm{Orb}_G(P)|<32$. Recall~from Proposition~\ref{proposition:P3-192-curves}
that the curve $\mathcal{C}_8$ is a disjoint union of four irreducible conics, and $C_1=\{x_3=x_0^2-x_1^2-x_2^2=0\}$ is one of them.
Then $\mathrm{Stab}_{G_{192,185}}(C_1)$ is generated by
$$
(-1,1,1),(1,-1,1),(1,1,-1),
\begin{pmatrix}
0 & 0 & 0 & i\\
i & 0 & 0 & 0\\
0 & 1 & 0 & 0\\
0 & 0 & 1 & 0
\end{pmatrix},
\begin{pmatrix}
0 & 1 & 0 & 0\\
1 & 0 & 0 & 0\\
0 & 0 & i & 0\\
0 & 0 & 0 & 1
\end{pmatrix},
$$
so that $\mathrm{Stab}_{G_{192,185}}(C_1)\cong\mathfrak{A}_4\rtimes\mumu_4$,
and the~$\mathrm{Stab}_{G_{192,185}}(C_1)$-action on the curve $C_1$ gives a~homomorphism
$\mathrm{Stab}_{G_{192,185}}(C_1)\to\mathrm{Aut}(C_1)$,
whose image is isomorphic to $\mathfrak{S}_4$. Then
\begin{itemize}
\item the curve $C_1$ has a unique $\mathrm{Stab}_{G_{192,185}}(C_1)$-orbit of~length~$6$,
\item the unique $\mathrm{Stab}_{G_{192,185}}(C_1)$-orbit in $C_1$ of length $6$ is $C_1\cap(\ell_{12}\cup\ell_{13}\cup\ell_{23})$,
\item other $\mathrm{Stab}_{G_{192,185}}(C_1)$-orbits in $C_1$ have length at least $8$.
\end{itemize}
Therefore, we conclude that $\mathcal{C}_8\cap\mathcal{L}_6$ is the unique $G_{192,185}$-orbit in $\mathcal{C}_8$ that has length $24$,
and other $G_{192,185}$-orbits in $\mathcal{C}_8$ has length at least $32$.
Hence, if $G=G_{192,185}$ and $C=\mathcal{C}_8$, then $P\in\mathcal{C}_8\cap\mathcal{L}_6$,
which is impossible, since we already proved that $P\not\in\mathcal{L}_6$.

Thus, we have $G=G_{48,50}$.
Recall that we already proved that $P\not\in\mathcal{L}_6\cup\mathcal{L}_6^\prime\cup\mathcal{L}_6^{\prime\prime}$.
Therefore, either $C$ is a smooth irreducible curve of degree $8$ described in Example~\ref{example:48-curve-degree-8},
or $C$ is one of the curves
$\mathcal{L}_4$, $\mathcal{L}_4^\prime$, $\mathcal{L}_4^{\prime\prime}$, $\mathcal{L}_4^{\prime\prime\prime}$,
$\mathcal{L}_6^{\prime\prime\prime}$,
$\mathcal{L}_6^{\prime\prime\prime\prime}$, $\mathcal{C}_8^{1}$, $\mathcal{C}_8^{2}$, $\mathcal{C}_8^{3}$,
$\mathcal{C}_8^{1,\prime}$, $\mathcal{C}_8^{2,\prime}$, $\mathcal{C}_8^{3,\prime}$,
$\mathcal{C}_8^{1,\prime\prime}$, $\mathcal{C}_8^{2,\prime\prime}$, $\mathcal{C}_8^{3,\prime\prime}$.
In the former case, it follows from \eqref{equation:the-proof-inequality} and Lemma~\ref{lemma:48-curves} that
$$
64-32m>|\mathrm{Orb}_G(P)|(4-m)\geqslant 16(4-m),
$$
which is absurd. Similarly, if $C$ is one of the curves $\mathcal{L}_4$, $\mathcal{L}_4^\prime$, $\mathcal{L}_4^{\prime\prime}$, $\mathcal{L}_4^{\prime\prime\prime}$,
then \eqref{equation:the-proof-inequality} gives
$$
64-16m>|\mathrm{Orb}_G(P)|(4-m),
$$
which implies that $|\mathrm{Orb}_G(P)|<16$, hence it follows from Lemma~\ref{lemma:48-orbits}
that $P$ is contained in one of the four $G_{48,50}$-orbits $\Sigma_{12}$, $\Sigma_{12}^\prime$,
$\Sigma_{12}^{\prime\prime}$, $\Sigma_{12}^{\prime\prime\prime}$, which are described earlier in Section~\ref{section:P3-48}.
But $\Sigma_{12}\cup\Sigma_{12}^\prime\subset\mathcal{L}_6$,
$\Sigma_{12}^{\prime\prime}\subset\mathcal{L}_6^{\prime\prime}$,
and $\Sigma_{12}^{\prime\prime\prime}\subset\mathcal{L}_6^{\prime\prime\prime}$,
which is impossible, since $P\not\in\mathcal{L}_6\cup\mathcal{L}_6^\prime\cup\mathcal{L}_6^{\prime\prime}$.
Likewise, if either $C=\mathcal{L}_6^{\prime\prime\prime}$ or $C=\mathcal{L}_6^{\prime\prime\prime\prime}$,
then \eqref{equation:the-proof-inequality} and Lemma~\ref{lemma:48-orbits} give $|\mathrm{Orb}_G(P)|=16$, because $P\not\in\Sigma_{12}\cup\Sigma_{12}^\prime\cup\Sigma_{12}^{\prime\prime}\cup\Sigma_{12}^{\prime\prime\prime}$.
But $\mathcal{L}_6^{\prime\prime\prime}$ and $\mathcal{L}_6^{\prime\prime\prime\prime}$ contain no $G_{48,50}$-orbits of length~$16$.
Thus, as above, we conclude that $C\ne \mathcal{L}_6^{\prime\prime\prime}$ and $C\ne\mathcal{L}_6^{\prime\prime\prime\prime}$.

If $C$ is one of the curves $\mathcal{C}_8^{1}$, $\mathcal{C}_8^{1,\prime}$, $\mathcal{C}_8^{1,\prime\prime}$,
then $C$ is smooth at the~point $P$, because the singular loci of the curves $\mathcal{C}_8^{1}$, $\mathcal{C}_8^{1,\prime}$, $\mathcal{C}_8^{1,\prime\prime}$
are contained in the curves $\mathcal{L}_6$, $\mathcal{L}_6^\prime$, $\mathcal{L}_6^{\prime\prime}$, respectively.
Therefore, in this case, it follows from \eqref{equation:the-proof-inequality} that
$$
64-32m>|\mathrm{Orb}_G(P)|(4-m)
$$
which implies that $|\mathrm{Orb}_G(P)|<16$, hence $P\not\in\Sigma_{12}\cup\Sigma_{12}^\prime\cup\Sigma_{12}^{\prime\prime}\cup\Sigma_{12}^{\prime\prime\prime}$ by Lemma~\ref{lemma:48-orbits}.
But~$P\not\in\Sigma_{12}\cup\Sigma_{12}^\prime\cup\Sigma_{12}^{\prime\prime}\cup\Sigma_{12}^{\prime\prime\prime}$,
so that $C$ is not one of the curves $\mathcal{C}_8^{1}$, $\mathcal{C}_8^{1,\prime}$, $\mathcal{C}_8^{1,\prime\prime}$.

We see that $C$ is one of the curves $\mathcal{C}_8^{2}$, $\mathcal{C}_8^{3}$,
$\mathcal{C}_8^{2,\prime}$, $\mathcal{C}_8^{3,\prime}$,
$\mathcal{C}_8^{2,\prime\prime}$, $\mathcal{C}_8^{3,\prime\prime}$.
Without loss of generality, we~may assume that $C=\mathcal{C}_8^{2}$,
because $G_{96,227}$ and $G_{144,184}$ transitively permutes these six curves.
Recall that $\mathcal{D}$ is a pencil, and its base locus consists of the curves $C=\mathcal{C}_8^{2}$ and $\mathcal{C}_8^{3}$.
As above, we write
$$
\lambda^2 M_1\cdot M_2=mC+m^\prime \mathcal{C}_8^{3}+\Delta^\prime,
$$
where $m^\prime$ is a non-negative rational number, and $\Delta^\prime$ is an effective one-cycle whose support contains none of the curves $\mathcal{C}_8^{2}$ and $\mathcal{C}_8^{3}$.
Then $m+m^\prime\leqslant 2$, since $\lambda^2 M_1\cdot M_2$ has degree~$16$.
Since $\lambda^2\big(M_1\cdot M_2\big)_P>4$,
if $P\not\in\mathcal{C}_8^{2}\cap\mathcal{C}_8^{3}$, then  $\mathrm{mult}_P(\Delta)>4-m$, hence $$
64-32m\geqslant 64-32m-32m^\prime=S\cdot\Delta\geqslant|\mathrm{Orb}_G(P)|\mathrm{mult}_P\big(\Delta\big)>|\mathrm{Orb}_G(P)|(4-m)
$$
for a general surface $S\in\mathcal{D}$.
Therefore, if $P\not\in\mathcal{C}_8^{2}\cap\mathcal{C}_8^{3}$, then we have $|\mathrm{Orb}_G(P)|<16$,
which contradicts Lemma~\ref{lemma:48-orbits},
because we have $P\not\in\Sigma_4\cup\Sigma_{4}^\prime\cup\Sigma_{4}^{\prime\prime}\cup\Sigma_{12}\cup\Sigma_{12}^\prime\cup\Sigma_{12}^{\prime\prime}\cup\Sigma_{12}^{\prime\prime\prime}$.
Hence, we see that $P\not\in\mathcal{C}_8^{2}\cap\mathcal{C}_8^{3}$.
Then $\mathrm{Orb}_G(P)$ is the~$G_{48,50}$-orbit of the point $[1:1:1:0]$,
which gives $|\mathrm{Orb}_G(P)|=16$.
Observe that this $G_{48,50}$-orbit is cut out by cubic surfaces, because it is a singular locus of the surface
$$
\big\{x_0^2x_1^2+x_0^2x_2^2+x_0^2x_3^2+x_1^2x_2^2+x_1^2x_3^2+x_2^2x_3^2=x_0^4+x_1^4+x_2^4+x_3^4\big\}\subset\mathbb{P}^3.
$$
Thus, if $S_3$ is a general cubic surface in $\mathbb{P}^3$
that contains $\mathrm{Orb}_G(P)$, then $S_3$ does not contain curves that are contained in
the support of the one-cycle $\lambda^2 M_1\cdot M_2$, hence $$
48=\lambda^2 M_1\cdot M_2\cdot S_3\geqslant\sum_{O\in\mathrm{Orb}_G(P)}\big(\lambda^2 M_1\cdot M_2\big)_O>4|\mathrm{Orb}_G(P)|=64,
$$
which is absurd. So, we conclude that our point $P$ is not contained in any $G$-irreducible curve in $\mathbb{P}^3$ whose degree is at most $8$.

Observe that $(\mathbb{P}^3,\frac{3}{2}\lambda\mathcal{M})$ is not log canonical at $P$.
Let $\mu$ be the largest rational number such that  $(\mathbb{P}^3,\mu\mathcal{M})$ is log canonical at $P$.
Then $\mu<\frac{3}{2}\lambda$ and $\mathrm{Orb}_G(P)\subseteq\mathrm{Nklt}(\mathbb{P}^3,\mu\mathcal{M})$.
Observe that the locus $\mathrm{Nklt}(\mathbb{P}^3,\mu\mathcal{M})$ is at most one-dimensional, because $\mathcal{M}$ does not have fixed components.
Moreover, this locus is $G$-invariant, since $\mathcal{M}$ is $G$-invariant.

We claim that the locus $\mathrm{Nklt}(\mathbb{P}^3,\mu\mathcal{M})$ does not contain curves that passes through $P$.
Indeed, suppose this is not true. Then $\mathrm{Nklt}(\mathbb{P}^3,\mu\mathcal{M})$ contains a $G$-irreducible curve~$Z$~that passes
through $P$. As above, for two general surfaces $M_1$ and $M_2$ in  $\mathcal{M}$, we write
$$
\mu^2 M_1\cdot M_2=\delta Z+\Omega,
$$
where $\delta$ is a non-negative rational number, and $\Omega$ is an effective one-cycle whose support does not contain the curve $Z$.
Then $\delta\geqslant 4$ by \cite[Theorem~3.1]{Co00}.
Now, taking into account that the degree of the one-cycle $\mu^2 M_1\cdot M_2$ is less that $36$,
we conclude that $\mathrm{deg}(Z)<9$.
But~we already proved that $P$ is not contained in any $G$-irreducible curve in $\mathbb{P}^3$ whose degree is at most $8$.
Thus, the locus $\mathrm{Nklt}(\mathbb{P}^3,\mu\mathcal{M})$ contains no curves passing through~$P$,
so that this locus  does not contain curves that pass through any point in $\mathrm{Orb}_G(P)$.

Let $\mathcal{I}$ be the multiplier ideal sheaf of the pair $(\mathbb{P}^3,\mu\mathcal{M})$,
and let $\mathcal{L}$ be the corresponding subscheme in $\mathbb{P}^3$.
Applying \cite[Theorem~9.4.8]{Lazarsfeld}, we get
$h^1(\mathbb{P}^3,\mathcal{I}\otimes\mathcal{O}_{\mathbb{P}^3}(2))=0$.
Then
$$
10=h^0\big(\mathbb{P}^3,\mathcal{O}_{\mathbb{P}^3}(2)\big)\geqslant h^0\big(\mathcal{O}_{\mathcal{L}}\otimes\mathcal{O}_{\mathbb{P}^3}(2)\big)\geqslant |\mathrm{Orb}_G(P)|\geqslant 12,
$$
because the~subscheme $\mathcal{L}$ contains at least $|\mathrm{Orb}_G(P)|\geqslant 12$
disjoint zero-dimensional components, since $\mathrm{Orb}_G(P)\subseteq\mathrm{Nklt}(\mathbb{P}^3,\mu\mathcal{M})$,
and $\mathrm{Nklt}(\mathbb{P}^3,\mu\mathcal{M})$ does not contain curves that are not disjoint from $\mathrm{Orb}_G(P)$.
The obtained contradiction completes the proof.
\end{proof}

Recall from \eqref{equation:Cremona-involution-2} that we have the following $G$-commutative diagram:
$$
\xymatrix{
&\widetilde{X}_{24}\ar@{->}[ld]_{\varpi}\ar@{->}[rd]^{\varphi}\\%
\mathbb{P}^3\ar@{-->}[rr]_{\psi}&&X_{24}}
$$
where $\varpi$, $\varphi$ and $\psi$ are birational maps described in Section~\ref{section:Fano-Enriques}.

\begin{proposition}
\label{proposition:technical-2}
Let $\mathcal{M}$ be a non-empty $G$-invariant linear system on $X_{24}$ that has no fixed components.
Choose $\lambda\in\mathbb{Q}_{>0}$ such that $\lambda\mathcal{M}\sim_{\mathbb{Q}}-K_{X_{24}}$.
If $(X_{24},\lambda\mathcal{M})$ is canonical at every point in $\mathrm{Sing}(X_{24})$, then $(X_{24},\lambda\mathcal{M})$ is canonical.
\end{proposition}

\begin{proof}
Suppose that the singularities of the log pair $(X_{24},\lambda\mathcal{M})$ are not canonical,
and the log pair $(X_{24},\lambda\mathcal{M})$ is canonical at every point in $\mathrm{Sing}(X_{24})$.
Let us seek for a~contradiction.
Let $Z$ be a center of non-canonical singularities of the pair $(X_{24},\lambda\mathcal{M})$ that has the~largest dimension.
Since the linear system $\mathcal{M}$ does not have fixed components,
we conclude that either $Z$ is an irreducible curve, or $Z$ is a smooth point of the threefold~$X_{24}$.

Let $\mathcal{E}=E_{11}+E_{12}+E_{21}+E_{22}+E_{31}+E_{32}$, where
$E_{11}$, $E_{12}$, $E_{21}$, $E_{22}$, $E_{31}$, $E_{32}$ are~surfaces in the~threefold $X_{24}$ defined in Section~\ref{section:Fano-Enriques}.
If $Z\subseteq\mathcal{E}$, then $(E_{11},\lambda\mathcal{M}\vert_{E_{11}})$~is~not~log~canonical
by the~inversion of adjunction \cite[Theorem~5.50]{KoMo98}, which is impossible by Lemma~\ref{lemma:alpha-quartic},
because $\lambda\mathcal{M}\vert_{E_{11}}\sim_{\mathbb{Q}}-K_{E_{11}}$ by the~adjunction formula.
Thus, we conclude that $Z\not\subseteq\mathcal{E}$.

If $G$ is one of the groups $G_{48,50}$, $G_{96,70}$, $G_{96,227}$, $G_{96,227}^\prime$, $G_{192,955}$,
we can use Lemma~\ref{lemma:alpha-quadric-Heisenberg} to show that $Z$ is not contained in the~locus
\begin{equation}
\label{equation:six-quadrics}
\psi\big(\mathcal{Q}_5\big)\cup\psi\big(\mathcal{Q}_6\big)\cup\psi\big(\mathcal{Q}_7\big)\cup\psi\big(\mathcal{Q}_8\big)\cup\psi\big(\mathcal{Q}_9\big)\cup\psi\big(\mathcal{Q}_{10}\big),
\end{equation}
where $\mathcal{Q}_5$, $\mathcal{Q}_6$, $\mathcal{Q}_7$, $\mathcal{Q}_8$, $\mathcal{Q}_9$, $\mathcal{Q}_{10}$
are quadric surfaces in $\mathbb{P}^3$, which are defined in Section~\ref{section:P3-48}.
Indeed, suppose that $G$ is one of the~subgroups $G_{48,50}$, $G_{96,70}$, $G_{96,227}$, $G_{96,227}^\prime$, $G_{192,955}$,
and there is a surface $S$ among $\psi(\mathcal{Q}_5)$, $\psi(\mathcal{Q}_6)$,
$\psi(\mathcal{Q}_7)$, $\psi(\mathcal{Q}_8)$, $\psi(\mathcal{Q}_9)$, $\psi(\mathcal{Q}_{10})$
that contains~$Z$.
Recall that $G$ contains the subgroup $\mathbb{H}\cong\mumu_2^4$ defined in Section~\ref{section:P3-48},
the quadrics $\mathcal{Q}_5$, $\mathcal{Q}_6$, $\mathcal{Q}_7$, $\mathcal{Q}_8$, $\mathcal{Q}_9$, $\mathcal{Q}_{10}$ are $\mathbb{H}$-invariant,
and the subgroup $\mathbb{H}$ acts faithfully on each of them.
Furthermore, the rational map $\psi\colon\mathbb{P}^3\dasharrow X_{24}$ induces $\mathbb{H}$-equivariant
isomorphisms
$$
\mathcal{Q}_5\cong\psi\big(\mathcal{Q}_5\big), \mathcal{Q}_6\cong\psi\big(\mathcal{Q}_6\big), \mathcal{Q}_7\cong\psi\big(\mathcal{Q}_7\big),
\mathcal{Q}_8\cong\psi\big(\mathcal{Q}_8\big), \mathcal{Q}_9\cong\psi\big(\mathcal{Q}_9\big), \mathcal{Q}_{10}\cong\psi\big(\mathcal{Q}_{10}\big).
$$
Moreover, one can also check that
\begin{itemize}
\item $\psi(\mathcal{Q}_7)$ and $\psi(\mathcal{Q}_{10})$ are the fibers of the~morphism $\eta_1$ over $[1:-1]$ and $[1:1]$,
\item $\psi(\mathcal{Q}_5)$ and $\psi(\mathcal{Q}_{8})$ are the fibers of the~morphism $\eta_2$ over $[1:-1]$ and $[1:1]$,
\item $\psi(\mathcal{Q}_6)$ and $\psi(\mathcal{Q}_{9})$ are the fibers of the~morphism $\eta_3$ over $[1:-1]$ and $[1:1]$.
\end{itemize}
Thus, it follows from the~inversion of adjunction  that
$(S,\lambda\mathcal{M}\vert_{S})$ is not log canonical,
which is impossible by Lemma~\ref{lemma:alpha-quadric-Heisenberg},
because $\lambda\mathcal{M}\vert_{S}\sim_{\mathbb{Q}}-K_{S}$.

Now, we are ready to show that $Z$ is a point. Namely, we suppose that~$Z$ is~a~curve.
Let $H$ be a hyperplane section of the threefold $X_{24}\subset\mathbb{P}^{13}$,
let $M_1$ and $M_2$ be two general surfaces in the~linear system $\mathcal{M}$. Then $\mathrm{deg}(Z)<24$, since
$$
24=\lambda^2 H\cdot M_1\cdot M_2\geqslant \lambda^2\mathrm{deg}(Z)\big(M_1\cdot M_2\big)_Z\geqslant\mathrm{deg}(Z)\mathrm{mult}_Z^2\big(\lambda\mathcal{D}\big)>\mathrm{deg}(Z).
$$
Therefore, it follows from Lemma~\ref{lemma:X24-curves} that
$G$ is one of the groups $G_{48,50}$, $G_{96,70}$, $G_{96,227}$, and~$Z$ is one of the~curves $\psi(\mathcal{L}_6^\prime)$, $\psi(\mathcal{L}_6^{\prime\prime})$, $\psi(\mathcal{L}_6^{\prime\prime\prime})$, $\psi(\mathcal{L}_6^{\prime\prime\prime\prime})$,
where $\mathcal{L}_6^\prime$, $\mathcal{L}_6^{\prime\prime}$, $\mathcal{L}_6^{\prime\prime\prime}$, $\mathcal{L}_6^{\prime\prime\prime\prime}$
are curves in the~projective space $\mathbb{P}^3$ introduced~in~Section~\ref{section:P3-48}.
But~this is impossible, since all these curves are contained in \eqref{equation:six-quadrics}.
This shows that $Z$ is a point.

Observe that $(X_{24},\frac{3}{2}\lambda\mathcal{M})$ is not log canonical at $Z$.
Let $\mu$ be the largest rational number such that  $(X_{24},\mu\mathcal{M})$ is log canonical at $Z$.
Then $\mu<\frac{3}{2}\lambda$ and $\mathrm{Orb}_G(Z)\subseteq\mathrm{Nklt}(X_{24},\mu\mathcal{M})$.
Note that the locus $\mathrm{Nklt}(X_{24},\mu\mathcal{M})$ is at most one-dimensional, because $\mathcal{M}$ does not have fixed components.
Moreover, this locus is $G$-invariant, since $\mathcal{M}$ is $G$-invariant.

Now, we claim that the locus $\mathrm{Nklt}(X_{24},\mu\mathcal{M})$ does not contain curves passing through~$Z$.
Indeed, we suppose that the~locus $\mathrm{Nklt}(X_{24},\mu\mathcal{M})$ contains some $G$-irreducible curve~$C$.
As above, we let $M_{1}$ and $M_{2}$ be two general surfaces in  $\mathcal{M}$. Write
$$
\mu^2 M_1\cdot M_2=\delta C+\Omega,
$$
where $\delta$ is a non-negative rational number, and $\Omega$ is an effective one-cycle whose support does not contain the curve $C$.
Then $\delta\geqslant 4$ by \cite[Theorem~3.1]{Co00}.
Now, taking into account that the degree of the one-cycle $\mu^2 M_1\cdot M_2$ is less that $54$,
we conclude that $\mathrm{deg}(C)\leqslant 13$.
Therefore, it follows from Lemma~\ref{lemma:X24-curves} that
$G$ is one of the groups $G_{48,50}$, $G_{96,70}$, $G_{96,227}$, and~$C$ is one of the~curves
$\psi(\mathcal{L}_6^\prime)$, $\psi(\mathcal{L}_6^{\prime\prime})$, $\psi(\mathcal{L}_6^{\prime\prime\prime})$, $\psi(\mathcal{L}_6^{\prime\prime\prime\prime})$.
But all of these four curves are contained in the subset \eqref{equation:six-quadrics},
so that none of them contains $Z$, since $Z$ is not in \eqref{equation:six-quadrics}.

We conclude that all curves in $\mathrm{Nklt}(\mathbb{P}^3,\mu\mathcal{M})$ are disjoint from $\mathrm{Orb}_G(Z)$.

Let $\mathcal{I}$ be the multiplier ideal sheaf of the pair $(X_{24},\mu\mathcal{M})$,
and let $\mathcal{L}$ be the corresponding subscheme in $X_{24}$.
Applying \cite[Theorem~9.4.8]{Lazarsfeld}, we get $H^1(X_{24},\mathcal{I}\otimes\mathcal{O}_{X_{24}}(H))=0$. Then
$$
14=h^0\big(X_{24},\mathcal{O}_{X_{24}}(H)\big)\geqslant h^0\big(\mathcal{O}_{\mathcal{L}}\otimes\mathcal{O}_{X_{24}}(H)\big)\geqslant |\mathrm{Orb}_G(Z)|,
$$
since the~subscheme $\mathcal{L}$ contains at least $|\mathrm{Orb}_G(Z)|$
disjoint zero-dimensional components.
Therefore, since $Z\not\in\mathcal{E}$, it follows from Corollary~\ref{corollary:X24-orbits}
that either $G=G_{48,50}$ or $G=G_{96,227}$, and $\mathrm{Orb}_G(Z)$ is one of the orbits $\psi(\Sigma_4^\prime)$, $\psi(\Sigma_4^{\prime\prime})$,
$\psi(\Sigma_{12}^{\prime\prime})$, $\psi(\Sigma_{12}^{\prime\prime\prime})$,
which is a contradiction, because these orbits are contained in \eqref{equation:six-quadrics},
while $Z$ is not contained in this locus.
\end{proof}

Now, arguing as in the proof of \cite[Proposition~6.11]{CDK},
we can prove our Main Theorem using~Theorem~\ref{theorem:NFI}, Propositions~\ref{proposition:technical-1} and \ref{proposition:technical-2}.
Similarly, we can prove Theorem~\ref{theorem:324-160} using both Propositions~\ref{proposition:technical-1} and \ref{proposition:technical-2}
together with the following lemma:

\begin{lemma}
\label{lemma:X-10}
Let $\mathfrak{C}$ be one of the two $G^\prime_{324,160}$-invariant irreducible curves \eqref{equation:48-curve-degree-9-1} or \eqref{equation:48-curve-degree-9-2},
and let $\vartheta\colon X\to\mathbb{P}^3$ be the~blow up of the~curve $\mathfrak{C}$.
There is a $G^\prime_{324,160}$-equivariant~diagram
$$
\xymatrix{
&X\ar@{->}[ld]_{\vartheta}\ar@{->}[rd]^{\kappa}&\\%
\mathbb{P}^3&&\mathbb{P}^1}
$$
where $\kappa$ is a fibration into cubic surfaces.
Now, let $\mathcal{M}$ be a non-empty $G$-invariant linear system on $X$ that does not have fixed components such that
$$
K_{X}+\lambda\mathcal{M}\sim_{\mathbb{Q}}\kappa^*(D)
$$
for some $\lambda\in\mathbb{Q}_{>0}$, and some $\mathbb{Q}$-divisor $D$ on $\mathbb{P}^1$.
Then $(X,\lambda\mathcal{M})$ is canonical.
\end{lemma}

\begin{proof}
Suppose the pair $(X,\lambda\mathcal{M})$ is not canonical.
Let $Z$ be its center of non-canonical singularities. Then
$$
\mathrm{mult}_Z\big(\mathcal{M}\big)>\frac{1}{\lambda}
$$
by \cite[Theorem~4.5]{KoMo98} or \cite[Exercise~6.18]{CoKoSm03}.

First, we suppose that $Z$ is a curve that is not contained in the fibers of the morphism~$\kappa$.
Let $F$ be a general fiber of $\kappa$, let $M_1$ and $M_2$ be general surfaces in $\mathcal{M}$. Then
$$
\frac{3}{\lambda^2}=M_1\cdot M_2\cdot F\geqslant \big(F\cdot Z\big)\big(M_1\cdot M_2\big)_Z\geqslant \big(F\cdot Z\big)\mathrm{mult}^2_Z\big(\mathcal{M}\big)>\frac{F\cdot Z}{\lambda^2}=\frac{|F\cap Z|}{\lambda^2},
$$
so that $|F\cap Z|=1$ or $|F\cap Z|=2$.
One the other hand, we have  $\mathrm{Stab}_{G^\prime_{324,160}}(F)\cong\mumu_3^3$,
and the surface $F$ does not have  $\mathrm{Stab}_{G^\prime_{324,160}}(F)$-orbits of length $1$ and $2$.
Contradiction.	

Thus, we conclude that there exists a fiber $S$ of the morphism $\kappa$ such that $Z\subset S$.

Suppose that the surface $S$ is singular and $Z$ is its singular point.
Then $S$ is a cubic cone in $\mathbb{P}^3$ with vertex at $Z$.
Let $M$ be a general surface in $\mathcal{M}$, and let $\ell$ be a general ruling of the cone $S$. Then $\ell\not\subset M$, hence $$
\frac{1}{\lambda}=\frac{1}{\lambda}\big(-K_{X}\big)\cdot\ell=\frac{1}{\lambda}\big(-K_{X}+\kappa^*(D)\big)\cdot\ell=M\cdot\ell\geqslant \mathrm{mult}_Z\big(\mathcal{M}_{X}\big)>\frac{1}{\lambda},
$$
which is absurd. This shows that $Z$ is not a singular point of the surface $S$.

Using the~inversion of adjunction \cite[Theorem~5.50]{KoMo98},
we conclude that $(S,\lambda\mathcal{M}\vert_{S})$ is not log~canonical at general point of the subvariety $Z$.
But this is impossible by Lemma~\ref{lemma:cubic-surface}, because we have $\lambda\mathcal{M}\vert_{S}\equiv-K_{S}$.
This completes the proof of the lemma.
\end{proof}

In the remaining part of this section, let us present a combined proof of Main Theorem and Theorem~\ref{theorem:324-160}
that does not use Theorem~\ref{theorem:NFI}. We decided to include this proof for convenience of the reader and for one application (see Corollary~\ref{corollary:final} below).

\begin{theorem}
\label{theorem:final}
Let $f\colon \mathbb{P}^3\dasharrow Y$ be a $G$-birational map  such that
$Y$~is~a~threefold with terminal singularities, and there is a $G$-morphism $\varphi\colon Y\to Z$ that is a $G$-Mori fiber space.
If $G$ is not conjugate to $G^\prime_{324,160}$, then $Z$ is a point,
and $Y$ is $G$-isomorphic to $\mathbb{P}^3$ or $X_{24}$.
Similarly, if $G=G^\prime_{324,160}$, then one of the following possibilities holds:
\begin{itemize}
\item $Z$ is a point, and $Y$ is $G$-isomorphic to $\mathbb{P}^3$;
\item $Z$ is a point, and $Y$ is $G$-isomorphic to $X_{24}$ from Example~\ref{example:Fano-Enriques};
\item $Z=\mathbb{P}^1$, and $Y$ is $G$-isomorphic to the threefold $X$ from Lemma~\ref{lemma:X-10}.
\end{itemize}
Moreover, one has $\mathrm{Bir}^G(\mathbb{P}^3)=\Gamma$, where $\Gamma$ is the subgroup in $\mathrm{Bir}(\mathbb{P}^3)$ that is generated by the involution $\iota$ constructed in Section~\ref{section:Fano-Enriques}
and the stabilizer of the subgroup $G$ in $\mathrm{PGL}_4(\mathbb{C})$.
\end{theorem}

\begin{proof}
Recall from Section~\ref{section:Fano-Enriques}, that there exists the following $G$-commutative diagram:
$$
\xymatrix{
\widetilde{X}_{24}\ar@{->}^{\varphi}[rr]\ar@/_1pc/@{->}_{\varpi}[drr]&&X_{24}\ar@{-->}_{\chi}[d]\ar@{->}[rr]^{\varsigma}&&X_{24}\ar@{-->}^{\chi}[d]\\%
&&\mathbb{P}^3\ar@{-->}[rr]^{\iota}&&{\mathbb{P}^3}\\
&&\widetilde{V_4}\ar@{->}^{\pi}[u]\ar@{-->}^{\nu}[rr]&&\widetilde{V_4}\ar@{->}_{\pi}[u]}
$$
For the detailed description of the $G$-birational maps $\pi$, $\varphi$, $\varpi$, $\varsigma$, $\chi$ and $\nu$, see Section~\ref{section:Fano-Enriques}.

Let $\mathfrak{C}$ be one of the~two $G^\prime_{324,160}$-invariant curves \eqref{equation:48-curve-degree-9-1} and \eqref{equation:48-curve-degree-9-2}.
If~$G=G^\prime_{324,160}$, then we also have the following $G$-equivariant~diagram:
$$
\xymatrix{
&X\ar@{->}[ld]_{\vartheta}\ar@{->}[rd]^{\kappa}&\\%
\mathbb{P}^3&&\mathbb{P}^1}
$$
where $\vartheta$ is a blow up of the curve $\mathfrak{C}$, and $\kappa$ is a fibration into cubic surfaces.
Using \cite{Breuer,LMFDB}, one can show that
$$
\mathrm{Aut}\big(\mathfrak{C}\big)=G^\prime_{324,160},
$$
which implies that the normalizer in $\mathrm{PGL}_4(\mathbb{C})$ of the group $G^\prime_{324,160}$ is the group $G_{648,704}^\prime$.
Observe that $G_{648,704}^\prime$ swaps the curves \eqref{equation:48-curve-degree-9-1} and \eqref{equation:48-curve-degree-9-2}.

If $G$ is not conjugate to $G_{324,160}^\prime$,
it is enough to prove that there exists $\gamma\in\Gamma$ such that
one of the maps $f\circ\gamma$, $f\circ\gamma\circ\chi$ or $f\circ\iota$ is an isomorphism.
If $G=G^\prime_{324,160}$, it is enough to prove that there exists $\gamma\in\Gamma$ such that
$f\circ\gamma$, $f\circ\gamma\circ \chi$, $f\circ\gamma\circ\vartheta$ is an isomorphism.
To~complete~the~proof, we suppose that none of these assertions are true.

Let $H_{\mathbb{P}^3}=\mathcal{O}_{\mathbb{P}^3}(1)$, let $H_{X_{24}}$ be the hyperplane section of the Fano threefold $X_{24}\subset\mathbb{P}^{13}$,
let $E_{\pi}$, $E_{\varphi}$, $E_{\varpi}$, $E_{\vartheta}$ be the $G$-irreducible exceptional divisors of  $\pi$, $\varphi$, $\varpi$, $\vartheta$, respectively,
and let $F$ be a fiber of the cubic fibration $\kappa$.
Then
\begin{align}
\label{equation:untwisting}
\begin{split}
2\varphi^*\big(H_{\mathbb{P}^3}\big)&\sim\varpi^*\big(H_{X_{24}}\big)-E_{\varpi},\\
\varpi^*\big(H_{X_{24}}\big)&\sim 6\varphi^*\big(H_{\mathbb{P}^3}\big)-2E_{\varphi},\\
(\pi\circ\nu)^*\big(H_{\mathbb{P}^3}\big)&\sim 3\pi^*\big(H_{\mathbb{P}^3}\big)-2E_{\pi},\\
E_{\varphi}&\sim_{\mathbb{Q}}\varpi^*\big(H_{X_{24}}\big)-\frac{3}{2}E_{\varpi},\\
E_{\varpi}&\sim_{\mathbb{Q}} 4\varphi^*\big(H_{\mathbb{P}^3}\big)-2E_{\varphi},\\
E_{\pi}&\sim_{\mathbb{Q}}4 (\pi\circ\nu)^*\big(H_{\mathbb{P}^3}\big)-3\nu^*\big(E_{\pi}\big),\\
F&\sim 3\vartheta^*\big(H_{\mathbb{P}^3}\big)-E_{\vartheta}.
\end{split}
\end{align}
Note also that $H_{X_{24}}$ generates the group $\mathrm{Cl}^G(X_{24})\otimes\mathbb{Q}$.
In fact, it is not hard to see that every $G$-invariant Weil divisor on $X_{24}$ is $\mathbb{Q}$-rationally equivalent to $kH_{X_{24}}$ for $k\in\frac{1}{2}\mathbb{Z}$.

Fix a~sufficiently large integer $n\gg 0$. Let $D_Z$ be a sufficiently general very ample divisor on $Z$,
and let $\mathcal{M}_{Y}=|-nK_Y+\varphi^*(D_Z)|$. For every $\gamma\in\Gamma$, we let
\begin{align*}
\mathcal{M}^\gamma_{\mathbb{P}^3}&=(f\circ\gamma)_*^{-1}(\mathcal{M}_Y),\\
\mathcal{M}^\gamma_{X_{24}}&=(f\circ\gamma\circ\chi)_*^{-1}(\mathcal{M}_Y).
\end{align*}
Similarly, if $G=G^\prime_{324,160}$, then we let
$$
\mathcal{M}_{X}^\gamma=(f\circ\gamma\circ\vartheta)_*^{-1}(\mathcal{M}_Y)
$$
for every element $\gamma\in\Gamma$.
Now, for every element $\gamma\in\Gamma$, let $n^\gamma$ be the~positive integer such that $\mathcal{M}^\gamma_{\mathbb{P}^3}\sim n^\gamma H_{\mathbb{P}^3}$,
and let $k_\gamma$ be the~positive half-integer  such that $\mathcal{M}^\gamma_{X_{24}}\sim_{\mathbb{Q}}k_\gamma H_{X_{24}}$.
It~follows from \emph{the~Noether--Fano inequality} \cite{Iskovskikh1980,Corti1995,CheltsovShramov} that the singularities of the pair
$$
\Big(\mathbb{P}^3,\frac{4}{n^\gamma}\mathcal{M}^\gamma_{\mathbb{P}^3}\Big)
$$
are not canonical for every $\gamma\in\Gamma$, because $f\circ\gamma$ is not an isomorphism by our assumption.
Similarly, we see that the singularities of the log pair
$$
\Big(X_{24},\frac{1}{k_\gamma}\mathcal{M}^\gamma_{X_{24}}\Big)
$$
are also not canonical for every $\gamma\in\Gamma$, since we assumed that $f\circ\gamma\circ\chi$ is not an isomorphism.

Moreover, if $G=G^\prime_{324,160}$, then
$$
K_{X}+\frac{1}{n^\gamma-3\mathrm{mult}_{\mathfrak{C}}(\mathcal{M}^\gamma_{\mathbb{P}^3})}\mathcal{M}_{X}^\gamma
\sim_{\mathbb{Q}}\frac{n^\gamma-4\mathrm{mult}_{\mathfrak{C}}(\mathcal{M}^\gamma_{\mathbb{P}^3})}{n^\gamma-3\mathrm{mult}_{\mathfrak{C}}(\mathcal{M}^\gamma_{\mathbb{P}^3})}F,
$$
where $\mathrm{mult}_{\mathfrak{C}}(\mathcal{M}^\gamma_{\mathbb{P}^3})<\frac{n^\gamma}{3}$.
Thus, if $G=G^\prime_{324,160}$ and $n^\gamma-4\mathrm{mult}_{\mathfrak{C}}(\mathcal{M}^\gamma_{\mathbb{P}^3})\geqslant 0$,
then it follows from the~Noether--Fano inequality that the singularities of the log pair
$$
\Big(X,\frac{1}{n^\gamma-3\mathrm{mult}_{\mathfrak{C}}(\mathcal{M}^\gamma_{\mathbb{P}^3})}\mathcal{M}_{X}^\gamma\Big)
$$
are not canonical for every $\gamma\in\Gamma$, because we assumed that $f\circ\gamma\circ\vartheta$ is not an isomorphism.
However, we already proved in Lemma~\ref{lemma:X-10}
that this log pair have canonical singularities.
Hence, if $G=G^\prime_{324,160}$, then $\mathrm{mult}_{\mathfrak{C}}(\mathcal{M}^\gamma_{\mathbb{P}^3})<\frac{n^\gamma}{4}$ for every $\gamma\in\Gamma$.

Now, for every $\gamma\in\Gamma$, we let \mbox{$m_{\mathcal{L}_6}^\gamma=\mathrm{mult}_{\mathcal{L}_6}(\mathcal{M}_{\mathbb{P}^3}^\gamma)$}
and \mbox{$m_{\Sigma_4}^\gamma=\mathrm{mult}_{\Sigma_4}(\mathcal{M}_{\mathbb{P}^3}^\gamma)$}.
Similarly, let $m_{E_{\phi}}^\gamma$ be the non-negative rational number such that
$$
\phi_*^{-1}\big(\mathcal{M}_{X_{24}}^\gamma\big)\sim_{\mathbb{Q}}\phi^*\big(\mathcal{M}_{X_{24}}^\gamma\big)-m^\prime_{E_{\phi}}E_{\phi}.
$$
Then, using \eqref{equation:untwisting}, we see that
\begin{align}
\label{equation:final-unwisting-numbers}
\begin{split}
	&n^\gamma=6k^\gamma-4m_{E_{\phi}}^\gamma; \\
	&n^{\gamma\circ\iota}=3n^\gamma-4m^\gamma_{\Sigma_4};\\
	&k^\gamma=\frac{n^\gamma}{2}-m_{\mathcal{L}_6}^\gamma.
\end{split}
\end{align}
Then, we let
$$
\delta=\min_{\gamma\in\Gamma}\Big\{\frac{n^\gamma}{4},k^\gamma\Big\}.
$$
Note that this minimum is attained for some $\gamma\in\Gamma$.

Suppose that $\delta=\frac{n^\gamma}{4}$ for some $\gamma\in\Gamma$.
If $G$ is not conjugate to $G_{48,50}$, $G_{96,227}$, $G_{324,160}^\prime$,
then it follows from Proposition~\ref{proposition:technical-1}
that
$$
m_{\mathcal{L}_6}^\gamma>\frac{n^\gamma}{4}
$$
or
$$
m^\gamma_{\Sigma_4}>\frac{n^\gamma}{4},
$$
therefore \eqref{equation:final-unwisting-numbers} gives
$$
\delta\leqslant k^\gamma=\frac{n^\gamma}{2}-m_{\mathcal{L}_6}^\gamma<\frac{n^\gamma}{4}=\delta
$$
or
$$
\delta\leqslant n^{\gamma\circ\iota}=3n^\gamma-4m^\gamma_{\Sigma_4}<\frac{n^\gamma}{4}=\delta,
$$
which is a contradiction.
Similarly, if $G=G_{324,160}^\prime$, it follows from Proposition~\ref{proposition:technical-1}
that at least one of the following strict inequalities holds:
\begin{align*}
m_{\mathcal{L}_6}^\gamma&>\frac{n^\gamma}{4}, \\
m^\gamma_{\Sigma_4}&>\frac{n^\gamma}{4},\\
\mathrm{mult}_{\mathfrak{C}}\big(\mathcal{M}^\gamma_{\mathbb{P}^3}\big)&>\frac{n^\gamma}{4},\\
\mathrm{mult}_{\mathfrak{C}}\big(\mathcal{M}^{\gamma\circ K}_{\mathbb{P}^3}\big)&>\frac{n^\gamma}{4}
\end{align*}
for
$$
K=\begin{pmatrix}
0 & 1 & 0 & 0\\
1 & 0 & 0 & 0\\
0 & 0 & 1 & 0\\
0 & 0 & 0 & -1
\end{pmatrix}\in G_{648,704}^\prime\subset\Gamma,
$$
since $\mathfrak{C}$ and $K(\mathfrak{C})$ are the $G^\prime_{324,160}$-invariant curves \eqref{equation:48-curve-degree-9-1} and \eqref{equation:48-curve-degree-9-2}.
But we already proved earlier that $\mathrm{mult}_{\mathfrak{C}}(\mathcal{M}^\gamma_{\mathbb{P}^3})<\frac{n^\gamma}{4}$
and $\mathrm{mult}_{\mathfrak{C}}(\mathcal{M}^{\gamma\circ K}_{\mathbb{P}^3})<\frac{n^\gamma}{4}$,
which implies that
$$
m_{\mathcal{L}_6}^\gamma>\frac{n^\gamma}{4}
$$
or
$$
m^\gamma_{\Sigma_4}>\frac{n^\gamma}{4},
$$
and \eqref{equation:final-unwisting-numbers} gives
$$
\delta\leqslant k^\gamma<\frac{n^\gamma}{4}=\delta
$$
or
$$
\delta\leqslant n^{\gamma\circ\iota}<\frac{n^\gamma}{4}=\delta.
$$
Thus, we see that either $G=G_{48,50}$ or $G=G_{96,227}$.
Then it follows from Proposition~\ref{proposition:technical-1} that at least one of the following inequalities holds:
\begin{align*}
m_{\mathcal{L}_6}^\gamma&>\frac{n^\gamma}{4}, \\
m_{\mathcal{L}_6}^{\gamma\circ R}&>\frac{n^\gamma}{4},\\
m_{\mathcal{L}_6}^{\gamma\circ R^2}&>\frac{n^\gamma}{4},\\
m^\gamma_{\Sigma_4}&>\frac{n^\gamma}{4},\\
m^{\gamma\circ R}_{\Sigma_4}&>\frac{n^\gamma}{4},\\
m^{\gamma\circ R}_{\Sigma_4}&>\frac{n^\gamma}{4},
\end{align*}
because $R(\mathcal{L}_6)=\mathcal{L}_6^\prime$, $R^2(\mathcal{L}_6)=\mathcal{L}_6^{\prime\prime}$,
$R(\Sigma_4)=\Sigma_4^\prime$, $R^2(\Sigma_4)=\Sigma_4^{\prime\prime}$, and $R\in G_{576,8654}\subset\Gamma$.
Here, $R$ is one of the~generators of the group $G_{576,8654}$ defined in \eqref{equation:48-R}, see Remark~\ref{remark:three-involutions}.
As above, we obtain a contradiction with~$\delta=\frac{n^\gamma}{4}$, because $n^\gamma=n^{\gamma\circ R}=n^{\gamma\circ R^2}$.

Hence, we conclude $\delta=k^\gamma$ for some $\gamma\in\Gamma$.
Then it follows from Proposition~\ref{proposition:technical-2}
that the log pair $(X_{24},\frac{1}{k^\gamma}\mathcal{M}_{X_{24}^\gamma})$ is not canonical at some singular point of the threefold~$X_{24}$.

Recall from Section~\ref{section:Fano-Enriques} that $\mathrm{Sing}(X_{24})$ is a union of two $G$-orbits: $\varphi(E_{\varphi})$ and $\varsigma\circ\varphi(E_{\varphi})$,
where $\varsigma=\chi^{-1}\circ\iota\circ\chi$.
If $(X_{24},\frac{1}{k^\gamma}\mathcal{M}_{X_{24}})$ is not canonical~at~$\varphi(E_{\varphi})$,
then
$$
m_{E_{\phi}}^\gamma>\frac{k^\gamma}{2}
$$
by Kawamata's theorem \cite{Kawamata}.
Likewise, if  $(X_{24},\frac{1}{k^\gamma}\mathcal{M}_{X_{24}})$ is not canonical at $\varsigma\circ\varphi(E_{\varphi})$,~then
$$
m_{E_{\phi}}^{\gamma\circ\iota}>\frac{k^\gamma}{2}.
$$
Therefore, using \eqref{equation:final-unwisting-numbers}, we see that
$$
\delta\leqslant\frac{n^\gamma}{4}=\frac{6k^\gamma-4m_{E_{\phi}}^\gamma}{4}<k^\gamma=\delta
$$
or
$$
\delta\leqslant\frac{n^{\gamma\circ\iota}}{4}=\frac{6k^{\gamma}-4m_{E_{\phi}}^{\gamma\circ\iota}}{4}<k^{\gamma}=\delta.
$$
The obtained contradiction completes the proof of Theorem~\ref{theorem:final}.
\end{proof}

\begin{corollary}
\label{corollary:final-large-groups}
If $G$ is not conjugate to $G_{48,50}$ or $G_{96,227}$, then $\mathrm{Bir}^{G}(\mathbb{P}^3)$ is finite.
\end{corollary}

\begin{proof}
Let $\mathfrak{N}$ be the normalizer of the group $G$ in $\mathrm{PGL}_4(\mathbb{C})$.
Then
$$
\mathrm{Bir}^{G}\big(\mathbb{P}^3\big)=\big\langle \iota,\mathfrak{N}\big\rangle
$$
by Theorem~\ref{theorem:final}.
The centralizer of the group $G$ in $\mathrm{PGL}_4(\mathbb{C})$ is trivial by Schur's lemma,
so we have an embedding $\mathfrak{N}\hookrightarrow\mathrm{Aut}(G)$, which implies that $\mathfrak{N}$ is finite.

Observe that $\Sigma_{4}$ is a $\mathfrak{N}$-orbit, because  $\Sigma_4$ is a unique $G$-orbit of length $4$.

Let $\iota$ be the involution in $\mathrm{Bir}^{G}(\mathbb{P}^3)$ described in Section~\ref{section:Fano-Enriques}.
Note that~$\iota$ is $\mathfrak{N}$-equivariant, since $\Sigma_{4}$ is a $\mathfrak{N}$-orbit.
This implies that $\iota$ also normalizes $\mathfrak{N}$,
so $\langle \iota,\mathfrak{N}\rangle$ is finite.
\end{proof}

If $G=G_{48,50}$ or $G=G_{96,227}$, it follows from Theorem~\ref{theorem:final} that $\mathrm{Bir}^{G}(\mathbb{P}^3)=\langle \iota,G_{576,8654}\rangle$.
In these two cases, the group $\mathrm{Bir}^{G}(\mathbb{P}^3)$ is infinite (but discrete) by the following result.

\begin{corollary}
\label{corollary:final}
Suppose that $G=G_{48,50}$ or $G=G_{96,227}$.
Then $\langle \iota,\iota^\prime,\iota^{\prime\prime}\rangle\cong\mu_2\ast\mu_2\ast\mu_2$,
and there exists the following split exact sequence of groups:
$$
1\longrightarrow\big\langle \iota,\iota^\prime,\iota^{\prime\prime}\big\rangle\longrightarrow\mathrm{Bir}^{G}\big(\mathbb{P}^3\big)\longrightarrow G_{576,8654}\longrightarrow 1,
$$
where $\iota$, $\iota^\prime$, $\iota^{\prime\prime}$ are birational involutions in $\mathrm{Bir}^{G}(\mathbb{P}^3)$  described in Remark~\ref{remark:three-involutions}.
\end{corollary}

\begin{proof}
It follows from Theorem~\ref{theorem:final} that $\mathrm{Bir}^{G}(\mathbb{P}^3)$
is generated by $\iota$, $\iota^\prime$, $\iota^{\prime\prime}$ and $G_{576,8654}$.
Using this, it is not very difficult to check that $\langle \iota,\iota^\prime,\iota^{\prime\prime}\rangle$ is a normal subgroup in~$\mathrm{Bir}^{G}(\mathbb{P}^3)$.
Recall from Lemma~\ref{lemma:48-50-normalizer} that $G_{576,8654}$ is the normalizer of the subgroup $G$ in $\mathrm{PGL}_4(\mathbb{C})$.

Fix a $G$-birational map $g\in\langle \iota,\iota^\prime,\iota^{\prime\prime}\rangle$.
Let us show that  $g$ can be uniquely written as a~composition of $\iota$, $\iota^\prime$,~$\iota^{\prime\prime}$.
The proof of this fact is similar to the proof of \cite[Theorem~3.10]{Iskovskikh1980}.
More precisely, the proof of Theorem~\ref{theorem:final} provides an algorithm how to decompose $g$ as a composition of $\iota$, $\iota^\prime$,~$\iota^{\prime\prime}$.
Let us remind this algorithm. To start with, we let
$$
\mathcal{M}_{\mathbb{P}^3}=g_*^{-1}\big(|\mathcal{O}_{\mathbb{P}^3}(1)|\big),
$$
and let $n\in\mathbb{Z}_{>0}$ such that $\mathcal{M}_{\mathbb{P}^3}\subset |\mathcal{O}_{\mathbb{P}^3}(n)|$.
The number $n$ is known as the \emph{degree} of~$g$.
Then, arguing as in the proof of Theorem~\ref{theorem:final}, we see that either $n=1$ and $g\in G_{576,8654}$,
or $n>1$ and the singularities of the log pair
$$
\Big(\mathbb{P}^3,\frac{4}{n}\mathcal{M}_{\mathbb{P}^3}\Big),
$$
are not canonical. Now, using Proposition~\ref{proposition:technical-1},
we see that \emph{at least} one inequality~holds among the following three inequalities:
\begin{align}
\label{equation:Bir-inequality-1}\mathrm{max}\Big(4\mathrm{mult}_{\mathcal{L}_6}\big(\mathcal{M}_{\mathbb{P}^3}\big),2\mathrm{mult}_{\Sigma_4}\big(\mathcal{M}_{\mathbb{P}^3}\big)\Big)&>n,\\
\label{equation:Bir-inequality-2}\mathrm{max}\Big(4\mathrm{mult}_{\mathcal{L}_6^{\prime}}\big(\mathcal{M}_{\mathbb{P}^3}\big),2\mathrm{mult}_{\Sigma_4^\prime}\big(\mathcal{M}_{\mathbb{P}^3}\big)\Big)&>n,\\
\label{equation:Bir-inequality-3}\mathrm{max}\Big(4\mathrm{mult}_{\mathcal{L}_6^{\prime\prime}}\big(\mathcal{M}_{\mathbb{P}^3}\big),2\mathrm{mult}_{\Sigma_4^{\prime\prime}}\big(\mathcal{M}_{\mathbb{P}^3}\big)\Big)&>n.
\end{align}
Moreover, if the~inequality \eqref{equation:Bir-inequality-1} holds, then it follows from the proof of Theorem~\ref{theorem:final} that the degree of the composition $g\circ\iota$ is strictly smaller than $n$.
Similarly, if \eqref{equation:Bir-inequality-2} holds, then the degree of the composition $g\circ\iota^\prime$ is strictly smaller than $n$.
Finally, if \eqref{equation:Bir-inequality-3} holds, then the degree of the composition $g\circ\iota^{\prime\prime}$ is smaller than $n$.
Thus, iterating this process, we decompose $g$ into a composition of involutions $\iota$, $\iota^\prime$, $\iota^{\prime\prime}$ and an element in $G_{576,8654}$.

To prove that $\langle \iota,\iota^\prime,\iota^{\prime\prime}\rangle\cong\mu_2\ast\mu_2\ast\mu_2$,
we must prove that \emph{precisely} one birational map among $g\circ\iota$, $g\circ\iota^\prime$, $g\circ\iota^\prime$
has degree (strictly) smaller than the degree of the birational~map~$g$, so the described algorithm decomposes $g$ in a unique way.
To~prove this, it is enough to show that \emph{precisely} one inequality among \eqref{equation:Bir-inequality-1},
\eqref{equation:Bir-inequality-2}, \eqref{equation:Bir-inequality-3} holds.

Without loss of generality, it is enough to show that both inequalities \eqref{equation:Bir-inequality-2} and \eqref{equation:Bir-inequality-3} cannot hold simultaneously.
By Proposition~\ref{proposition:48-curves}, we have
$$
\mathrm{mult}_{\mathcal{L}_6^\prime}\big(\mathcal{M}_{\mathbb{P}^3}\big)+\mathrm{mult}_{\mathcal{L}_6^{\prime\prime}}\big(\mathcal{M}_{\mathbb{P}^3}\big)\leqslant\frac{n}{2}.
$$
Similarly, it follows from the proof of Proposition~\ref{proposition:technical-1} that
$$
\mathrm{mult}_{\Sigma_4^\prime}\big(\mathcal{M}_{\mathbb{P}^3}\big)+\mathrm{mult}_{\Sigma_4^{\prime\prime}}\big(\mathcal{M}_{\mathbb{P}^3}\big)\leqslant n.
$$
Moreover, if $\mathrm{mult}_{\mathcal{L}_6^{\prime}}(\mathcal{M}_{\mathbb{P}^3})>\frac{n}{4}$,
then it follows from the proof of Proposition~\ref{proposition:technical-1} that the~degree of the composition
$g\circ\iota^\prime$ is strictly less than $n$, so \eqref{equation:final-unwisting-numbers} gives
$$
\mathrm{mult}_{\Sigma_4^{\prime}}\big(\mathcal{M}_{\mathbb{P}^3}\big)>\frac{n}{2}.
$$
Likewise, if $\mathrm{mult}_{\mathcal{L}_6^{\prime\prime}}(\mathcal{M}_{\mathbb{P}^3})>\frac{n}{4}$,
then
$$
\mathrm{mult}_{\Sigma_4^{\prime\prime}}\big(\mathcal{M}_{\mathbb{P}^3}\big)>\frac{n}{2}.
$$
Therefore, if \eqref{equation:Bir-inequality-2}~holds, then the~inequality \eqref{equation:Bir-inequality-3} does not hold.
\end{proof}

\end{document}